\documentclass{article}

\usepackage[url=false,
isbn=false,
doi=true,
maxbibnames=99,
style=alphabetic,
backend=bibtex,
natbib=true,
sorting=nty]{biblatex}
\renewbibmacro{in:}{}
\bibliography{RoughPathTranslationsRefs}

\usepackage{amsfonts}
\usepackage{amsmath, amsthm}
\usepackage{color}
\usepackage{amsfonts}
\usepackage{amssymb}
\usepackage{hyperref}
\usepackage{authblk}

\usepackage{shuffle}

\usepackage{enumitem}

\usepackage{mathrsfs}
\usepackage{tikz}
\usepackage{mhequ} 
\usetikzlibrary{snakes}
\usetikzlibrary{decorations}
\usetikzlibrary{positioning}
\usetikzlibrary{shapes}
\usetikzlibrary{external}

\makeatletter
\pgfdeclareshape{crosscircle}
{
  \inheritsavedanchors[from=circle] 
  \inheritanchorborder[from=circle]
  \inheritanchor[from=circle]{north}
  \inheritanchor[from=circle]{north west}
  \inheritanchor[from=circle]{north east}
  \inheritanchor[from=circle]{center}
  \inheritanchor[from=circle]{west}
  \inheritanchor[from=circle]{east}
  \inheritanchor[from=circle]{mid}
  \inheritanchor[from=circle]{mid west}
  \inheritanchor[from=circle]{mid east}
  \inheritanchor[from=circle]{base}
  \inheritanchor[from=circle]{base west}
  \inheritanchor[from=circle]{base east}
  \inheritanchor[from=circle]{south}
  \inheritanchor[from=circle]{south west}
  \inheritanchor[from=circle]{south east}
  \inheritbackgroundpath[from=circle]
  \foregroundpath{
    \centerpoint%
    \pgf@xc=\pgf@x%
    \pgf@yc=\pgf@y%
    \pgfutil@tempdima=\radius%
    \pgfmathsetlength{\pgf@xb}{\pgfkeysvalueof{/pgf/outer xsep}}%
    \pgfmathsetlength{\pgf@yb}{\pgfkeysvalueof{/pgf/outer ysep}}%
    \ifdim\pgf@xb<\pgf@yb%
      \advance\pgfutil@tempdima by-\pgf@yb%
    \else%
      \advance\pgfutil@tempdima by-\pgf@xb%
    \fi%
    \pgfpathmoveto{\pgfpointadd{\pgfqpoint{\pgf@xc}{\pgf@yc}}{\pgfqpoint{-0.707107\pgfutil@tempdima}{-0.707107\pgfutil@tempdima}}}
    \pgfpathlineto{\pgfpointadd{\pgfqpoint{\pgf@xc}{\pgf@yc}}{\pgfqpoint{0.707107\pgfutil@tempdima}{0.707107\pgfutil@tempdima}}}
    \pgfpathmoveto{\pgfpointadd{\pgfqpoint{\pgf@xc}{\pgf@yc}}{\pgfqpoint{-0.707107\pgfutil@tempdima}{0.707107\pgfutil@tempdima}}}
    \pgfpathlineto{\pgfpointadd{\pgfqpoint{\pgf@xc}{\pgf@yc}}{\pgfqpoint{0.707107\pgfutil@tempdima}{-0.707107\pgfutil@tempdima}}}
  }
}
\makeatother

\colorlet{symbols}{blue!90!black}
\colorlet{testcolor}{green!60!black}
\colorlet{connection}{red!30!black}
\def\symbol#1{\textcolor{symbols}{#1}}
\def\symbol#1{\textcolor{symbols}{#1}}

\tikzset{
root/.style={circle,fill=black!50,inner sep=0pt, minimum size=3mm},
        dot/.style={circle,fill=black,inner sep=0pt, minimum size=1.5mm},
        dotred/.style={circle,fill=black!50,inner sep=0pt, minimum size=2mm},
        var/.style={circle,fill=black!10,draw=black,inner sep=0pt, minimum size=3mm},
        kernel/.style={semithick,shorten >=2pt,shorten <=2pt},
        kernel1/.style={thick},
        kernels/.style={snake=zigzag,shorten >=2pt,shorten <=2pt,segment amplitude=1pt,segment length=4pt,line before snake=2pt,line after snake=5pt,},
		kernels1/.style={snake=zigzag,segment amplitude=0.5pt,segment length=2pt},
		rho1/.style={dotted,semithick},
        rho/.style={densely dashed,semithick,shorten >=2pt,shorten <=2pt},
           testfcn/.style={dotted,semithick,shorten >=2pt,shorten <=2pt},
        renorm/.style={shape=circle,fill=white,inner sep=1pt},
        labl/.style={shape=rectangle,fill=white,inner sep=1pt},
        xic/.style={very thin,circle,fill=symbols,draw=black,inner sep=0pt,minimum size=1.2mm},
        xi/.style={very thin,circle,fill=blue!10,draw=black,inner sep=0pt,minimum size=1.2mm},
        xix/.style={crosscircle,fill=blue!10,draw=black,inner sep=0pt,minimum size=1.2mm},
	xib/.style={very thin,circle,fill=blue!10,draw=black,inner sep=0pt,minimum size=1.6mm},
	xie/.style={very thin,circle,fill=green!50!black,draw=black,inner sep=0pt,minimum size=1mm},
	xid/.style={very thin,circle,fill=symbols,draw=black,inner sep=0pt,minimum size=1.6mm},
	xibx/.style={crosscircle,fill=blue!10,draw=black,inner sep=0pt,minimum size=1.6mm},
	edgetype/.style={very thin,circle,draw=black,inner sep=0pt,minimum size=5mm},
	nodetype/.style={very thick,circle,draw=black,inner sep=0pt,minimum size=5mm},
	kernels2/.style={very thick,draw=connection,segment length=12pt},
clean/.style={thin,circle,fill=black,inner sep=0pt,minimum size=1mm},	not/.style={thin,circle,fill=symbols,draw=connection,fill=connection,inner sep=0pt,minimum size=0.5mm},
	>=stealth,
        }

\tikzset{ individus/.style={scale=0.40,draw,circle,thick,fill=black!10},
 individu/.style={scale=0.40,draw,circle,thick,fill=black!50},       } 
\makeatletter
\def\DeclareSymbol#1#2#3{\expandafter\gdef\csname MH@symb@#1\endcsname{\tikz[baseline=#2,scale=0.15,draw=symbols]{#3}}\expandafter\gdef\csname MH@symb@#1s\endcsname{\scalebox{0.7}{\tikz[baseline=#2,scale=0.15,draw=symbols]{#3}}}}
\def\<#1>{\csname MH@symb@#1\endcsname}
\makeatother

 \def\1{\mathbf{\symbol{1}}}

\def\one{\mathbf{1}}
\def\eps{\varepsilon}

\newtheorem{theorem}{Theorem}

\newtheorem{corollary}[theorem]{Corollary}

\newtheorem{definition}[theorem]{Definition}
\newtheorem{example}[theorem]{Example}

\newtheorem{lemma}[theorem]{Lemma}

\newtheorem{proposition}[theorem]{Proposition}

\theoremstyle{remark}
\newtheorem{remark}[theorem]{Remark}

\newcommand{\Vect}{\textnormal{Vect}}

\newcommand{\op}{\textnormal{op}}

\newcommand{\Hai}{\mathrm{Hai14}}
\newcommand{\BHZr}{\mathrm{BHZr}}
\newcommand{\BHZ}{\mathrm{BHZ}}

\newcommand{\GPrimal}{T(\R^{1+d})}
\newcommand{\GDual}{T((\R^{1+d}))}
\newcommand{\BPrimal}{\HH}
\newcommand{\BDual}{\HH^*}
\newcommand{\extract}{\delta}

\newcommand{\Deltam}{\Delta^{\!-}}
\newcommand{\Deltap}{\Delta^{\!+}}
\newcommand{\Deltashuff}{\Delta_{\shuffle}}
\newcommand{\tensor}{{\ast}}
\newcommand{\poly}{\odot}

\newcommand{\II}{\mathcal{I}}
\newcommand{\JJ}{\mathcal{J}}
\newcommand{\TT}{\mathcal{T}}

\newcommand{\HH}{\mathcal{H}}

\newcommand{\GG}{\mathcal{G}}
\newcommand{\BB}{\mathcal{B}}
\newcommand{\LL}{\mathcal{L}}

\newcommand{\MM}{\mathcal{M}}

\newcommand{\SSS}{\mathcal{S}}
\newcommand{\A}{\mathcal{A}}
\newcommand{\WW}{\mathcal{W}}

\newcommand{\R}{\mathbb R}

\newcommand{\N}{\mathbb N}
\newcommand{\Z}{\mathbb Z}

\newcommand{\W}{\mathbb W}
\newcommand{\Pbb}{\mathbb P}

\newcommand{\Abb}{\mathbb A}

\newcommand{\Xbf}{\mathbf X}

\newcommand{\x}{\mathbf x}

\newcommand{\B}{\mathbf B}

\newcommand{\G}{G}

\newcommand{\Hol}[1]{#1\textnormal{-H{\"o}l}}

\newcommand{\Strat}{\textnormal{Strat}}
\newcommand{\Ito}{\textnormal{It{\^o}}}

\newcommand{\EEE}[1]{\mathbb E \left[ #1 \right]}

\newcommand{\id}{\textnormal{id}}
\newcommand{\norm}[1]{\left|\left|#1 \right|\right|}

\newcommand{\scal}[1]{\langle #1 \rangle}
\newcommand{\floor}[1]{\lfloor #1 \rfloor}

\newcommand\restr[2]{{\left.\kern-\nulldelimiterspace 
  #1 
  \vphantom{\big|} 
  \right|_{#2} 
  }}

\usepackage[normalem]{ulem}

\usepackage{geometry}
\input{tcilatex}

\begin{document}

\title{A rough path perspective on renormalization}

\author[1]{Y. Bruned}
\author[2]{I. Chevyrev}
\author[3,4]{P.K. Friz}
\author[3]{R. Prei\ss}
\affil[1]{{School of Mathematics, University of Edinburgh}}
\affil[2]{Mathematical Institute,
University of Oxford}
\affil[3]{{Institut f\"ur Mathematik, Technische Universit\"at Berlin}}
\affil[4]{{Weierstra\ss --Institut f\"ur Angewandte Analysis und Stochastik, Berlin}}

\date{\today}
%
%

\maketitle

\begin{abstract}
We develop the algebraic theory of rough path translation. Particular attention is given to the case of branched rough paths, whose underlying algebraic structure (Connes-Kreimer, Grossman-Larson) makes it a useful model case of a regularity structure in the sense of Hairer. Pre-Lie structures are seen to play a fundamental rule which allow a direct understanding of the translated (i.e. renormalized) equation under consideration.  This construction is also novel with regard to the algebraic renormalization theory for regularity structures due to Bruned--Hairer--Zambotti (2016), the links with which are discussed in detail.
\end{abstract}

\tableofcontents

\section{Introduction}

\subsection{Rough paths and regularity structures}

The theory of rough paths~\cite{Lyons98} deals with controlled
differential equations of the form
\[
dY_{t}= { f_0 (Y_t) dt +}  \sum_{i=1}^{d}f_{i}\left( Y_{t}\right) dX_{t}^{i} 
\ \ Y_0 = y_0 \in \R^e \, .
\]
with $(X^1,...,X^d) :\left[ 0,T\right] \rightarrow \mathbb{R}^{d}$, of
of low, say $\alpha $-H\"{o}lder, regularity for $0<\alpha \leq 1$. 
As may
be seen by formal Picard iteration, { given a collection $f_0,f_1,...,f_d$ of nice vector fields on $\R^e$,} the solution can be expanded in terms of
certain integrals. Assuming validity of the chain-rule, and writing $X^0(t) \equiv t$ for notational convenience,
these are just
iterated integrals of the form $\int dX^{i_{1}}\cdots dX^{i_{n}}$ with
integration over $n$-dimensional simplex. In geometric rough path theory
one postulates the existence of such integrals, for sufficiently many words $w=\left( i_{1},\dots ,i_{n}\right)$, namely $\left\vert w\right\vert =n\leq\left[ 1/\alpha \right] $, such as to regain analytic control: the
collection of resulting objects
\[
\left\langle \mathbf{X},w\right\rangle =\int \dots \int dX^{i_{1}}\dots
dX^{i_{n}}\text{ (integration over }s<t_{1}<\dots <t_{n}<t\text{, for all }
0\leq s<t\leq T\text{)}
\]
subject to suitable analytic and algebraic constraints (in particular, \textit{Chen's relation}, which describes the recentering $s\rightarrow \tilde{s}$ )
is then known as a (level-$n$) \textit{geometric rough path}, introduced by \cite{Lyons98}.
For the readers convenience we give some precise recalls, along the lines of Hairer--Kelly \cite{HairerKelly15}, in Section \ref{sec:GNG} below.
Without assuming a chain-rule (think:\
It\^o), iterated integrals of the form $\int X^{i}X^{j}dX^{k}$ appear in the expansion,
the resulting objects are then naturally indexed by trees, for example
\[
\left\langle \mathbf{X},\tau \right\rangle =\int X^{i}X^{j}dX^{k}\text{ with 
}\tau =\left[ \bullet_i \bullet_j \right] _{\bullet_k} \equiv 
\tikzexternaldisable  \begin{tikzpicture}[scale=0.2,baseline=0.1cm]
        \node at (0,0)  [dot,label= {[label distance=-0.2em]below: \scriptsize  $ k $} ] (root) {};
         \node at (1,2)  [dot,label={[label distance=-0.2em]above: \scriptsize  $ j $}] (right) {};
         \node at (-1,2)  [dot,label={[label distance=-0.2em]above: \scriptsize  $ i $} ] (left) {};
            \draw[kernel1] (right) to
     node [sloped,below] {\small }     (root); \draw[kernel1] (left) to
     node [sloped,below] {\small }     (root);
     \end{tikzpicture} \tikzexternaldisable .
\]
The collection of all such objects, again for sufficiently many trees, $\left\vert \tau \right\vert =\#$nodes $\leq \left[ 1/\alpha \right] $ and
subject to algebraic and analytic constraints, form what is known as
a {\it branched rough path} \cite{Gubinelli10, HairerKelly15}. 
Here again, we refer to Section \ref{sec:GNG} for a precise definition and further recalls.

\medskip

A basic result - known as the \textit{extension theorem} \cite{Lyons98, Gubinelli10} - asserts that
all ``higher'' iterated integrals, $n$-fold with $n>\left[ 1/\alpha \right] $,
are automatically well-defined, with validity of all algebraic and analytic
constraints in the extended setting.\footnote{This entire ensemble of iterated integrals is called the {\it signature} or the {\it fully lifted rough path}.} Solving differential equations driven
by such rough paths can then be achieved, following
\cite{Gubinelli04}, see also \cite{FrizHairer14}, by formulating a fixed point problem in a space
of controlled rough paths, essentially a (linear) space of good integrands
for rough integration (mind that rough path spaces are, in contrast,
fundamentally non-linear due to the afore-mentioned algebraic constraints).
Given a rough differential equation (RDE) of the form
\[
dY= f_0(Y)dt + f\left( Y\right) d\mathbf{X}
\]
it is interesting to see the effect on $Y$ induced by higher-order
perturbations (``translations'') of the driving rough path $\mathbf{X}$. For instance, one can use It\^{o} integration
to lift a $d$-dimensional Brownian motion $(B^1,...,B^d)$ to a (level-2) random rough path, $\mathbf{X}=\mathbf{B}^{\Ito}\left( \omega \right) $ of regularity $\alpha \in \left(
1/3,1/2\right) $, in which case the above RDE corresponds to the classical It{\^o}
SDE { 
$$
dY_{t}= f_0 (Y_t) dt +  \sum_{i=1}^{d}f_{i}\left( Y_{t}\right) dB_{t}^{i} \, ,
\ \ Y_0 = y_0 \in \R^e \, .
$$ }
However, we may perturb $\mathbf{B}^{\Ito}=\left( B,\mathbb{B}
^{\Ito}\right) $ via $\mathbb{B}_{s,t}^{\Ito}\mapsto \mathbb{B}_{s,t}^{\Ito}+
\frac{1}{2}I\left( t-s\right) =:\mathbb{B}_{s,t}^{\Strat}$, without touching
the underlying Brownian path $B$. The above RDE then becomes a Stratonovich
SDE. { On the level of the original (It\^o)-equation, the effect of this perturbation is 
a modified drift vector field, 
$$
             f_0   \rightsquigarrow f_0 + \frac{1}{2} \sum_{i=1}^{d}\nabla _{f_{i}}f_{i} \ , 
$$
famously known as It{\^o}-Stratonovich correction.}
Examples from physics (e.g. Brownian motion
in a magnetic field) suggest second order perturbation of the form $\mathbb{B}_{s,t}^{\Strat}\mapsto \mathbb{B}_{s,t}^{\Strat}+a\left( t-s\right) $, for
some $a\in \mathfrak{so}\left( d\right) $, the SDE is then affected by a drift correction of the form
$$
             f_0   \rightsquigarrow f_0 + \sum_{i,j}a^{ij}\left[ f_{i},f_{j}\right] \ .
$$
In the context of classical SDEs, area corrections are also discussed in~\cite{ikeda1989stochastic}, and carefully designed twisted Wong--Zakai type approximations led Sussmann~\cite{Sussmann91} to drift corrections involving higher iterated Lie brackets.
This was reconciled with geometric rough path theory in~\cite{FrizOberhauser09}, and provides a nice example where (Brownian) rough paths (with $\gamma=\frac12-$ regularity) need to be studied in the {\it entire} scale of different rough path topologies indexed by $\gamma \in (0,1/2)$. 

As we shall see, all these examples are but the tip of an iceberg.
It will also be seen that there is a substantial difference between the geometric rough path case and the generality aimed for in this paper.

We finally note that tampering with ``higher-levels'' of the lifted noise also affects the structure of stochastic partial differential equations: this is not only omnipresent in the case of singular SPDEs~\cite{Hairer14, GIP15}, but very much in every SPDE with rough path noise as remarked e.g. in~\cite{CFO11}.

\bigskip

\textbf{From rough paths to regularity structures.} The theory of
regularity structures~\cite{Hairer14} extends rough path theory and then provides a
remarkable framework to analyse (singular) semi-linear stochastic partial differential
equations, e.g. of the form
\[
\left( \partial _{t}-\Delta \right) u=f\left( u,Du\right) +g\left( u\right)
\xi \left( t,x,\omega \right) 
\]
with $\left( d+1\right) $-dimensional space-time white noise. The demarche
is similar as above: noise is lifted to a {\it model}, whose algebraic properties
(especially with regard to recentering) are formulated with the aid of the
{\it structure group}. Given an (abstract) model, a fixed point problem is solved and gives a solution
in a
(linear) space of {\it modelled distributions}. The abstract solution can then be mapped (``{\it reconstructed}\hspace{0.4mm}'') into an actual distribution (a.k.a generalized function). In fact,
one has the rather precise correspondences as follows (see \cite{FrizHairer14} for explicit
details in the level-2 setting): 

\begin{table}[!htbp]
\centering
\begin{tabular}{|lll|}
\hline
    rough path                   &      $\longleftrightarrow$                 &  model \\ \hline
     Chen's relation                  &      $\longleftrightarrow$           & structure group \\ \hline
     controlled rough path                  &      $\longleftrightarrow$                &  modelled
distribution \\ \hline
       rough integration                 &   $\longleftrightarrow$               &  reconstruction map \\ \hline
\end{tabular}
\\
\caption{Basic correspondences: rough paths $\longleftrightarrow$  regularity structures}
\label{my-label}

\end{table}
\noindent Furthermore, one has similar results concerning continuity properties of the solution
map as a function of the enhanced noise:

\begin{table}[!htbp]
\centering
\begin{tabular}{|l|}
\hline
   continuity of solution in (rough path $\longleftrightarrow $ model) topology \\ \hline
\end{tabular}
\end{table}

\noindent Unfortunately mollified lifted noise - in infinite dimensions - in general does not converge (as a model), hence {\it renormalization} plays a fundamental role in regularity structures.  The algebraic formalism of how to conduct this renormalization then relies on heavy Hopf algebraic considerations \cite{Hairer14}, pushed to 
great generality
in \cite{BHZ16}, see also \cite{Hairer16} for a
summary. Our investigation was driven by two questions:

\vspace{0.3cm}
{\it

\newcounter{intro_qs}
\begin{enumerate}[label=(\arabic*)]
\item \label{ques:fd_renorm} Are there meaningful (finite-dimensional) examples from stochastics
which require (infinite) renormalization?

\vspace{0.3cm}

\item \label{ques:alg_st} Do we have algebraic structures in rough path theory comparable
with those seen in regularity structures? 
\setcounter{intro_qs}{\value{enumi}}
\end{enumerate}

}

\vspace{0.3cm}

To be more specific, with regard to~\ref{ques:fd_renorm}, consider the situation of a differential equation driven by some finite-dimensional Brownian (or more general Gaussian) noise, mollified at scale $\eps$, followed by the question 
if the resulting (random) ODE solutions converge as $\eps \to 0$.
As remarked explicitly in~\cite[p.230]{FrizHairer14}, this is very often the case (with concrete results given in~\cite[Sec.10]{FrizHairer14}), with the potential caveat of area (and higher order) anomalities (\cite{FrizOberhauser09, FGL15} ...), leading to a more involved description (sometimes called {\it finite renormalization}) of the limit.
 We emphasize, however, that this is not always the case; there are perfectly meaningful (finite-dimensional) situations which require (infinite) renormalization, which we sketch in
Section~\ref{sec:introex} below together with precise references.
We further highlight that a natural example of geometric rough path (over $\R^2$) with a logarithmically diverging area term requiring (infinite) renormalization appears in Hairer's solution of the KPZ equation~\cite{KPZ}.
This situation is characteristic of singular SPDEs, in which the procedure described above typically leads to plain divergence, cured by ``subtracting infinities'', a.k.a. {\it infinite renormalization}.

Much effort in this work is then devoted to question~\ref{ques:alg_st}: In~\cite{BHZ16}, the algebraic formalism in
regularity structures relies crucially on two Hopf algebras, $\TT_{+}$ and $\TT_{-}$ (which are further in ``cointeraction''). The first one helps to
construct the structure group which, in turn, provides the recentering
(``positive renormalization'' in the language of \cite{BHZ16}) and hence constitutes
an abstract form of Chen's relation in rough path theory. In this sense, $\TT_{+}$ was always present in rough path theory, the point being enforced in
the model case of branched rough paths \cite{Gubinelli10, HairerKelly15} where $\TT_{+}$ is effectively given by the Connes-Kreimer Hopf algebra.

Question~\ref{ques:alg_st} is then reduced to the question if $\TT_{-}$, built to carry
out the actual renormalization of models, and subsequently SPDEs, (``negative
renormalization'' in the language of \cite{BHZ16}), has any correspondence in rough
path theory. Our answer is again affirmative and we establish the precise relation:

\begin{table}[!htbp]
\centering
\begin{tabular}{|l|}
\hline
   translation of rough paths $\longleftrightarrow$ renormalization of models \\ \hline
\end{tabular}
\end{table}

At last, during the course of writing this paper, we realized that we have been touching on a third important point, whose importance seems to transcend the
rough path setting in which it is discussed. 

\vspace{0.3cm}
{\it

\begin{enumerate}[label=(\arabic*)]
\setcounter{enumi}{\value{intro_qs}}
\item \label{ques:renorm_eq} How does one obtain from the renormalized model, in some algebraic and automated fashion, the renormalized equation? 
\end{enumerate}
}
\vspace{0.3cm}

It is indeed the algebraic approach to ``translation of rough paths'' (i.e. renormalization of a branched rough path model)
that indicated an important role played by pre-Lie structures, which first
appear in Section~\ref{subsec:transBranched} to
construct the translation operator (on forest series) and then to characterize its dual.
These considerations help answer the (not very precise) question of what pre-Lie structures (after all, a well-known tool in the renormalization theory, e.g.~\cite{Manchon11} and the references therein) have to do with rough paths and regularity structures. 
From a regularity structures perspective, a remarkable consequence is that this allows to understand {\it directly} the action of the renormalization group on the (to-be-renormalized) equation at hand, thus providing an answer to question~\ref{ques:renorm_eq}.
Indeed, by exploiting the pre-Lie structure of the space of trees, we obtain a direct conversion formula for the RDE driven by a translated branched rough path; see Section~\ref{subsec:effectsRDEs}, Theorem~\ref{thm:transRDEs}, and Remark~\ref{rem:new52}.
The analogous statement in regularity structures is an explicit form of an arbitrary renormalised SPDE, a result which was recently established in~\cite{BCCH17}.

\noindent Several remarks are in order. 

\begin{itemize}
\item We first develop the algebraic renormalization theory for rough paths
in its own right, analytic considerations then take place in Section \ref{sec:RDE}. 
The link to regularity structures and its renormalization
theory is only made in Section \ref{sec:renorm}.

\item While pre-Lie morphisms play a crucial role in the construction of translation maps, we point out that in certain situations the fine-details of pre-Lie structures are not visible; see the final point of Theorem~\ref{thm:transRPs}, as well as Remarks~\ref{remark:algMorphs} and~\ref{remark:algMorphs2}.

\item In Section~\ref{sec:introex} we present several examples, 
based on finite- (and even one-) dimensional Brownian motion, which do require genuine renormalization. 
Another interesting type of rough path renormalization, aiming at fractional Brownian (recall divergence of L\'evy area for Hurst parameter
 $H \le 1/4$, \cite{CQ02}) based on Fourier normal ordering, was proposed by Unterberger \cite{U13}. That said, his methods and aims are quite different from those considered in this paper and/or those from Hairer's regularity structures.

\item The existence of a finite-dimensional renormalization group is much
related to the stationarity of the (lifted) noise, see~\cite{Hairer14} and the recent article~\cite{ChandraHairerForth}. In the rough path context, this amounts to saying that a
random (branched) rough path $\mathbf{X}=\mathbf{X}\left( \omega \right) $,
with values in a (truncated) Butcher (hence finite-dimensional Lie) group $\GG$, actually has independent increments with respect to the Grossmann-Larson
product $\star$ (dual to the Connes-Kreimer coproduct $\Delta _{\star}$). In
other words, $\mathbf{X}$ is a (continuous) $\GG$-valued L\'{e}vy process.
This yields a close connection to the works~\cite{FrizShekhar17,Chevyrev18} devoted to the study of L{\'e}vy rough paths;
in Section~\ref{sec:levy}  we shall see how L{\'e}vy triplets {(or rather tuples in the absence of jumps)} behave under renormalization. 
\end{itemize}

{
\subsection{Geometric and non-geometric rough paths} \label{sec:GNG}  

In this subsection, we briefly recall the notions of geometric and branched rough paths; see~\cite{FrizVictoir10, Gubinelli10, HairerKelly15} for further details.
See also Sections~\ref{subsec:TensorPrelims} and~\ref{subsec:PrelimsForest} for further details on the algebraic structures involved. 

\textbf{Geometric rough paths.} Consider a path $X : [0,T] \to \R^d$.
A (geometric) rough path over $X$ is a map $\mathbf{X} : [0,T]^2 \to T((\R^d))$, where $T((R^d)) = \prod_{k=0}^\infty (\R^d)^{\otimes k}$ is the space of ``tensor series'' over $\R^d$, which should be thought of as the iterated integrals of $X$.
Equipping $\R^d$ with an inner product, 
we can identify $T((\R^d))$ with the algebraic dual of the tensor algebra
\begin{equ}
T(\R^d) = \R \oplus \R^d \oplus (\R^d)^{\otimes 2} \oplus \dots \;.
\end{equ}
One should think of the components of $\mathbf{X}$ as formally being given by
\begin{equ}\label{eq:iter_int}
\langle\mathbf{X}_{s,t}, e_{i_1 \dots i_n}\rangle := \int_s^t \dots \int_s^{t_{2}} dX_{t_1}^{i_1} \dots dX_{t_n}^{i_n}\;,
\end{equ}
for $i_1,\dots,i_n \in \{ 1 , \dots, d\}$, where $X^i_{t} = \scal{X_t,e_i}$ and where we use the shorthand $e_{i_1 \dots i_n} = e_{i_1}\tensor \dots \tensor e_{i_n}$ with $\tensor$ denoting the tensor product in $T(\R^d)$.
We emphasize that, unless $n=1$, the definition~\eqref{eq:iter_int} is, in general, only formal and one should think of the rough path $\mathbf{X}$ as defining the RHS.
\par
Observe that if $X$ is smooth and~\eqref{eq:iter_int} is used to define $\mathbf{X}$, then the so-called shuffle identity holds
\begin{equ}\label{eq:shuffle}
\langle\mathbf{X}_t,e_{i_1 \dots i_n}\rangle\langle\mathbf{X}_t,e_{j_1 \dots j_m}\rangle= \langle\mathbf{X}_t,e_{i_1 \dots i_n} \shuffle e_{j_1 \dots j_m}\rangle\;, \quad \text{for all } e_{i_1 \dots i_n}, e_{j_1 \dots j_m} \in T(\R^d)\;,
\end{equ}
where $\shuffle$ denotes the (commutative) shuffle product~\cite{reutenauer93}.
While we do not give the definition of $\shuffle$ here or prove this identity, we remark that it is a direct consequence of integration by parts.
Another important algebraic identity which holds in this case is Chen's relation
\begin{equ}
\mathbf{X}_{s,t} = \mathbf{X}_{s,u}\tensor\mathbf{X}_{u,t} \;, \quad \text{for all } s,t,u \in [0,T]\;,
\end{equ}
which can be shown by an application of Fubini's theorem.

The concept of a (weakly geometric) rough path should be thought of as a generalisation of these identities to paths of lower regularity.

\begin{definition}\label{def:geo_rp}
Let $\gamma \in (0,1]$.
A \emph{$\gamma$-H{\"o}lder weakly geometric rough path} is a map $\mathbf{X}:[0,T]^2\to T((\R^d))$ satisfying
\begin{enumerate}[label=\roman*)]
\item\label{pt:shuffle} $\langle\mathbf{X}_{st},x \shuffle y\rangle = \langle\mathbf{X}_{st},x\rangle\langle\mathbf{X}_{st},y\rangle$, for all $x,y \in T(\R^d)$,
\item\label{pt:chen} $\mathbf{X}_{st} = \mathbf{X}_{su}\tensor \mathbf{X}_{ut}$ for all $s,t,u \in [0,T]$,
\item\label{pt:holder} $\sup_{s\neq t} {|\langle\mathbf{X}_{st},e_{i_1,\ldots, i_n}}|\rangle/{|t-s|^{\gamma n}} < \infty$, for all $n \geq 1$ and $i_1,\ldots, i_n \in \{1,\ldots, d\}$.
\end{enumerate}
\end{definition}

\textbf{Branched rough paths.} One is often interested in paths $X$ for which natural definitions of ``iterated integrals'' do not satisfy classical integration by parts and thus do not constitute geometric rough paths, e.g., integrals defined in the sense of It{\^o} for a semi-martingale $X$.
Branched rough paths are a generalisation of geometric rough paths which allows for violation of the shuffle identity~\eqref{eq:shuffle} and thus of the usual rules of calculus.
This is achieved by substituting the space $T((\R^d))$ with a larger (Hopf) algebra $\mathcal{H}^*$ in which natural generalisations of properties~\ref{pt:shuffle},~\ref{pt:chen}, and~\ref{pt:holder} are required to hold.
The Hopf algebra $\HH^*$ is known as the Grossman--Larson algebra of series of forests, and is the algebraic dual of the Connes--Kreimer Hopf algebra~\cite{connes98} consisting of polynomials of rooted trees with nodes decorated by the set $\{1,\ldots, d\}$.

Denoting by $\poly$ the (commutative) polynomial product on $\HH$ and by $\star$ the (non-commutative) Grossman--Larson product on $\HH^*$, we have the following analogue of Definition~\ref{def:geo_rp}. 

\begin{definition}\label{def:branched_rp}
Let $\gamma \in (0,1]$.
A \emph{$\gamma$-H{\"older} branched rough path} is a map $\mathbf{X}:[0,T]^2\to\mathcal{H}^*$ satisfying
\begin{enumerate}[label=\alph*)]
\item\label{pt:poly} $\langle\mathbf{X}_{st}, \tau_1 \poly\tau_2\rangle = \langle\mathbf{X}_{st},\tau_1\rangle\langle\mathbf{X}_{st},\tau_2\rangle$ for all $\tau_1,\tau_2 \in \HH$,
\item\label{pt:chen2} $\mathbf{X}_{st} = \mathbf{X}_{su} \star \mathbf{X}_{ut}$ for all $s,t,u \in [0,T]$,
\item\label{pt:holder2} $\sup_{s\neq t} |\langle\mathbf{X}_{st},\tau\rangle|/{|t-s|^{\gamma |\tau|}} < \infty$ for every forest $\tau \in \mathcal{H}$, where $|\tau|$ denotes the number of nodes in $\tau$.
\end{enumerate}
\end{definition}

Here we set $\langle\mathbf{X}_{s,t}, {\bullet_i} \rangle := X^i_{s,t}$ and then think of the components of $\mathbf{X}$ given by the formal recursion
\begin{equ}\label{eq:iter_int_branched}
\langle\mathbf{X}_{s,t}, [\tau_1 \poly \ldots \poly \tau_n]_{\bullet_i}\rangle = \int_s^t \scal{\mathbf{X}_{s,u},\tau_1}\ldots\scal{\mathbf{X}_{s,u},\tau_n} d X^i_u
\end{equ}
for trees $\tau_1,\ldots,\tau_n \in \HH$ and $i \in \{1,\ldots, d\}$, where $[\tau_1\poly\ldots\poly\tau_n]_{\bullet_i}$ denotes the tree formed by grafting the trees $\tau_1,\ldots, \tau_n$ onto a single root with label $i$.
If $X$ is smooth and one uses~\eqref{eq:iter_int_branched} to define $\mathbf{X}$, then, as before, points~\ref{pt:poly} and~\ref{pt:chen2} are direct consequences of integration by parts and Fubini's theorem respectively.

Equipping $T((\R^d))$ with the tensor Hopf algebra structure, there is a canonical graded embedding of Hopf algebras $T((\R^d)) \hookrightarrow \HH^*$.
Points~\ref{pt:poly},~\ref{pt:chen2}, and~\ref{pt:holder2} are therefore generalisations of points~\ref{pt:shuffle},~\ref{pt:chen}, and~\ref{pt:holder}, hence every geometric rough path constitutes a branched rough path.
We emphasize however that this embedding is strict and~\ref{pt:poly} is more general than~\ref{pt:shuffle}, which allows a general branched rough path $\mathbf{X}$ to violate classical integration by parts.
For example, if $\mathbf{X}$ is defined via~\eqref{eq:iter_int_branched} using It{\^o} integrals for a semi-martingale $X$, then $\mathbf{X}$ is an example of a $\gamma$-H{\"o}lder branched (but in general not geometric!) rough path for any $\gamma \in (0,\frac12)$.
}

\subsection{Translation of paths} \label{sec:12}

Consider a $d$-dimensional path $X_{t}$, written with respect to an orthonormal basis $e_{1},\dots ,e_{d}$ of $\R^{d}$,
\[
X_{t}=\sum_{i=1}^{d}X_{t}^{i}e_{i}.
\]
We are interested in constant speed perturbations, of the form
\[
T_{v}X_{t}:=X_{t}+tv,\,\,\,\text{\ with }v=\sum_{i=1}^{d}v^{i}e_{i}\in \R^{d}
\text{.}
\]
In coordinates, $\left( T_{v}X_{t}\right) ^{i}=X_{t}^{i}+tv^{i}$ for $i=1,\dots, d$, i.e,
\[
\left\langle T_{v}X,e_{i}\right\rangle =\left\langle
X_{t},e_{i}\right\rangle +\left\langle tv,e_{i}\right\rangle
.
\]
Consider now an orthonormal basis $e_{0},e_{1},\dots ,e_{d}$ of $\R^{1+d}$, and consider the $\R^{1+d}$-valued ``time-space'' path 
$$\bar{X}
_{t}=X_{t}+X_{t}^{0}e_{0}=\sum_{i=0}^{d}X_{t}^{i}e_{i}
$$ with scalar-valued $X_{t}^{0}\equiv t$. We can now write
\[
T_{v}\bar{X}_{t}=\bar{X}_{t}+tv=X_{t}+X_{t}^{0}\left( e_{0}+v\right) 
\]
which identifies $T_{v}$ as linear map on $\R^{1+d}$, which maps $e_{0}\mapsto e_{0}+v$, and $e_{i}\mapsto e_{i}$ for $i=1,\dots, d$.
We then can (and will) also look at general endomorphisms
of the vector space $\R^{1+d}$, which we still write in the form  
\begin{eqnarray*}
e_{j} &\mapsto &e_{j}+v_{j},\,\,j=0,\dots \,,d \\
v_{j} &=&\sum_{i=0}^{d}v_{j}^{i}e_{i}\in \R^{1+d}.
\end{eqnarray*}
(The initially discussed case corresponds to $\left( v_{0},v_{1},\dots
,v_{d}\right) =\left( v_{0},0,\dots ,0\right) $, with {$\scal{v_{0},e_{0}}=0$,} and much of the sequel,
will take advantage of this additional structure.)

\bigskip

We shall be interested to understand how such perturbations propagate to higher level iterated integrals, whenever $X$ has sufficient structure to make this meaningful.
For instance, if $X=B(\omega)$, a $d$-dimensional Brownian motion, an object of interest would be, with repeated (Stratonovich) integration over $\{ (r,s,t): 0 \le r \le s \le t \le T \}$,
$$
    (T_v B)_{0,T}^{ijk} :=  \int \circ (dB^i + v^i\, dr ) \circ (dB^j + v^j\, ds)  \circ (dB^k + v^k\, dt) =  B_{0,T}^{ijk} + ... $$
where the omitted terms (dots) can be spelled out (algebraically) in terms of contractions of $v$ (resp. tensor-powers of $v$) and iterated integrals of $(1+d)$-dimensional time-space Brownian motion ``$(t,B)$''. (Observe that we just gave a dual description of this perturbation, as seen on the third level, while the initial perturbation took place at the first level: $v$ is a vector here.)

There is interest in higher-level perturbations. In particular, given a $2$-tensor {$v = \sum_{i,j=1}^d v^{ij} e_{i,j}$,} we can consider the level-2-perturbation with no effect on the first level, i.e., $(T_v B)_t^{i} \equiv B_t^{i}$, while for all $i,j=1,...,d$,
$$
        (T_v B)_t^{ij} = B_t^{ij} + v^{ij}\, t
$$
For instance, writing $B^{I;w}$ for iterated It\^{o} integrals, in contrast to $B^w$ defined by iterated Stratonovich integration, we have with $v := \frac{1}{2} I^d$ where $(I^d)^{ij} = \delta^{ij}$, i.e., the identity matrix,
$$
       (T_v B^I)_t^{ij} = B_t^{ij}.
$$
This is nothing but a restatement of the familiar formula $ \int_0^t B^i dB^j +\tfrac{1}{2} \delta^{ij} t = \int_0^t B^i \circ dB^j$.
It is a non-trivial exercise to understand the It\^{o}-Stratonovich correction at the level of higher iterated integrals, cf.~\cite{BenArous89} and a ``branched'' version thereof discussed in Section~\ref{subsec:ItoStrat} below.
{Further examples where such translations serve as a ``renormalisation'' are discussed in Section~\ref{sec:introex}, notably the case $\bar B^{ij} = (T_a B)_t^{ij}$ with an anti-symmetric $2$-tensor $a = (a^{ij})$ which arises in the study of Brownian particles in a magnetic field.} 

It will be important for us to understand (explicitly) how to formulate (constant speed, higher) order translations, an analytic operation on rough paths, algebraically and ``point-wise''  terms of the time-space rough path.

\subsection{Organization of the paper}

This paper is organized as follows.
In Section~\ref{sec:TgRP}, we first discuss renormalization/translation in the by now well established setting of geometric rough paths.
The algebraic 
background is found for instance in~\cite{reutenauer93}.
We then, in Section~\ref{sec:TbRP}, move to branched rough paths~\cite{Gubinelli10}, in the notation and formalism from Hairer-Kelly~\cite{HairerKelly15}, and in particular introduce the relevant pre-Lie structures.
In Section~\ref{sec:Ex} we illustrate the use of the (branched) translation operator (additional examples were already mentioned in Section 
\ref{sec:introex}), while in Section \ref{sec:RDE} we describe the analytic and algebraic effects of such translations on rough paths and associated RDEs.
Lastly, Section~\ref{sec:renorm} is devoted to the systematic comparison of the translation operator and ``negative renormalization'' introduced in~\cite{BHZ16}.

\medskip

{\bf Acknowledgements.} P.K.F. is partially supported by the European Research Council through CoG-683164 and DFG research unit FOR2402.
I.C., affiliated to TU Berlin when this project was commenced, was supported by DFG research unit FOR2402, and is currently supported by a Junior Research Fellowship of St John's College, Oxford. 
Y.B. thanks Martin Hairer for financial support from his Leverhulme Trust award.
R.P. is supported by European Research Council through CoG-683164 and was affiliated to Max
Planck Institute for Mathematics in the Sciences, Leipzig, in autumn 2018.
R.P. also thanks
Alexander Schmeding for helpful discussions and information on topological
questions.

\section{Translation of geometric rough paths} \label{sec:TgRP}

{\it We review the algebraic setup for geometric rough paths, as enhancements of $X=(X_0,X_1, ... , X_d)$, a signal with values in $V=\R^{1+d}$. Recall the natural state-space of such rough paths is $T((V))$, a space of tensor series (resp. a suitable truncation thereof related to the regularity of $X$). Typically $\dot X \equiv (\xi_0,\xi_1,..., \xi_d)$ models noise. Eventually, we will be interested in $X_0 (t) = t$, so that $X$ is a time-space (rough) path, though this plays little role in this section.
}

\subsection{Preliminaries for tensor series}\label{subsec:TensorPrelims}

We first establish the notation and conventions used throughout the paper. Most algebraic aspects used in this section may be found in~\cite{reutenauer93} and~\cite{FrizVictoir10} Chapter~7.

Throughout the paper we let $\{e_0, e_1,\ldots, e_d\}$ be a basis for $\R^{1+d}$.
Consider the vector space of formal tensor series over $\R^{1+d}$
\[
\GDual = \prod_{k=0}^\infty (\R^{1+d})^{\otimes k}
\]
(with the usual convention $(\R^{1+d})^{\otimes 0} = \R$), as well as the vector space of polynomials over $\R^{1+d}$
\[
\GPrimal = \bigoplus_{k=0}^\infty (\R^{1+d})^{\otimes k}.
\]
Note that $\GPrimal$ and $\GDual$ can equivalently be considered as the vector space of words and non-commutative series respectively in $1+d$ indeterminates.

\medskip

\noindent Recall that $\GDual$ can be equipped with a Hopf-type\footnote{The structure here is not exactly of a Hopf algebra since $\Deltashuff$ does not map $\GDual$ into $\GDual^{\otimes 2}$, but rather into the complete tensor product $\GPrimal^{\overline\otimes 2} \simeq \prod_{k,m=0}^\infty (\R^{1+d})^{\otimes k}\otimes (\R^{1+d})^{\otimes m}$, see~\cite[Sec.~1.4]{reutenauer93}.} algebra structure
\[
(\GDual, \tensor ,\Deltashuff, \alpha)
\]
with tensor (concatenation) product $\tensor$, coproduct $\Deltashuff$ which is dual to the shuffle product $\shuffle$ on $\GPrimal$, and antipode $\alpha$. Recall that $\Deltashuff$ is explicitly given as the unique continuous\footnote{We equip henceforth $\GDual$ and $\GPrimal\overline\otimes\GPrimal$ with the product topology.} algebra morphism such that
\begin{align*}
\Deltashuff &: \GDual \rightarrow \GPrimal\overline\otimes\GPrimal \\
\Deltashuff &: v \mapsto v\otimes 1 + 1\otimes v, \; \; \text{for all } v \in \R^{1+d}.
\end{align*}
We shall often refer to elements $e_{i_1}\tensor\ldots \tensor e_{i_k}$ as words consisting of the letters $e_{i_1}, \ldots, e_{i_k} \in \{e_0,\ldots, e_d\}$, and shall write $e_{i_1,\ldots, i_k} = e_{i_1}\tensor\ldots \tensor e_{i_k}$. We likewise denote by
\[
(\GPrimal, \shuffle, \Delta_\tensor, \tilde \alpha)
\]
the shuffle Hopf algebra.
Recall that by identifying $\R^{1+d}$ with its dual through the basis $\{e_0,\ldots, e_d\}$, there is a natural duality between $\GPrimal$ and $\GDual$ in which $\shuffle$ is dual to $\Delta_\shuffle$, and $\tensor$ is dual to $\Delta_\tensor$.

We let $G(\R^{1+d})$ and $\LL((\R^{1+d}))$ denote the set of group-like and primitive elements of $\GDual$ respectively. Recall that $\LL((\R^{1+d}))$ is precisely the space of Lie series over $\R^{1+d}$, and that
\[
G(\R^{1+d}) = \exp_\tensor(\LL((\R^{1+d}))).
\]

For any integer $N \geq 0$, we denote by $T^N(\R^{1+d})$ the truncated algebra obtained as the quotient of $\GDual$ by the ideal consisting of all series with no words of length less than $N$ (we keep in mind that the tensor product is always in place on $T^N(\R^{1+d})$). Similarly we let $\G^N(\R^{1+d}) \subset T^N(\R^d)$ and $\LL^N(\R^{1+d}) \subset T^N(\R^{1+d})$ denote the step-$N$ free nilpotent Lie group and Lie algebra over $\R^{1+d}$ respectively, constructed in analogous ways.

Finally, we identify $\R^d$ with the subspace of $\R^{1+d}$ with basis $\{e_1,\ldots, e_d\}$. From this identification, we canonically treat all objects discussed above built from $\R^d$ as subsets of their counterparts built from $\R^{1+d}$.
For example, we treat the algebra $T((\R^d))$ and Lie algebra $\LL^N(\R^d)$ as a subalgebra of $\GDual$ and a Lie subalgebra of $\LL^N(\R^{1+d})$ respectively.

\subsection{Translation of tensor series}\label{subsec:TensorTranslate}

Recall that, by the universal property of $\GPrimal$ and the graded structure of $\GDual$, any linear map $M: \R^{1+d} \to \GDual$ such that $M(e_i)$ has no component of order zero (i.e., $\scal{M(e_i),1} = 0$) for all $i \in \{0,\ldots, d\}$ extends uniquely to a continuous algebra morphism $M : \GDual \to \GDual$.

\begin{definition}\label{def:TensorTranslate}
For a collection of Lie series $v = (v_0,\ldots, v_d)\subset \LL((\R^{1+d}))$, define $T_v : \GDual \rightarrow \GDual$ as the unique extension to a continuous algebra morphism
of the linear map
\begin{align*}
T_v &: \R^{1+d} \rightarrow \LL((\R^{1+d})) \subset \GDual \\
T_v &: e_i \mapsto e_i + v_i, \; \; \text{for all } i \in \{0,\ldots, d\}.
\end{align*}
\end{definition}

In the sequel we shall often be concerned with the case that $v_i = 0$ for $i = 1,\ldots, d$ and $v_0$ takes a special form. We shall make precise whenever such a condition is in place by writing, for example, $v = v_0 \in \LL^N(\R^{d})$.

We observe the following immediate properties of $T_v$:

\begin{itemize}
\item Since $T_v$ is a continuous algebra morphism which preserves the Lie algebra $\LL((\R^{1+d}))$, it holds that $T_v$ maps $G(\R^{1+d})$ into $G(\R^{1+d})$;

\item $T_v \circ T_u = T_{v + T_v(u)}$, where we write $T_v(u) := (T_v(u_0),\ldots, T_v(u_d))$. In particular, $T_{v+u} = T_v \circ T_u$ for all $v = v_0,u = u_0 \in \LL((\R^{d}))$;

\item For every integer $N \geq 0$, $T_v$ induces a well-defined algebra morphism $T^N_v : T^N(\R^{1+d}) \rightarrow T^N(\R^{1+d})$, which furthermore maps $\G^N(\R^{1+d})$ into itself.
\end{itemize}

The following lemma moreover shows that $T_v$ respects the Hopf algebra structure of $\GDual$.
Note that $\GDual^{\otimes 2}$ embeds densely into $\GPrimal^{\overline\otimes2}$, and thus $T_v\otimes T_v$ extends uniquely to a continuous algebra morphism $\GPrimal^{\overline\otimes 2} \to \GPrimal^{\overline\otimes 2}$.

\begin{lemma}\label{lem:HopfMorph}
The map $T_v : \GDual \rightarrow \GDual$  satisfies $(T_v \otimes T_v) \Delta_\shuffle = \Delta_\shuffle T_v$ and commutes with the antipode $\alpha$.
\end{lemma}

\begin{proof}
To show that $(T_v \otimes T_v) \Delta_\shuffle = \Delta_\shuffle T_v$, note that both $(T_v \otimes T_v) \Delta_\shuffle$ and $\Delta_\shuffle T_v$ are continuous algebra morphisms, and so they are equal provided they agree on $e_0,\ldots, e_d$. Indeed, we have
\[
\Delta_\shuffle T_v (e_i) = \Delta_\shuffle (e_i + v_i) = 1 \otimes (e_i + v_i) + (e_i + v_i) \otimes 1
\]
(here we used that each $v_i$ is a Lie element, i.e., primitive in the sense $\Delta_\shuffle v_i = 1 \otimes v_i + v_i \otimes 1$) and
\[
(T_v \otimes T_v) \Delta_\shuffle (e_i) = (T_v\otimes T_v) (1\otimes e_i + e_i\otimes 1) = 1 \otimes (e_i + v_i) + (e_i + v_i) \otimes 1,
\]
as required.
It remains to show that $T_v$ commutes with the antipode $\alpha$.
Actually, this is implied by general principles (e.g. \cite[Theorem~2.14]{Preiss16}, and the references therein), but as it is short to spell out, we give a direct argument: consider the opposite algebra $(\GDual)^{\op}$ (same set and vector space structure as $\GDual$ but with reverse multiplication). Then $\alpha : \GDual \rightarrow (\GDual)^{\op}$ is an algebra morphism, and again it suffices to check that $\alpha T_v$ and $T_v \alpha$ agree on $e_0,\ldots, e_d$. Indeed, since $v_i \in \LL((\R^{1+d}))$, we have $\alpha(v_i) = -v_i$, and thus
\[
\alpha T_v(e_i) = \alpha(e_i + v_i) = -e_i - v_i
\] 
and
\[
T_v \alpha(e_i) = T_v(-e_i) = -e_i - v_i.
\]
\end{proof}

\subsection{Dual action on the shuffle Hopf algebra \texorpdfstring{$T(\R^{1+d})$}{T(R\^{}(1+d))}} \label{subsec:DualMapGeom}

We now wish to describe the dual map $T_v^* : \GPrimal \rightarrow \GPrimal$ for which
\[
\scal{T_v x, y} = \scal{x, T_v^* y}, \; \; \text{for all } x \in \GDual, \; \; y \in \GPrimal.
\]
We note immediately that Lemma~\ref{lem:HopfMorph} implies $T^*_v$ is a Hopf algebra morphism from $(\GPrimal, \shuffle, \Delta_\tensor, \tilde \alpha)$ to itself.

For simplicity, and as this is the case most relevant to us, we only consider in detail the case $v = v_0 \in \LL((\R^{1+d}))$, i.e., $v_i = 0$ for $i = 1,\ldots, d$ (but see Remark~\ref{remark:generalGeom} for a description of the general case).

Let $\SSS$ denote the unital free commutative algebra generated by the non-empty words $e_{i_1,\ldots,i_k} = e_{i_1}\tensor \ldots\tensor e_{i_k}$ in $\GPrimal$. We let $\mathbf{1}$ and $\cdot$ denote the unit element and product of $\SSS$ respectively. For example, 
\begin{align*}
e_{0,1} \cdot e_2 &= e_2 \cdot e_{0,1} \in \SSS, \\
e_0 \cdot e_{1,2} &\ne e_0 \cdot  e_{2,1} \in \SSS.
\end{align*}

For a word $w \in \GPrimal$, we let $D(w)$ denote the set of all elements
\[
w_1\cdot\ldots \cdot w_k \otimes \tilde w \in \SSS \otimes T(\R^{1+d})
\]
where $w_1,\ldots,w_k$ is formed from disjoint subwords of $w$ and $\tilde w$ is the word obtained by replacing every $w_i$ in $w$ with $e_0$ (note that $\mathbf{1} \otimes w$, corresponding to $k=0$, is also in $D(w)$).

Consider the linear map $S : \GPrimal \rightarrow \SSS \otimes \GPrimal$ defined for all words $w \in \GPrimal$ by
\[
S(w) = \sum_{w_1\cdot\ldots \cdot w_k \otimes \tilde w \in D(w)} w_1\cdot\ldots \cdot w_k \otimes \tilde w.
\]
For example
\begin{align*}
S(e_{0,1,2}) =& \mathbf{1} \otimes (e_{0,1,2,}) \\
&+ (e_0)\otimes (e_{0,1,2}) + (e_1)\otimes(e_{0,0,2}) + (e_2) \otimes (e_{0,1,0}) \\
&+ (e_0 \cdot e_1) \otimes (e_{0,0,2}) + (e_0 \cdot e_2)\otimes (e_{0,1,0}) + (e_1 \cdot e_2) \otimes (e_{0,0,0}) \\
&+ (e_0 \cdot e_1 \cdot e_2)\otimes (e_{0,0,0}) + (e_{0,1})\otimes (e_{0,2}) + (e_{1,2}) \otimes (e_{0,0}) \\
&+ (e_{0,1} \cdot e_2) \otimes (e_{0,0}) + (e_0 \cdot e_{1,2}) \otimes (e_{0,0}) \\
&+ (e_{0,1,2}) \otimes (e_0).
\end{align*}

\begin{proposition}\label{prop:adjointTv}
Let $v = v_0 \in \LL((\R^{1+d}))$. The dual map $T_v^* : \GPrimal \rightarrow \GPrimal$ is given by
\[
T_v^* w = (v \otimes \id) S(w),
\]
where $v(w_1 \cdot \ldots \cdot w_k) := \scal{w_1,v}\ldots\scal{w_k, v}$ and $v(\mathbf{1}) := 1$.
\end{proposition}

In principle, Proposition~\ref{prop:adjointTv} can be proved algebraically by showing that the adjoint of $\Phi := (v \otimes \id) S$ is an algebra morphism from $\GDual$ to itself, and check that $\Phi^*(e_i) = T_v(e_i)$ for every generator $e_i$. Indeed this is the method used in Section~\ref{subsec:DualMapBranched} to prove the analogous result for the translation map on branched rough paths. However in the current setting of geometric rough paths, we can provide a direct combinatorial proof.

\begin{proof}
Note that the claim is equivalent to showing that for every two words $u, w \in \GPrimal$ (treating $u \in \GDual$)
\begin{equation}\label{eq:tensorSum}
\scal{T_v u, w} = \sum_{w_1\cdot\ldots \cdot w_k \otimes \tilde w \in D(w)} \scal{w_1,v} \ldots \scal{w_k, v} \scal{\tilde w, u}.
\end{equation}
Consider a word $u =e_{i_1}\tensor \ldots \tensor e_{i_k} \in \GPrimal$. Then
\[
T_v(u) = e_{i_1} \tensor \ldots \tensor (e_0 +v) \tensor \ldots \tensor e_{i_k},
\]
where every occurrence of $e_0$ in $u$ is replaced by $e_0 + v$. We readily deduce that for every $w \in \GPrimal$
\begin{equation}\label{eq:wiSum}
\scal{T_v(u),w} = \sum_{\substack{w_1\cdot\ldots \cdot w_k \otimes \tilde w \in D(w) \\ u=\tilde w}} \scal{w_1,v} \ldots \scal{w_k, v}.
\end{equation}
For example, with $v = [e_1,e_2] = e_{1,2} - e_{2,1}$ and $u = e_{0,1,2}$, we have
\[
T_v(u) = e_{0,1,2} + e_{1,2,1,2} - e_{2,1,1,2},
\]
and we see that indeed for
\[
w \in A := \{e_{0,1,2}, e_{1,2,1,2}, e_{2,1,1,2}\},
\]
the right hand side of~\eqref{eq:wiSum} gives $\scal{T_v(u),w}$, whilst $\scal{w_1,v}\ldots\scal{w_k,v} = 0$ for all $w \in \GPrimal \setminus A$ and $w_1 \cdot \ldots \cdot w_k \otimes \tilde w \in D(w)$ such that $u= \tilde w$. But now~\eqref{eq:wiSum} immediately implies~\eqref{eq:tensorSum}.
\end{proof}

\begin{remark}\label{remark:generalGeom}
A similar result to Proposition~\ref{prop:adjointTv} holds for the general case $v = (v_0,\ldots, v_d)$. The definition of $S$ changes in the obvious way that in the second tensor, instead of replacing every subword by the letter $e_0$, one instead replaces every combination of subwords by all combinations of $e_i$, $i \in \{0,\ldots, d\}$, while in the first tensor, one marks each extracted subword $w_j$ with the corresponding label $i \in \{0, \ldots, d\}$ that replaced it, which gives $(w_j)_i$ (so the left tensor no longer belongs to $\SSS$ but instead to the free commutative algebra generated by $(w)_i$, for all words $w \in \GPrimal$ and labels $i \in \{0,\ldots, d\}$). Finally the term $\scal{w_1,v}\ldots\scal{w_k, v}$ would then be replaced by $\scal{(w_1)_{i_1}, v_{i_1}}\ldots \scal{(w_k)_{i_k}, v_{i_k}}$.
\end{remark}

\section{Translation of branched rough paths}   \label{sec:TbRP}
\label{sec:Branched}

{\it In the previous section we studied the translation operator $T$, in the setting relevant for geometric rough path. Here we extend these results to the branched rough path setting, calling the translation operator $M$ to avoid confusion. Our construction of $M$ faces new difficulties, which we resolve with pre-Lie structures. The dual view then leads us to an extraction procedure of subtrees (a concept familiar from regularity structures, to be explored in Section~\ref{sec:renorm}).
}

\subsection{Preliminaries for forest series}\label{subsec:PrelimsForest}

As in the preceding section, we first introduce the notation used throughout the section. Our setup closely follows Hairer-Kelly~\cite{HairerKelly15}. (For additional  algebraic background the reader can consult e.g. ~\cite{GraciaBondia01}, Chapter~14.)

{  Recall that a rooted tree is a finite connected graph without cycles with a distinguished node called the root.
A rooted tree is unordered if there is no order on the edges leaving a node.}
We let $\BB  = \BB (\bullet_0, ... ,\bullet_d)$ denote the real vector space spanned by the set of unordered rooted trees with vertices labelled from the set $\{0,\ldots, d\}$.
We denote by $\BB^*$ its algebraic dual, which we identify with the space of formal series of labelled trees; we write $\BB^* = \BB^*  (\bullet_0, ... ,\bullet_d)$ accordingly.
We canonically identify with $\R^{1+d}$ the subspace of $\BB$ (and of $\BB^*$) spanned by the trees with a single node $\{\bullet_0,\ldots, \bullet_d\}$.

We further denote by $\BPrimal = \BPrimal ( \bullet_0, ... , \bullet_d)$ the vector space spanned by (unordered) forests composed of trees (including the empty forest denoted by $1$), and let $\BDual  = \BDual      (\bullet_0, ... ,\bullet_d)$ denote its algebraic dual which we identify with the space of formal series of forests. We canonically treat $\BB^*$ as a subspace of $\BDual$. Following commonly used notation (e.g.~\cite{HairerKelly15}),
for trees $\tau_1,\ldots,\tau_n \in \BB$ we let $[\tau_1\ldots\tau_n]_{\bullet_i} \in \BB$ denote the forest $\tau_1\ldots\tau_n \in \BPrimal$ grafted onto the node $\bullet_i$.

We equip $\BDual$ with the structure of the Grossman-Larson Hopf-type\footnote{ Again, $\Delta_\poly$ does not map $\BDual$ into ${\BDual}^{\otimes 2}$, but instead into the complete tensor product $\BPrimal^{\overline\otimes 2} \simeq \prod_{k,m=0}^\infty \mathcal{H}^{(k)} \otimes \mathcal{H}^{(m)}$, where $\mathcal{H}^{(k)}$ denotes the vector space of forests with $k$ vertices, and therefore the structure is not exactly that of a Hopf algebra.
Note also that $\Delta_\poly$ is continuous for the product topologies, which we equip $\BDual$ and $\BPrimal^{\overline\otimes 2}$ with henceforth.} algebra
\[
(\BDual, \star, \Delta_\poly, \alpha)
\]
and $\BPrimal$ with the structure of the dual graded Hopf algebra (the Connes-Kreimer Hopf algebra)
\[
(\BPrimal, \poly, \Delta_\star, \tilde\alpha).
\]
In other words, $\BPrimal$ is the free commutative algebra over $\BB$ equipped with a coproduct $\Delta_\star$, and graded by the number of vertices in a forest. We shall often drop the product $\poly$ and simply write $\tau\poly \sigma = \tau \sigma$.

The coproduct $\Delta_\star$ may be described in terms of admissible cuts, for which we use the convention to keep the ``trunk'' on the right: for every tree $\tau \in \BB$
\[
\Delta_\star \tau = \sum_c \tau^c_1\ldots \tau^c_k \otimes \tau^c_0,
\]
where we sum over all admissible cuts $c$ of $\tau$, and denote by $\tau^c_0$ the trunk and by $\tau^c_1\ldots \tau^c_k$ the branches of the cut respectively.

In the sequel, we shall also find it convenient to treat the space $\BPrimal$ equipped with $\star$ as a subalgebra of $\BDual$, in which case we explicitly refer to it as the algebra $(\BPrimal,\star)$.

Recall that the space of series $\BB^*$ is exactly the set of primitive elements of $\BDual$. We let $\GG = \GG ( \bullet_0, ... , \bullet_d)$ denote the group-like elements of $\BDual$, often called the Butcher group, for which it holds that
\[
\GG = \exp_\star (\BB^*).
\]

All the objects introduced above play an analogous role to those of the previous section. To summarise this correspondence, it is helpful to keep the following picture in mind

\begin{align*}
&\text{{``Series space''} ...} &&   \BDual  ( \bullet_0, ... , \bullet_d) \equiv \BDual &\longleftrightarrow &&\GDual \\
&\text{{``Polynomial space''} ...} && \BPrimal ( \bullet_0, ... , \bullet_d) \equiv  \BPrimal   &\longleftrightarrow &&\GPrimal \\
&\text{Lie elements ...} && \BB^* ( \bullet_0, ... , \bullet_d) \equiv \BB^* \subset \BDual  &\longleftrightarrow && \LL((\R^{1+d})) \\
&\text{Group-like elements ...} && \GG ( \bullet_0, ... , \bullet_d) \equiv  \GG \subset \BDual &\longleftrightarrow && G(\R^{1+d}).
\end{align*}

As in the previous section, for any integer $N \geq 0$ we let $\BPrimal^N$ denote ``truncated'' algebra obtained by the quotient of $\BDual$ by the ideal consisting of all series with no forests having less than $N$ vertices (we keep in mind that the product $\star$ is always in place for the truncated objects). Similarly we let $\GG^N \subset \BPrimal^N$ and $\BB^N \subset \BPrimal^N$ denote the step-$N$ Butcher Lie group over $\R^{1+d}$ its and Lie algebra respectively, constructed in analogous ways.

Finally, as before, we write ``$(\R^{d})$'' to denote the analogous objects built over $\R^d$, treated as subsets of their ``full'' counterparts built over $\R^{1+d}$ (by identifying $\R^d$ with the subspace of $\R^{1+d}$ with basis $\{e_1,\ldots, e_d\}$). For example, we treat $\BDual(\R^d)$ and $\BB^N(\R^d)$ as a subalgebra of $\BDual$ and a Lie subalgebra of $\BB^N$ respectively.

\subsection{Translation of forest series} 
\label{subsec:transBranched}

\subsubsection{Non-uniqueness of algebra extensions}

In the previous section, we defined a map $T_v$ which ``translated'' elements in $\GDual$ in directions $(v_0,\ldots,v_d) \subset \LL((\R^{1+d}))$, and which mapped the set of group-like elements $\G(\R^{1+d})$ into itself. In the same spirit, we aim to define a map $M_v$ which translates elements in $\BDual$ in directions $(v_0,\ldots,v_d) \subset \BB^*$, and which likewise maps $\GG$ into itself.

Note that our construction of $T_v$ relied on the fact that any linear map $M : \R^{1+d} \rightarrow \GDual$ {such that $\langle M v,1\rangle=0$ for $v\in\R^{1+d}$ } extended uniquely to a continuous algebra morphism $M : \GDual \rightarrow \GDual$ (for the product $\tensor$). We note here that no such universal property holds for $\BDual$; indeed, there exists a canonical injective algebra morphism
\begin{align}
\begin{split}\label{eq:tensorForestEmbed}
\imath &: \GDual \rightarrow \BDual \\
\imath &: e_i \mapsto \bullet_i
\end{split}
\end{align}
which embeds $\GDual$ into a \emph{strict} subalgebra of $\BDual$. 

Specifically, we can see that $\imath$ is injective by considering the space $\BB^*_\ell \subset \BB^*$ of linear trees, i.e., trees of the form $[\ldots[\bullet_{i_1}]_{\bullet_{i_2}}]\ldots]_{\bullet_{i_k}}$. Then there is a natural projection $\pi_\ell : \BDual \rightarrow \BB^*_\ell$, and one can readily see that $\pi_\ell \circ \imath$ is a vector space isomorphism (this is the same isomorphism described in Remark~2.7 of~\cite{HairerKelly15}). To see further that the image of $T((\R^{1+d}))$ under $\imath$ is not all of $\BDual$, it suffices to observe that the linear tree $[\bullet_i]_{\bullet_j}$ is not in the algebra generated by $\{\bullet_i\}_{i=1}^{1+d}$.

\begin{remark}
The embedding $\imath$ arises naturally in the context of branched rough paths as this is essentially the embedding used in~\cite{HairerKelly15} to realise geometric rough paths as branched rough paths (though note $\imath$ in~\cite{HairerKelly15} denotes $\pi_\ell \circ \imath$ in our notation).
\end{remark}

\begin{remark}\label{remark:freeAlgebra}
While the above argument shows that $(\BB, [\cdot,\cdot])$ is clearly not isomorphic to $\LL(\R^{1+d})$ as a Lie algebra, it is a curious and non-trivial fact that $(\BB, [\cdot,\cdot])$ \emph{is} isomorphic to a free Lie algebra generated by another subspace of $\BB$. Correspondingly, $(\HH,\star)$, being isomorphic to the universal enveloping algebra of $\BB$, is isomorphic to a tensor algebra (see~\cite{Foissy02} Section~8, or~\cite{Chapoton10}).
This was used in~\cite{BC17} to show that the space of branched rough paths is canonically isomorphic to a space geometric rough paths over an enlarged vector space.
\end{remark}

It follows form the above discussion that given a map $M : \R^{1+d} \rightarrow \BDual$, even one whose range is in $\BB^*$, there is in general no canonical choice of how to extend $M$ to elements outside $\imath(\GDual)$ if we only demand that the extension $M : \BDual \rightarrow \BDual$ is an algebra morphsim (moreover, without calling on Remark~\ref{remark:freeAlgebra}, it is {\it a priori} not even clear that such an extension always exists).

\begin{example}
Consider the case of a single label $0$ (i.e. $d=0$), and the map $M : \{\bullet_0, [\bullet_0]_{\bullet_0}\} \rightarrow \BB^*$ given by
\begin{align*}
M &: \bullet_0 \mapsto \bullet_0 \\
M &: [\bullet_0]_{\bullet_0} \mapsto \bullet_0.
\end{align*}
Since
\[
\bullet_0 \star \bullet_0 = [\bullet_0]_{\bullet_0} + 2 \bullet_0 \bullet_0,
\]
we may extend $M$ to an algebra morphism on the truncated space $\HH^2 \rightarrow \HH^2$ by setting
\[
M(\bullet_0\bullet_0) = \frac{1}{2}\left( [\bullet_0]_{\bullet_0} + 2\bullet_0\bullet_0 - \bullet_0 \right).
\]
This example shows that, on the level of the truncated algebras, there is not a unique algebra morphism above the identity map $\id : \bullet_0 \mapsto \bullet_0$.

Of course it is not clear from the above that the identity map $\id : \bullet_0 \mapsto \bullet_0$ can extend in a non-trivial way to an algebra morphism on all of $\BDual \mapsto \BDual$, but such extensions will always exist due to Remark~\ref{remark:freeAlgebra}.
\end{example}

In what follows, we address this non-uniqueness issue by demanding a finer structure on the extension of $M$, namely that $M : \BB^* \rightarrow \BB^*$ is a \emph{pre-Lie algebra} morphism.
The notion of a pre-Lie algebra will be recalled in the following subsection, and the significance of preserving the pre-Lie product on $\BB^*$ is first seen when establishing a dual characterization of $M$ (Proposition \ref{prop:adjointMv}), and then again in Section~\ref{subsec:effectsRDEs} when studying the impact on (rough) differential equations. For now we simply state that this is a natural condition to demand given the role of pre-Lie algebras in control theory and Butcher series~\cite{Calaque11,Manchon11}. 

\subsubsection{The free pre-Lie algebra over \texorpdfstring{$\R^{1+d}$}{R\^{}(1+d)}} \label{subsubsec:PreLie}

\begin{definition}
A (left) \emph{pre-Lie algebra} is a vector space $V$ with a bilinear map $\triangleright : V\times V \rightarrow V$, called the \emph{pre-Lie product}, such that
\[
(x \triangleright y)\triangleright z - x\triangleright(y\triangleright z) = (y \triangleright x)\triangleright z - y\triangleright(x \triangleright z), \; \; \text{for all } x,y,z \in V.
\]
That is, the associator $(x,y,z) := (x \triangleright y)\triangleright z - x\triangleright(y\triangleright z)$ is invariant under exchanging $x$ and $y$.
\end{definition}

One can readily check that every pre-Lie algebra $(V, \triangleright)$ induced a Lie algebra $(V, [\cdot,\cdot])$ consisting of the same vector space $V$ with bracket $[x,y] := x\triangleright y - y\triangleright x$.

\begin{example}\label{example:vectfieldsPreLie}
A basic example of a pre-Lie algebra is the space of smooth vector fields on $\R^e$ with the product $(f_i \partial_i) \triangleright (f_j\partial_j) := (f_i \partial_i f_j) \partial_j$. The induced bracket is the usual Lie bracket of vector fields.
\end{example}

The space of trees $\BB$ can be equipped with a (non-associative) pre-Lie product $\curvearrowright : \BB \times \BB \rightarrow \BB$ defined by
\[
\tau_1 \curvearrowright \tau_2 = \sum_{\tau} n(\tau_1,\tau_2,\tau)\tau,
\]
where the sum is over all trees $\tau \in \BB$ and $n(\tau_1,\tau_2,\tau)$ is the number of \emph{single} admissible cuts of $\tau$ for which the branch is $\tau_1$ and the trunk is $\tau_2$.
Equivalently, $\curvearrowright$ is given in terms of $\star$ by
 \[
\tau_1 \curvearrowright \tau_2 = \pi_{\BB}(\tau_1 \star \tau_2),
\]
where $\pi_\BB : \BPrimal \rightarrow \BB$ is the projection onto $\BB$.

It holds that $(\BB,\curvearrowright)$ indeed defines a Lie algebra for which
\[
[\tau_1,\tau_2] := \tau_1 \curvearrowright \tau_2 - \tau_2 \curvearrowright \tau_1 = \tau_1 \star \tau_2 - \tau_2 \star \tau_1, 
\]
i.e., the Lie algebra structures on $\BB$ induced by $\star$ and $\curvearrowright$ coincide. Moreover since $\curvearrowright$ respects the grading of $\BB$, we can naturally extend $\curvearrowright$ to a bilinear map on the space of series, so that $(\BB^*, \curvearrowright)$ is also a pre-Lie algebra.

We now recall the following universal property of $(\BB, \curvearrowright)$ first established by Chapoton and Livernet~\cite{Chapoton01} Corollary~1.10 (see also~\cite{Dzhumadildaev02} Theorem~6.3).

\begin{theorem}\label{thm:freePreLie}
The space $(\BB, \curvearrowright)$ is the free pre-Lie algebra over $\R^{1+d}$.
\end{theorem}

An equivalent formulation of Theorem~\ref{thm:freePreLie} is that for any pre-Lie algebra $(V, \triangleright)$ and linear map $M : \R^{1+d} \rightarrow V$, there exists a unique extension of $M$ to a pre-Lie algebra morphism $M : (\BB,\curvearrowright) \rightarrow (V, \triangleright)$.

\subsubsection{Construction of the translation map}\label{subsubsec:TransBranched}

An immediate consequence of Theorem~\ref{thm:freePreLie} is the following.

\begin{theorem}\label{thm:branchedAlgMorph}
Every linear map $M : \R^{1+d} \rightarrow \BB^*$ extends to a unique continuous algebra morphism $M : \BDual \rightarrow \BDual$ whose restriction to $\BB^*$ is a pre-Lie algebra morphism from $\BB^*$ to itself.
\end{theorem}

\begin{proof}
By Theorem~\ref{thm:freePreLie}, $M$ extends uniquely to a pre-Lie algebra morphism $M : \BB \rightarrow \BB^*$.
Recall also that, by the Milnor-Moore theorem, $(\HH,\star)$ is isomorphic to the universal enveloping algebra of $(\BB, [\cdot,\cdot])$.
It follows that $M$ extends further to a unique algebra morphism $M : (\HH,\star) \rightarrow (\BDual, \star)$. Finally, since $M$ necessarily does not decrease the degree of every element $x \in \HH$, we obtain a unique continuous extension $M : \BDual \rightarrow \BDual$ for which the restriction $M : \BB^* \rightarrow \BB^*$ is a pre-Lie algebra morphism as desired.
\end{proof}

We can finally define a natural translation map $M_v : \BDual \rightarrow \BDual$ analogous to $T_v$.

\begin{definition}
For $v = (v_0,\ldots, v_d) \subset \BB^*$, define $M_v : \BDual \rightarrow \BDual$ as the unique continuous algebra morphism obtained in Theorem~\ref{thm:branchedAlgMorph} from the linear map
\begin{align*}
M_v &: \R^{1+d} \rightarrow \BB^* \\
M_v &: \bullet_i \mapsto \bullet_i + v_i,  \; \; \text{for all } i \in \{0,\ldots, d\}.
\end{align*}
\end{definition}

\begin{example}
Let us illustrate how the construction works in the case of two nodes with a single label $0$. Since $M_{v}$ is constructed as pre-Lie algebra morphism, we compute
\[
M_{v}\left( \left[ \bullet_0 \right] _{\bullet_0} \right) = M_{v}\left( \bullet_0 \curvearrowright \bullet_0 \right) = M_{v}\left( \bullet_0 \right) \curvearrowright M_{v}\left( \bullet_0 \right) =\left( \bullet_0 +v_0\right) \curvearrowright \left( \bullet_0 +v_0\right).
\]
Since $M_{v}$ is in addition an algebra morphism w.r.t. $\star$ we have
\[
\left( \bullet_0 +v_0\right) \star \left( \bullet_0 +v_0\right) =\left( M_{v}\bullet_0 \right) \star \left( M_{v}\bullet_0 \right) = M_{v}\left( \bullet_0 \star \bullet_0 \right) = M_{v}\left( 2\bullet_0 \bullet_0 +\left[ \bullet_0 \right]_{\bullet_0}\right) 
\]
from which we can uniquely determine $M_{v}\left( \bullet_0 \bullet_0 \right) $.
\end{example}

As in the previous section, we shall often be concerned with the case that $v_i = 0$ for $i = 1,\ldots, d$ and $v_0$ takes a special form. We again make precise whenever such a condition is in place by writing, for example, $v = v_0 \in \BB^N(\R^{d})$.

We observe the following immediate properties of $M_v$, analogous to those of $T_v$:

\begin{itemize}
\item Since $M_v$ is an algebra morphism which preserves the Lie algebra $\BB^*$, it holds that $M_v$ maps $\GG$ into $\GG$;

\item $M_v \circ M_u = M_{v + M_v(u)}$, where we write $M_v(u) = (M_v(u_0),\ldots, M_v(u_d))$. In particular, $M_{v+u} = M_{v} \circ M_u$ for all $v = v_0,u = u_0 \in \BB^*(\R^d)$;

\item For every integer $N \geq 0$, $M_v$ induces a well-defined algebra morphism $M_v^N : \BPrimal^N \rightarrow \BPrimal^N$, which maps $\GG^N$ into $\GG^N$;

\item Recall the embedding $\imath : \GDual \rightarrow \BDual$ from~\eqref{eq:tensorForestEmbed}. Then for all $v = (v_0,\ldots, v_d) \subset \LL((\R^{1+d}))$, it holds that $M_{\imath(v)} \circ \imath = \imath \circ T_v$ (as both are continuous algebra morphisms from $\GDual$ to $\BDual$ which agree on $e_0,\ldots, e_d$).
\end{itemize}

As in the remark before Lemma~\ref{lem:HopfMorph}, note that ${\BDual}^{\otimes2}$ embeds densely into $\BPrimal^{\overline\otimes2}$, and thus $M_v\otimes M_v$ extends uniquely to a continuous algebra morphism $\BPrimal^{\overline\otimes2} \to \BPrimal^{\overline\otimes2}$.

\begin{lemma}\label{lem:HopfMorphBranched}
The map $M_v : \BDual \rightarrow \BDual$ satisfies $(M_v\otimes M_v)\Delta_\poly = \Delta_\poly M_v$ and commutes with the antipode $\alpha$.
\end{lemma}

\begin{remark}
We note that in the following proof, we only use the fact that $M_v$ is a continuous algebra morphism from $\BDual$ to itself which preserves the space of primitive elements $\BB^*$, and so do not directly use the fact that $M_v$ preserves the pre-Lie product of $\BB^*$.
\end{remark}

\begin{proof}
To show that the maps $(M_v\otimes M_v) \Delta_\poly$ and $\Delta_\poly M_v$ agree, by continuity it suffices to show they agree on $\HH$. In turn, their restrictions to $\HH$ are algebra morphisms on $(\HH,\star)$, and, since $(\HH,\star)$ is the universal enveloping algebra of its space of primitive elements $\BB$ by the Milnor-Moore theorem, it suffices to show that
\[
(M_v\otimes M_v)\Delta_\poly \tau = \Delta_\poly M_v \tau, \; \; \text{for all } \tau \in \BB.
\]
But this is immediate since $M_v$ maps $\BB^*$ into itself and $M_v(1) = 1$.
It remains to show that $M_v$ commutes with the antipode, which follows from the same argument as in the proof of Lemma~\ref{lem:HopfMorph}.
\end{proof}

\subsection{Dual action on the Connes--Kreimer Hopf algebra \texorpdfstring{$\BPrimal$}{H}}  \label{subsec:DualMapBranched}

As in Section~\ref{subsec:DualMapGeom}, we now wish to describe the dual map $M_v^* : \BPrimal \rightarrow \BPrimal$ for which
\[
\scal{M_v x, y} = \scal{x, M_v^* y}, \; \; \text{for all } x \in \BDual, \; \; y \in \BPrimal.
\]
For simplicity, we again consider in detail only the special case $v_i = 0$ for $i=1,\ldots, d$ (but see Remark~\ref{remark:generalBranched} for a description of the general case).

Let $\A$ denote the unital free commutative algebra generated by the trees $\tau \in \BB$. We let $\mathbf{1}$ and $\cdot$ denote the unit element and product of $\A$ respectively. The algebra $\A$ plays here the same role as the algebra $\SSS$ in Section~\ref{subsec:DualMapGeom}.

\begin{remark}
Although the algebras $(\A,\cdot)$ and $(\BPrimal,\poly)$ are isomorphic, they should be thought of as separate spaces and thus we make a clear distinction between the two.
\end{remark}

For a tree $\tau \in \BB$, we let $D(\tau)$ denote the set of all elements
\[
\tau_1\cdot\ldots \cdot \tau_k \otimes \tilde \tau \in \A \otimes \BB
\]
where $\tau_1,\ldots,\tau_k$ is formed from all disjoint collections of non-empty subtrees of $\tau$ (including subtrees consisting of a single node), and $\tilde \tau$ is the tree obtained by contracting every subtree $\tau_i$ to a single node which is then labelled by $0$ (note that $\mathbf{1} \otimes \tau$, corresponding to $k=0$, is also in $D(\tau)$).

Consider the linear map $\extract : \BPrimal \rightarrow \A \otimes \BPrimal$ defined for all trees $\tau \in \BB$ by
\[
\extract \tau = \sum_{\tau_1\cdot\ldots\cdot\tau_k \otimes \tilde \tau \in D(\tau)} \tau_1\cdot\ldots\cdot\tau_k \otimes \tilde \tau,
\]
and then extended multiplicatively to all of $\BPrimal$, where we canonically treat $\A\otimes\BPrimal$ as an algebra with multiplication $\MM_{\A\otimes\BPrimal}(\tau_1\otimes\hat{\tau}_1\otimes\tau_2\otimes\hat{\tau}_2):=(\tau_1\cdot\tau_2)\otimes(\hat{\tau}_1\poly\hat{\tau}_2)$ for $\tau_1,\tau_2\in\A$, $\hat{\tau}_1,\hat{\tau}_2\in\BPrimal$.

For example,
\begin{equs}\label{eq:deltaExample}
\begin{aligned}
\extract  \tikzexternaldisable  \begin{tikzpicture}[scale=0.2,baseline=0.1cm]
        \node at (0,0)  [dot,label= {[label distance=-0.2em]below: \scriptsize  $ i $} ] (root) {};
         \node at (1,2)  [dot,label={[label distance=-0.2em]above: \scriptsize  $ k $}] (right) {};
         \node at (-1,2)  [dot,label={[label distance=-0.2em]above: \scriptsize  $ j $} ] (left) {};
            \draw[kernel1] (right) to
     node [sloped,below] {\small }     (root); \draw[kernel1] (left) to
     node [sloped,below] {\small }     (root);
     \end{tikzpicture} \tikzexternaldisable  =&
     \mathbf{1} \otimes  \tikzexternaldisable  \begin{tikzpicture}[scale=0.2,baseline=0.1cm]
        \node at (0,0)  [dot,label= {[label distance=-0.2em]below: \scriptsize  $ i $} ] (root) {};
         \node at (1,2)  [dot,label={[label distance=-0.2em]above: \scriptsize  $ k $}] (right) {};
         \node at (-1,2)  [dot,label={[label distance=-0.2em]above: \scriptsize  $ j $} ] (left) {};
            \draw[kernel1] (right) to
     node [sloped,below] {\small }     (root); \draw[kernel1] (left) to
     node [sloped,below] {\small }     (root);
     \end{tikzpicture} \tikzexternaldisable
     +  \tikzexternaldisable  \begin{tikzpicture}[scale=0.2,baseline=-0.1cm]
        \node at (0,0)  [dot,label= {[label distance=-0.2em]below: \scriptsize  $ i $} ] (root) {};
     \end{tikzpicture} \tikzexternaldisable 
     \otimes  \tikzexternaldisable  \begin{tikzpicture}[scale=0.2,baseline=0.1cm]
        \node at (0,0)  [dot,label= {[label distance=-0.2em]below: \scriptsize  $ 0 $} ] (root) {};
         \node at (1,2)  [dot,label={[label distance=-0.2em]above: \scriptsize  $ k $}] (right) {};
         \node at (-1,2)  [dot,label={[label distance=-0.2em]above: \scriptsize  $ j $} ] (left) {};
            \draw[kernel1] (right) to
     node [sloped,below] {\small }     (root); \draw[kernel1] (left) to
     node [sloped,below] {\small }     (root);
     \end{tikzpicture} 
      +  \tikzexternaldisable  \begin{tikzpicture}[scale=0.2,baseline=-0.1cm]
        \node at (0,0)  [dot,label= {[label distance=-0.2em]below: \scriptsize  $ j $} ] (root) {};
     \end{tikzpicture} \tikzexternaldisable 
     \otimes  \tikzexternaldisable  \begin{tikzpicture}[scale=0.2,baseline=0.1cm]
        \node at (0,0)  [dot,label= {[label distance=-0.2em]below: \scriptsize  $ i $} ] (root) {};
         \node at (1,2)  [dot,label={[label distance=-0.2em]above: \scriptsize  $ k $}] (right) {};
         \node at (-1,2)  [dot,label={[label distance=-0.2em]above: \scriptsize  $ 0 $} ] (left) {};
            \draw[kernel1] (right) to
     node [sloped,below] {\small }     (root); \draw[kernel1] (left) to
     node [sloped,below] {\small }     (root);
     \end{tikzpicture}  \tikzexternaldisable
     + \tikzexternaldisable \begin{tikzpicture}[scale=0.2,baseline=-0.1cm]
        \node at (0,0)  [dot,label= {[label distance=-0.2em]below: \scriptsize  $ k $} ] (root) {};
     \end{tikzpicture} \tikzexternaldisable 
     \otimes  \tikzexternaldisable  \begin{tikzpicture}[scale=0.2,baseline=0.1cm]
        \node at (0,0)  [dot,label= {[label distance=-0.2em]below: \scriptsize  $ i $} ] (root) {};
         \node at (1,2)  [dot,label={[label distance=-0.2em]above: \scriptsize  $ 0 $}] (right) {};
         \node at (-1,2)  [dot,label={[label distance=-0.2em]above: \scriptsize  $ j $} ] (left) {};
            \draw[kernel1] (right) to
     node [sloped,below] {\small }     (root); \draw[kernel1] (left) to
     node [sloped,below] {\small }     (root);
     \end{tikzpicture}  \tikzexternaldisable
     \\ &+ \tikzexternaldisable \begin{tikzpicture}[scale=0.2,baseline=0.1cm]
        \node at (0,0)  [dot,label= {[label distance=-0.2em]below: \scriptsize  $ i $} ] (root) {};
         \node at (0,2)  [dot,label={[label distance=-0.2em]above: \scriptsize  $ k$}] (right) {};
            \draw[kernel1] (right) to
     node [sloped,below] {\small }     (root);
     \end{tikzpicture} \tikzexternaldisable \otimes \tikzexternaldisable \begin{tikzpicture}[scale=0.2,baseline=0.1cm]
        \node at (0,0)  [dot,label= {[label distance=-0.2em]below: \scriptsize  $ 0 $} ] (root) {};
         \node at (0,2)  [dot,label={[label distance=-0.2em]above: \scriptsize  $ j$}] (right) {};
            \draw[kernel1] (right) to
     node [sloped,below] {\small }     (root);
     \end{tikzpicture} \tikzexternaldisable   
     + \tikzexternaldisable \begin{tikzpicture}[scale=0.2,baseline=0.1cm]
        \node at (0,0)  [dot,label= {[label distance=-0.2em]below: \scriptsize  $ i $} ] (root) {};
         \node at (0,2)  [dot,label={[label distance=-0.2em]above: \scriptsize  $ j$}] (right) {};
            \draw[kernel1] (right) to
     node [sloped,below] {\small }     (root);
     \end{tikzpicture} \tikzexternaldisable \otimes \tikzexternaldisable \begin{tikzpicture}[scale=0.2,baseline=0.1cm]
        \node at (0,0)  [dot,label= {[label distance=-0.2em]below: \scriptsize  $ 0 $} ] (root) {};
         \node at (0,2)  [dot,label={[label distance=-0.2em]above: \scriptsize  $ k$}] (right) {};
            \draw[kernel1] (right) to
     node [sloped,below] {\small }     (root);
     \end{tikzpicture} \tikzexternaldisable
  + \tikzexternaldisable  \begin{tikzpicture}[scale=0.2,baseline=0.1cm]
        \node at (0,0)  [dot,label= {[label distance=-0.2em]below: \scriptsize  $ i $} ] (root) {};
         \node at (1,2)  [dot,label={[label distance=-0.2em]above: \scriptsize  $ k $}] (right) {};
         \node at (-1,2)  [dot,label={[label distance=-0.2em]above: \scriptsize  $ j $} ] (left) {};
            \draw[kernel1] (right) to
     node [sloped,below] {\small }     (root); \draw[kernel1] (left) to
     node [sloped,below] {\small }     (root);
     \end{tikzpicture} \tikzexternaldisable \otimes \begin{tikzpicture}[scale=0.2,baseline=-0.1cm]
        \node at (0,0)  [dot,label= {[label distance=-0.2em]below: \scriptsize  $ 0 $} ] (root) {};
     \end{tikzpicture} \tikzexternaldisable 
     \\ &+  \tikzexternaldisable  \begin{tikzpicture}[scale=0.2,baseline=-0.1cm]
        \node at (0,0)  [dot,label= {[label distance=-0.2em]below: \scriptsize  $ i $} ] (root) {};
     \end{tikzpicture} \tikzexternaldisable \cdot 
     \tikzexternaldisable  \begin{tikzpicture}[scale=0.2,baseline=-0.1cm]
        \node at (0,0)  [dot,label= {[label distance=-0.2em]below: \scriptsize  $ j $} ] (root) {};
     \end{tikzpicture} \tikzexternaldisable 
     \otimes  \tikzexternaldisable  \begin{tikzpicture}[scale=0.2,baseline=0.1cm]
        \node at (0,0)  [dot,label= {[label distance=-0.2em]below: \scriptsize  $ 0 $} ] (root) {};
         \node at (1,2)  [dot,label={[label distance=-0.2em]above: \scriptsize  $ k $}] (right) {};
         \node at (-1,2)  [dot,label={[label distance=-0.2em]above: \scriptsize  $ 0 $} ] (left) {};
            \draw[kernel1] (right) to
     node [sloped,below] {\small }     (root); \draw[kernel1] (left) to
     node [sloped,below] {\small }     (root);
     \end{tikzpicture} \tikzexternaldisable
     +  \tikzexternaldisable  \begin{tikzpicture}[scale=0.2,baseline=-0.1cm]
        \node at (0,0)  [dot,label= {[label distance=-0.2em]below: \scriptsize  $ i $} ] (root) {};
     \end{tikzpicture} \tikzexternaldisable 
   \cdot  \tikzexternaldisable  \begin{tikzpicture}[scale=0.2,baseline=-0.1cm]
        \node at (0,0)  [dot,label= {[label distance=-0.2em]below: \scriptsize  $ k $} ] (root) {};
     \end{tikzpicture} \tikzexternaldisable 
     \otimes  \tikzexternaldisable  \begin{tikzpicture}[scale=0.2,baseline=0.1cm]
        \node at (0,0)  [dot,label= {[label distance=-0.2em]below: \scriptsize  $ 0 $} ] (root) {};
         \node at (1,2)  [dot,label={[label distance=-0.2em]above: \scriptsize  $ 0 $}] (right) {};
         \node at (-1,2)  [dot,label={[label distance=-0.2em]above: \scriptsize  $ j $} ] (left) {};
            \draw[kernel1] (right) to
     node [sloped,below] {\small }     (root); \draw[kernel1] (left) to
     node [sloped,below] {\small }     (root);
     \end{tikzpicture} \tikzexternaldisable
     +  \tikzexternaldisable  \begin{tikzpicture}[scale=0.2,baseline=-0.1cm]
        \node at (0,0)  [dot,label= {[label distance=-0.2em]below: \scriptsize  $ j $} ] (root) {};
     \end{tikzpicture} \tikzexternaldisable 
  \cdot   \tikzexternaldisable  \begin{tikzpicture}[scale=0.2,baseline=-0.1cm]
        \node at (0,0)  [dot,label= {[label distance=-0.2em]below: \scriptsize  $ k $} ] (root) {};
     \end{tikzpicture} \tikzexternaldisable 
     \otimes  \tikzexternaldisable  \begin{tikzpicture}[scale=0.2,baseline=0.1cm]
        \node at (0,0)  [dot,label= {[label distance=-0.2em]below: \scriptsize  $ i $} ] (root) {};
         \node at (1,2)  [dot,label={[label distance=-0.2em]above: \scriptsize  $ 0 $}] (right) {};
         \node at (-1,2)  [dot,label={[label distance=-0.2em]above: \scriptsize  $ 0 $} ] (left) {};
            \draw[kernel1] (right) to
     node [sloped,below] {\small }     (root); \draw[kernel1] (left) to
     node [sloped,below] {\small }     (root);
     \end{tikzpicture} \tikzexternaldisable
     \\ & + \tikzexternaldisable  \begin{tikzpicture}[scale=0.2,baseline=-0.1cm]
        \node at (0,0)  [dot,label= {[label distance=-0.2em]below: \scriptsize  $ j $} ] (root) {};
     \end{tikzpicture} \tikzexternaldisable 
  \cdot   \tikzexternaldisable \begin{tikzpicture}[scale=0.2,baseline=0.1cm]
        \node at (0,0)  [dot,label= {[label distance=-0.2em]below: \scriptsize  $ i $} ] (root) {};
         \node at (0,2)  [dot,label={[label distance=-0.2em]above: \scriptsize  $ k$}] (right) {};
            \draw[kernel1] (right) to
     node [sloped,below] {\small }     (root);
     \end{tikzpicture} \tikzexternaldisable
     \otimes  \tikzexternaldisable \begin{tikzpicture}[scale=0.2,baseline=0.1cm]
        \node at (0,0)  [dot,label= {[label distance=-0.2em]below: \scriptsize  $ 0 $} ] (root) {};
         \node at (0,2)  [dot,label={[label distance=-0.2em]above: \scriptsize  $ 0$}] (right) {};
            \draw[kernel1] (right) to
     node [sloped,below] {\small }     (root);
     \end{tikzpicture} \tikzexternaldisable
     + \tikzexternaldisable  \begin{tikzpicture}[scale=0.2,baseline=-0.1cm]
        \node at (0,0)  [dot,label= {[label distance=-0.2em]below: \scriptsize  $ k $} ] (root) {};
     \end{tikzpicture} \tikzexternaldisable 
  \cdot   \tikzexternaldisable \begin{tikzpicture}[scale=0.2,baseline=0.1cm]
        \node at (0,0)  [dot,label= {[label distance=-0.2em]below: \scriptsize  $ i $} ] (root) {};
         \node at (0,2)  [dot,label={[label distance=-0.2em]above: \scriptsize  $ j$}] (right) {};
            \draw[kernel1] (right) to
     node [sloped,below] {\small }     (root);
     \end{tikzpicture} \tikzexternaldisable
     \otimes  \tikzexternaldisable \begin{tikzpicture}[scale=0.2,baseline=0.1cm]
        \node at (0,0)  [dot,label= {[label distance=-0.2em]below: \scriptsize  $ 0 $} ] (root) {};
         \node at (0,2)  [dot,label={[label distance=-0.2em]above: \scriptsize  $ 0$}] (right) {};
            \draw[kernel1] (right) to
     node [sloped,below] {\small }     (root);
     \end{tikzpicture} \tikzexternaldisable
+  \tikzexternaldisable  \begin{tikzpicture}[scale=0.2,baseline=-0.1cm]
        \node at (0,0)  [dot,label= {[label distance=-0.2em]below: \scriptsize  $ i $} ] (root) {};
     \end{tikzpicture} \tikzexternaldisable  \cdot
     \tikzexternaldisable  \begin{tikzpicture}[scale=0.2,baseline=-0.1cm]
        \node at (0,0)  [dot,label= {[label distance=-0.2em]below: \scriptsize  $ j $} ] (root) {};
     \end{tikzpicture} \tikzexternaldisable 
     \cdot
      \tikzexternaldisable  \begin{tikzpicture}[scale=0.2,baseline=-0.1cm]
        \node at (0,0)  [dot,label= {[label distance=-0.2em]below: \scriptsize  $ k $} ] (root) {};
     \end{tikzpicture} \tikzexternaldisable 
     \otimes  \tikzexternaldisable  \begin{tikzpicture}[scale=0.2,baseline=0.1cm]
        \node at (0,0)  [dot,label= {[label distance=-0.2em]below: \scriptsize  $ 0 $} ] (root) {};
         \node at (1,2)  [dot,label={[label distance=-0.2em]above: \scriptsize  $ 0 $}] (right) {};
         \node at (-1,2)  [dot,label={[label distance=-0.2em]above: \scriptsize  $ 0 $} ] (left) {};
            \draw[kernel1] (right) to
     node [sloped,below] {\small }     (root); \draw[kernel1] (left) to
     node [sloped,below] {\small }     (root);
     \end{tikzpicture} \tikzexternaldisable .
     \end{aligned}
     \end{equs}

\begin{proposition} \label{prop:adjointMv} 
Let $v = v_0 \in \BB^*$. The dual map $M_v^* : \BPrimal \rightarrow  \BPrimal$ is given by
\[
M_v^* \tau = (v \otimes \id)\circ \extract (\tau),
\]
where $v(\tau_1 \cdot \ldots \cdot \tau_k) := \scal{\tau_1,v}\ldots\scal{\tau_k, v}$ and $v(\mathbf{1}) := 1$.
\end{proposition}

For the proof of Proposition~\ref{prop:adjointMv}, we require the following combinatorial lemma. We note that similar ``cointeraction'' results appear for closely related algebraic structures in~\cite[Thm 8]{Calaque11}  and~\cite[Thm 5.37]{BHZ16}. We will particularly discuss in further detail the link with the work of~\cite{BHZ16} in Section~\ref{sec:renorm}.

\begin{lemma}\label{lem:adjointComm}
Let $\curvearrowright^*: \BB \rightarrow \BB \otimes \BB$ denote the adjoint of $\curvearrowright$. It holds that
\begin{equation}\label{eq:triangDelta}
(\id \otimes \curvearrowright^*)\extract = \MM_{1,3}(\extract \otimes \extract)\curvearrowright^*,
\end{equation}
where $\MM_{1,3} : \A \otimes \BB \otimes \A \otimes \BB \rightarrow \A \otimes \BB \otimes \BB$ is the linear map defined by ${\MM_{1,3}(\tau_1\otimes \tau_2\otimes\tau_3\otimes\tau_4)} = \tau_1\tau_3\otimes \tau_2\otimes\tau_4$.
\end{lemma}

\begin{proof}
Note that
\[
\curvearrowright^* \tau = \sum_{c} b_c \otimes \tau_c
\]
where the sum runs of all \emph{single} admissible cuts $c$ of $\tau$, and $b_c$ is the branch, $\tau_c$ the trunk of $c$.
Consider a single cut $c$ of $\tau$ across an edge $e$. Let $\tau^c$ denote the sum of the terms of $(\id \otimes \curvearrowright^*)\extract \tau$ obtained by contracting all collections 
of subtrees of $\tau$ which do not contain $e$, followed by a cut (on the second tensor) along the edge $e$ (which necessarily remains). One immediately sees that $\tau^c$ is 
equivalently given by first cutting along $e$, and then contracting along all collections of subtrees of $b_c$ and $\tau_c$, and then grouping the extracted subtrees together, i.e., $\tau^c 
= \MM_{1,3}(\extract \otimes \extract)(b_c\otimes \tau_c)$.
It finally remains to observe that summing over all single cuts $c$ gives~\eqref{eq:triangDelta}.
\end{proof}

\begin{proof}[Proof of Proposition~\ref{prop:adjointMv}]
Denote by
\[
\Phi = (v \otimes \id)\circ \extract : \BB \rightarrow \BB.
\]
By duality, it follows from Lemma~\ref{lem:HopfMorphBranched} that $M_v^*$ is a Hopf algebra morphism. In particular, it suffices to show that $\Phi\tau = M_v^* \tau$ for every tree $\tau \in \BB$.

To this end, observe that Lemma~\ref{lem:adjointComm} implies $\curvearrowright^* \Phi = (\Phi \otimes \Phi)\curvearrowright^*$, from which it follows that $\Phi^* : \BB^* \rightarrow \BB^*$ is a pre-Lie algebra morphism. Furthermore, for every tree $\tau \in \BB$
\begin{align*}
\text{for all } i \in\{1,\ldots,d\}, \; \; &\scal{\Phi^* \bullet_i, \tau} = \scal{\bullet_i, \Phi \tau} = \scal{\bullet_i, \tau} = \scal{M_v\bullet_i, \tau}; \\
&\scal{\Phi^* \bullet_0, \tau} = \scal{\bullet_0, \Phi \tau} = \scal{\bullet_0, \tau} + \scal{v,\tau} = \scal{M_v\bullet_0, \tau}.
\end{align*}
It follows that $\Phi^*$ is a pre-Lie algebra morphism on $(\BB^*,\curvearrowright)$ which agrees with $M_v$ on the set $\{\bullet_0,\ldots,\bullet_d\} \subset \BB^*$. Hence, by the universal property of $(\BB,\curvearrowright)$ (Theorem~\ref{thm:freePreLie}), $\Phi^*$ agrees with $M_v$ on all of $\BB^*$, which concludes the proof.
\end{proof} 

\begin{remark}\label{remark:generalBranched}
A similar result to Proposition~\ref{prop:adjointMv} holds for the general case $v = (v_0,\ldots, v_d)$. The definition of $\extract$ changes in the obvious way that in the second tensor, instead of replacing every subtree by the node $\bullet_0$, one instead replaces every combination of subtrees by all combinations of $\bullet_i$, $i \in \{0,\ldots, d\}$, while in the first tensor, one marks each extracted subtree $\tau_j$ with the corresponding label $i \in \{0, \ldots, d\}$ that replaced it, which gives $(\tau_j)_i$ (so the left tensor no longer belongs to $\A$ but instead to the free commutative algebra generated by $(\tau)_i$, for all trees $\tau \in \BB$ and labels $i \in \{0,\ldots, d\}$). Finally the term $\scal{\tau_1,v}\ldots\scal{ \tau_k, v}$ would then be replaced by $\scal{(\tau_1)_{i_1}, v_{i_1}}\ldots \scal{(\tau_k)_{i_k}, v_{i_k}}$.
\end{remark}

\section{Examples} \label{sec:Ex}

In the following examples, we assume that we are given a probability space $(\Omega, \mathcal{F},\mathbb{P})$ and a filtration $(\mathcal{F}_t)_{t\geq 0}$ satisfying the usual hypotheses and to which all mentioned stochastic processes are adapted.

\subsection{It\^{o}-Stratonovich conversion}\label{subsec:ItoStrat}

As an application of Proposition~\ref{prop:adjointMv}, we illustrate how to re-express iterated Stratonovich integrals (and products thereof) over some interval $\left[s,t \right] $ as It{\^o} integrals. Consider the $\R^{1+d}$-valued process $B_t = (B^0_t, B^1_t, \ldots, B^d_t)$, where $(B^1_t, \ldots, B^d_t)$ is an $\R^d$-valued Brownian motion with covariance $[B^i,B^j]_t = C^{i,j}t$, and $B^0_t \equiv t$ denotes the time component. Let $\B^{\Strat}$ denote the enhancement of $B_t$ to an $\alpha$-H{\"o}lder branched rough path, $\alpha \in (0,1/2)$, using Stratonovich iterated integrals. For example,  
\begin{eqnarray}
\left\langle \B^\Strat_{s,t},\tau \right\rangle  &=&\underset{
s<t_{1}<\dots <t_{m}<t}{\int \cdots \int }\circ dB_{t_{1}}^{i_{1}}\circ
\dots \circ dB_{t_{m}}^{i_{m}}   \label{equ:BStr}   \\  
\text{for the linear tree \ }\tau  &=& [\ldots [\bullet_{i_1}]_{\bullet_{i_2}} \ldots ]_{\bullet_{i_m}}, \; \; i_1,\ldots,i_m \in \{0,\ldots, d\},   \nonumber
\end{eqnarray}
and
\begin{eqnarray*}
\left\langle \B^\Strat_{s,t},\tau \right\rangle 
&=&\int_{s}^{t}B_{u}^{j}B_{u}^{k}\circ dB_{u}^{i} \\
\text{for }\tau  &=&[\bullet_j \bullet_k]_{\bullet_i}, \; \; i,j,k \in \{0,\ldots, d\}.
\end{eqnarray*}
Similarly, we define $\B^\Ito$ in exactly the same way using It\^{o} integrals.

For a tree $\tau \in \BB$, recall the definition of $D(\tau) \subset \A \otimes \BB$ from Section~\ref{subsec:DualMapBranched} (which was used to define $\delta$).
{ Consider the function $C : D(\tau) \to \R$ defined by
\[
C(\tau_1\cdot\ldots\cdot \tau_k \otimes \tilde \tau) =
\begin{cases}
1 &\textnormal{ if } \tau_1\cdot\ldots\cdot \tau_k \otimes \tilde \tau = \mathbf{1} \otimes \tau
\\
2^{-k} \prod_{n=1}^k C_n^{i_n,j_n} &\textnormal{ if } \tau_n = [\bullet_{i_n}]_{\bullet_{j_n}} \textnormal{ for all } n=1,\ldots, k
\\
0 &\textnormal{ otherwise}.
\end{cases}
\]
}

\begin{proposition}
For every tree $\tau \in \BB$ it holds that
\begin{equation}\label{eq:ItoStratConv}
\scal{\B^\Strat_{s,t},\tau} = {\sum_{\tau_1 \cdot \ldots \cdot \tau_k \otimes \tilde \tau \in D(\tau)} C(\tau_1\cdot\ldots\cdot \tau_k \otimes \tilde \tau) \scal{\B^\Ito_{s,t}, \tilde \tau}.}
\end{equation}
\end{proposition}

\begin{proof}
Consider the sum of linear trees {$v = v_0 = \frac{1}{2}\sum_{i,j=1}^d C^{i,j}[\bullet_{i}]_{\bullet_j} \in \BB^2(\R^d)$}.
One can readily verify that $\B^\Strat = M_v(\B^\Ito)$, understood in the pointwise sense $\B^\Strat_{s,t} = M_v(\B^\Ito_{s,t})$.
Indeed, both $\B^\Strat $ and $M_v(\B^\Ito)$ are a.s. ``full'' $\alpha$-H\"older rough paths, where this fact - in the case of $M_v(\B^\Ito)$ - either requires an (easy) check by hand, or an appeal to Theorem~\ref{thm:transRPs},~\ref{point:trans2}, below.
Since, by construction, both agree on the first two levels, and $\alpha \in (1/2,1/3)$, we see that $\B^\Strat $ and $M_v(\B^\Ito)$ must be equal, a.s., thanks to the uniqueness part of the extension theorem. 

It then follows by Proposition~\ref{prop:adjointMv} that
\[
\scal{\B^\Strat_{s,t}, \tau} = \scal{\B^\Ito_{s,t}, M_v^* \tau} = \sum_{\tau_1 \cdot \ldots \cdot \tau_k \otimes \tilde \tau \in D(\tau)} \scal{\B^\Ito_{s,t}, \scal{v,\tau_1}\ldots \scal{v,\tau_k} \tilde \tau}.
\]
Since $\scal{v,\mathbf{1}} = 1$, while {$\scal{v, \tau_n} = \frac12 C^{i,j}$ if $\tau_n = [\bullet_i]_{\bullet_j}$} and zero otherwise, we obtain precisely~\eqref{eq:ItoStratConv}.
\end{proof}

\begin{example}
{ Suppose $B$ is a standard Brownian motion, i.e., $C^{i,j} = \delta_{ij}$.
Consider the tree} $\tau =[\bullet_j \bullet_k]_{\bullet_i}$, so that
\[
\scal{\B^\Strat_{s,t},\tau} =\int_{s}^{t} B_{u}^{j}B_{u}^{k}\circ dB_{u}^{i}.
\]
{Recalling the explicit form of $\extract \tau$ in~\eqref{eq:deltaExample}, we see that if $i$ is distinct from both $j,k$, then only $\mathbf{1}\otimes \tau$ remains in $D(\tau)$ for which $C$ is non-zero,} and so (in trivial agreement with stochastic calculus)
\[
\scal{\B^\Strat_{s,t},\tau} = \scal{\B^\Ito_{s,t},\tau}.
\]
On the other hand,if $i=j\neq k$, an
additional term $[\bullet_i]_{\bullet_i} \otimes [\bullet_k]_{\bullet_0}$ { appears in $D(\tau)$ at which $C$ is $\frac12$,} and so
\begin{eqnarray*}
\scal{ \B^\Strat_{s,t},\tau} &=& \scal{\B^\Ito_{s,t},\tau} +\frac{1}{2} \underset{s<t_{1}<t_{2}<t}{\int \int }
dB_{t_{1}}^{k}dB_{t_{2}}^{0} \\
&=&\scal{\B^\Ito_{s,t},\tau}  +\frac{1}{2}
\int_{s}^{t} B_{u}^{k} du.
\end{eqnarray*}
The case $i=k\neq j$ is identical.
At last, in the case $i=j=k$, looking at $\extract \tau$ shows that
\begin{eqnarray*}
\scal{\B^\Strat,\tau }  &=& \scal{\B^\Ito,\tau } +\frac{1}{2}\underset{s<t_{1}<t_{2}<t}{\int \int } dB_{t_{1}}^{i}dB_{t_{2}}^{0}+\frac{1}{2}\underset{s<t_{1}<t_{2}<t}{\int \int 
}dB_{t_{1}}^{i}dB_{t_{2}}^{0} \\
&=& \scal{ \B^\Ito,\tau } +\int_{s}^{t} B_{u}^{i} du.
\end{eqnarray*}
\end{example}

\begin{remark}
When $\tau = [\ldots [\bullet_{i_1}]_{\bullet_{i_2}}\ldots ]_{\bullet_{i_m}}$ is a linear tree, this is in agreement with~\cite{BenArous89} Proposition 1. In fact, by considering general semi-martingales $X^1_t,\ldots, X^d_t$ and adding extra labels $\bullet_{i,j}$, $1 \leq i \leq j \leq d$ (thus increasing the underlying dimension from $d$ to $d + d(d+1)/2$) to encode the quadratic variants $[X_i,X_j]$, the above procedure (in the more general setting with elements $v_{ij} = [\bullet_i]_{\bullet_j} \in \BB^2(\R^d)$, see Remark~\ref{remark:generalBranched}), immediately provides an It{\^o}-Stratonovich conversion formula for general semi-martingales.
\end{remark}

\subsection{L\'evy rough paths} \label{sec:levy}

Note that the example in the previous section can be viewed as follows: $\B^\Ito$ and $\B^{\Strat}$ are both $\GG^2$-valued L{\'e}vy processes which are branched $p$-rough paths, $2 < p < 3$, and one can recover the signature of one from the other by a suitable (deterministic) translation map $M_v : \GG \rightarrow \GG$. We now consider a generalisation of this setting to arbitrary $\GG^N$-valued L{\'e}vy processes, which have already been studied in the context of rough paths in~\cite{FrizShekhar17, Chevyrev18}.

Let $\tau_1,\ldots, \tau_m$ be a basis for $\BB^N$ consisting of trees, which we identify with left-invariant vector fields on $\GG^N$, where we suppose for convenience that $\tau_1 = \bullet_0$.
Recall that $\GG^N$ is a homogenous group in the sense of~\cite{FollandStein82} (cf.~\cite{HairerKelly15} Remark~2.15).

Recall that to every (left) L{\'e}vy process $\Xbf$ in $\GG^N$ without jumps and with identity starting point (i.e., $\Xbf_0 = 1_{\GG^N}$  a.s.) there is an associated L{\'e}vy {tuple $(A,B)$,} where $B = \sum_{i=1}^m B^i \tau_i$ is an element of $\BB^N$ and $(A^{i,j})_{i,j=1}^m$ is a correlation matrix. Then the generator of $\Xbf$ is given for all $f \in C^2_0(\GG^N)$ by (see, e.g., \cite{Liao04})
\[
\lim_{t \rightarrow 0} t^{-1}\EEE{f(x\star\Xbf_t) - f(x)} = \sum_{i=1}^m B^i (\tau_i f) (x) + \frac{1}{2}\sum_{i,j=1}^m A^{i,j} (\tau_i\tau_j f)(x).
\]

\begin{lemma}\label{lem:LevyTriplet}
Let $M : \HH^N \rightarrow \HH^N$ be an algebra morphism which preserves $\GG^N$ and $\Xbf$ a L{\'e}vy process in $\GG^N$ with L{\'e}vy {tuple $(A,B)$.}

Then $M(\Xbf)$ is the (unique in law) $\GG^N$-valued (left) L{\'e}vy process with generator given for all $f \in C^2_0(\GG^N)$ by
\begin{equation}\label{eq:LevyGen}
\lim_{t \rightarrow 0} t^{-1}\EEE{f(x \star M\Xbf_t) - f(x)} = \sum_{i=1}^m B^i (M\tau_i f) (x) + \frac{1}{2}\sum_{i,j=1}^m A^{i,j} (M\tau_i M\tau_j f)(x).
\end{equation}
\end{lemma}

\begin{proof}
The fact that $M\Xbf$ is a L{\'e}vy process is immediate from the fact that $\Xbf$ is a L{\'e}vy process and that $M : \GG^N \rightarrow \GG^N$ is a (continuous) group morphism. It thus only remains to show~\eqref{eq:LevyGen}, where we may suppose without loss of generality that $x=1_{\GG^N}$. To this end, define $h = f\circ M$ and observe that
\[
\lim_{t \rightarrow 0} t^{-1}\EEE{f(M\Xbf_t) - f(1_{\GG^N})} = \sum_{i=1}^m B^i (\tau_i h)(1_{\GG^N}) + \frac{1}{2}\sum_{i,j=1}^m A^{i,j} (\tau_i\tau_j h)(1_{\GG^N})
\]
(note that in general $h$ might fail to decay at infinity and thus not be an element of $C^2_0(\GG^N)$, however the above limit is readily justified by taking suitable approximations). Using the fact that $(\tau h)(x) = \frac{d}{dt} h(x\star e^{t\tau})\mid_{t=0}$, one can easily verify that for all $\sigma,\tau \in \BB^N$ and $x \in \GG^N$
\begin{align*}
(\tau h)(x) &= (M\tau f)(Mx), \\
(\sigma \tau h)(x) &= ((M\sigma) (M \tau) f )(Mx),
\end{align*}
from which~\eqref{eq:LevyGen} follows.
\end{proof}

We now specialise to the case that $(A^{i,j})_{i,j=1}^m$ is a correlation matrix for which $A^{i,i} = 0$ whenever $\tau_i$ has more than $\floor{N/2}$ nodes, which is a necessary and sufficient condition for the sample paths of $\Xbf$ to a.s. have finite $p$-variation for all $N < p < N+1$~\cite{Chevyrev18}. Assume also that $A^{i,i} = 0$ whenever $\tau_i$ contains a node with label $0$, and that $B = \tau_1 = \bullet_0$, so that for all $f \in C^2_0(\GG^N)$
\[
\lim_{t \rightarrow 0} t^{-1}\EEE{f(x\star\Xbf_t) - f(x)} = (\tau_1 f) (x) + \frac{1}{2}\sum_{i,j=1}^m A^{i,j} (\tau_i\tau_j f)(x).
\]
The drift term $(\tau_1 f) (x)$ should be interpreted as the time component of the branched rough path $\Xbf$ (which also explains the zero-diffusion condition in  the direction of trees with a label $0$).

Any other $\GG^N$-valued L{\'e}vy process $\tilde \Xbf$ without jumps and the same correlation matrix $(A^{i,j})_{i,j=1}^m$ is also a branched $p$-rough, and its generator differs from that of $\Xbf$ only by a drift term. As a consequence of Lemma~\ref{lem:LevyTriplet}, we see that every such $\tilde \Xbf$ can be constructed by applying a (deterministic) translation map $M_v$ to $\Xbf$. In particular, the full signature of $\tilde \Xbf$ can be recovered from that of $\Xbf$, generalising the example from Section~\ref{subsec:ItoStrat}.

\begin{corollary}
Let $v = v_0 \in \BB^N$ and $M_v : \HH^N \rightarrow \HH^N$ the truncation of the translation map from Section~\ref{subsubsec:TransBranched}.

Then $M_v(\Xbf)$ is the (unique in law) $\GG^N$-valued (left) L{\'e}vy process with generator given for all $f \in C^2_0(\GG^N)$ by
\begin{equation*}
\lim_{t \rightarrow 0} t^{-1}\EEE{f(x \star M_v(\Xbf_t)) - f(x)} = (\bullet_0 + v) f (x) + \frac{1}{2}\sum_{i,j=1}^m A^{i,j} (\tau_i\tau_j f)(x).
\end{equation*}
\end{corollary}

\begin{remark}
The statement of the corollary likewise holds for every algebra morphism $M : \HH^N \rightarrow \HH^N$ satisfying $M\bullet_0 = \bullet_0 + v$ and $M\tau = \tau$ for all forests $\tau \in \HH^N$ without a label $0$, which is a manifestation of the final point of the upcoming Theorem~\ref{thm:transRPs}~\ref{point:trans2}.
\end{remark}

\subsection{Higher-order translation and renormalization in finite-dimensions}  \label{sec:introex}

In~\cite{BCF18}, from which we give an excerpt in this subsection, two examples are studied of families of random bounded variation paths $(X^\varepsilon)_{\varepsilon > 0}$ whose canonical lifts to geometric rough paths $(\mathbf{X}^\varepsilon)_{\varepsilon > 0}$ diverge as $\varepsilon \rightarrow 0$. In particular, ODEs driven by $X^\varepsilon$ in general also fail to converge. However, for suitably chosen $v^\varepsilon = v^\varepsilon_0 \in \LL^N(\R^{d})$, for which in general $\lim_{\varepsilon\rightarrow 0}|v^\varepsilon| = \infty$, one obtains convergence of the translated rough paths $T_{v^\varepsilon}\mathbf{X}^\varepsilon$. In particular, it follows from the upcoming Theorem~\ref{thm:transRDEs} that solutions to modified ODEs driven by $X^\varepsilon$, with terms generally diverging as $\varepsilon \rightarrow 0$, converge to well-defined limits. In this specific context, the translation maps $T_{v^\varepsilon}$ are precisely the renormalization maps occurring in regularity structures when applied to the setting of SDEs; we shall make this connection precise in Section~\ref{sec:renorm}.

\bigskip \textbf{Physical Brownian motion in a (large) magnetic field.} It
was shown in \cite{FGL15} that the motion of a charged Brownian particle, in the
zero mass limit, in a magnetic field which is kept constant while taking the limit, naturally leads to a perturbed second
level, of the form $\mathbb{\bar{B}}_{s,t}=\mathbb{B}_{s,t}^{\Strat}+v\left(
t-s\right) $ for some $0\neq v\in \mathfrak{so}\left( d\right)$, $v$ being proportional to the strength of the magnetic field. We now want to look at the
evolution of the system under the blow-up of the magnetic field.

Consider a physical Brownian motion in a magnetic field with dynamics given by
\[
m \ddot{x} = -A\dot{x} + B\dot{x} + \xi, \; \; x(t) \in \R^d, 
\]
where $A$ is a symmetric matrix with strictly positive spectrum (representing friction), $B$ is an anti-symmetric matrix (representing the Lorentz force due to a magnetic field), and $\xi$ is an $\R^d$-valued white noise in time. We shall consider the situation that $A$ is constant whereas $B$ is a function of the mass $m$.

We rewrite these dynamics as
\begin{align*}
dX_t &= \frac{1}{m} P_t dt, \; \; X_0 = 0, \\
dP_t &= -\frac{1}{m} M P_t dt + dW_t, \; \; P_0 = 0,
\end{align*}
where $M = A - B$, and we have chosen the starting point as zero simply for convenience. We furthermore introduce the parameter $\varepsilon^2 = m$ and write $X^\varepsilon_t, P^\varepsilon_t$, and $M^\varepsilon = A-B^\varepsilon$ to denote the dependence on $\varepsilon$.

We are interested in the convergence of the processes $P^\varepsilon$ and $M^\varepsilon X^\varepsilon$ in rough path topologies. {As before in Section \ref{sec:TgRP},} let $G^2(\R^d)$ and {
$\LL^2(\R^d)$} denote the step-$2$ free nilpotent Lie group and Lie algebra respectively. Let us also write {$\LL^2(\R^d) = \R^d \oplus \LL^{(2)}(\R^d)$} for the decomposition of {$\LL^2(\R^d)$} into the first and second levels, where we identify {$\LL^{(2)}(\R^d)$} with the space of anti-symmetric $d\times d$ matrices. 

For every $\varepsilon > 0$, define the matrix
\[
C^\varepsilon = \int_0^\infty e^{-M^\varepsilon s}e^{-(M^{\varepsilon})^* s}ds,
\]
and the element
\[
v^\varepsilon = -\frac{1}{2}(M^\varepsilon C^\varepsilon - C^\varepsilon(M^\varepsilon)^*) \in {\LL^{(2)}(\R^d)}.
\]

For $\alpha\in(1/3,1/2]$, due to the extension theorem, any $\alpha$-H\"older weakly geometric rough path $\mathbf{Z}:[0,T]^2\to G(\R^d)$ is fully characterized by the truncation $\pi_2\mathbf{Z}:[0,T]^2\to G^2(\R^d)$. Thus, for the purpose of this example, we represent any such rough path $\mathbf{Z}$ by the increments $Z_{s,t}$ of the underlying path and the second level $\Z_{s,t}$, i.e.
\begin{equation*}
 Z_{s,t}^i=\langle \mathbf{Z}_{s,t},e_i\rangle,\qquad\Z_{s,t}^{j,k}=\langle\mathbf{Z}_{s,t},e_{j,k}\rangle.
\end{equation*}

In this special case and for any $v=v_0\in\LL^{(2)}(\R^d)$, the translation map introduced in Definition\nobreakspace\ref{def:TensorTranslate} is given by

\begin{equation}\label{eq:TvDef}
T_v(Z_{s,t}, \Z_{s,t}) = (Z_{s,t}, \Z_{s,t} + (t-s)v).
\end{equation}

Consider the $G^2(\R^d)$-valued processes
\begin{align*}
(P^\varepsilon_{s,t}, \Pbb^\varepsilon_{s,t}) &= \left(P^\varepsilon_{s,t}, \int_s^t P^\varepsilon_{s,r} \tensor \circ dP^\varepsilon_r\right), \\
(Z^\varepsilon_{s,t}, \Z^\varepsilon_{s,t}) &= \left(M^\varepsilon X^\varepsilon_{s,t}, \int_s^t M^\varepsilon X_{s,r} \tensor d(M^\varepsilon X^\varepsilon)_r \right),
\end{align*}
and the canonical lift of the Brownian motion $W$
\[
(W_{s,t}, \W_{s,t}) = \left(W_{s,t}, \int_s^t W_{s,r} \tensor \circ dW_r\right),
\]
where the integrals in the definition of $\Pbb^\varepsilon_{s,t}$ and $\W_{s,t}$ are in the Stratonovich sense.

Contrary to~\cite{FGL15}, we allow blow-up of the magnetic field with rate $B^\varepsilon \lesssim \varepsilon^{-\kappa}$, $\kappa \in [0,1]$, as a method to model magnetic fields which are large (in a quantified way) in comparison to the (small) mass.
The paths $Z^\varepsilon$ then form approximations of Brownian motion, whose canonical rough path lifts $(Z^\varepsilon, \Z^\varepsilon)$ do not converge in rough path space (due to
divergence of the L\'{e}vy's area).
The following result establishes convergence of the ``renormalised'' paths $T_{v^\varepsilon}(P^\varepsilon_{s,t}, \Pbb^\varepsilon_{s,t})$ and $T_{v^\varepsilon}(Z^\varepsilon_{s,t}, \Z^\varepsilon_{s,t})$.

\begin{theorem}[\cite{BCF18} Theorem~1]\label{thm:magneticConv}
Suppose that
\begin{equation}\label{eq:MBound}
\lim_{\varepsilon \rightarrow 0} |M^\varepsilon|\varepsilon^\kappa = 0 \; \textnormal{ for some } \kappa \in [0,1].
\end{equation}
Then for any $\alpha \in [0,1/2-\kappa/4)$ and $q < \infty$, it holds that $T_{v^\varepsilon}(P^\varepsilon, \Pbb^\varepsilon) \rightarrow (0,0)$ and $T_{v^\varepsilon}(Z^\varepsilon, \Z^\varepsilon) \rightarrow (W,\W)$ in $L^q$ and $\alpha$-H{\"o}lder topology as $\varepsilon \rightarrow 0$. More precisely, as $\varepsilon \rightarrow 0$, in $L^q$
\[
\sup_{s,t \in [0,T]} \frac{|P^\varepsilon_{s,t}|}{|t-s|^\alpha} + \sup_{s,t \in [0,T]} \frac{|\Pbb^\varepsilon_{s,t} + (t-s)v^\varepsilon|}{|t-s|^{2\alpha}} \rightarrow 0.
\]
and
\[
\sup_{s,t \in [0,T]} \frac{|Z^\varepsilon_{s,t} - W_{s,t}|}{|t-s|^\alpha} + \sup_{s,t \in [0,T]} \frac{|\Z^\varepsilon_{s,t} + (t-s)v^\varepsilon - \W_{s,t}|}{|t-s|^{2\alpha}} \rightarrow 0.
\]
In particular, if $\kappa \in [0,\frac23)$, one can take $\alpha \in (\frac13,\frac12-\frac\kappa4)$ and convergence takes place in $\alpha$-H{\"o}lder rough path topology.
\end{theorem}

Lastly, we would like to point out that higher order renormalization can be expected in the presence of highly oscillatory fields, which also points to some natural connections with homogenization theory.

\medskip \textbf{Fractional delay / Hoff process }
Viewed as two-dimensional
rough paths, Brownian motion and its $\varepsilon $-delay, $t\mapsto
(B_{t},B_{t-\varepsilon })$, does not converge to $\left( B,B\right) $, with
- as one may expect - zero area. Instead, the quadratic variation of
Brownian motion leads to a rough path limit of the form $\left( B,B;A\right) 
$ with area of order one. It is then possible
to check that, replacing $B$ by a fractional Brownian motion with Hurst parameter $H<1/2$, the same
construction will yield exploding L\'{e}vy area as $\varepsilon \downarrow 0$.

The same phenomena is seen in lead-lag situations, popular in time series
analysis. As in the case of physical Brownian motion in a (large) magnetic
field, these divergences can be cured by applying suitable (second-level)
translation / renormalization operators, as we shall now see; for details on the (non-divergent) Brownian / semi-martingale case, see e.g.~\cite[Ch.13]{FrizVictoir10} and~\cite{Flint16}.

Consider a path $X : [0,1] \mapsto \R^d$. Let $n \geq 1$ be an integer and write for brevity $X^n_i = X_{i/n}$. Consider the piecewise linear path $\tilde X^n : [0,1] \mapsto \R^{2d}$ defined by
\begin{align*}
\tilde X^n_{2i/2n} &= (X^n_i, X^n_i), \\
\tilde X^n_{(2i+1)/2n} &= (X^n_i, X^n_{i+1}),
\end{align*}
and linear on the intervals $\left[\frac{2i}{2n}, \frac{2i+1}{2n}\right]$ and $\left[\frac{2i+1}{2n}, \frac{2i+2}{2n}\right]$ for all $i = 0,\ldots, n-1$. Note that this is a variant of the Hoff process considered in~\cite{Flint16}.

Denote by {$\tilde \Xbf^n_{s,t} = \pi_2\exp_\tensor(\tilde X^n_{s,t} + \Abb^n_{s,t})$} the level-$2$ lift of $\tilde X^n$, where $\Abb^n_{s,t}$ is the $(2d) \times (2d)$ anti-symmetric L{\'e}vy area matrix given by
\[
\Abb^n_{s,t}
= \frac{1}{2}\left(\int_s^t \tilde X^n_{s,r} \tensor d\tilde X^n_r - \int_s^t \tilde X^n_{s,r} \tensor d\tilde X^n_r\right).
\]

Let $H \in (0,1)$ and consider a fractional Brownian motion $B^H$ with covariance $R(s,t) = \frac{1}{2}(t^{2H} + s^{2H} - |t-s|^{2H})$. Let $X : [0,1] \mapsto \R^d$ be $d$ independent copies of $B^H$.

Recall the definition of $T_v$ from~\eqref{eq:TvDef}. We are interested in the convergence in rough path topologies of $T_{\tilde v^n}(\tilde \Xbf^n)$ where {$\tilde v^n \in \LL^{(2)}(\R^{2d})$} is appropriately chosen. Define the (diagonal) $d\times d$ matrix
\[
v^n = \frac{1}{2}\EEE{\sum_{k=0}^{n-1} (X^n_{k+1}-X^n_k) \otimes (X^n_{k+1}-X^n_k)} = \frac{n^{1-2H}}{2} I,
\]
and the anti-symmetric $(2d) \times (2d)$ matrix
\[
\tilde v^n = \left( \begin{array}{cc} 0 & -v^n  \\ v^n  & 0 \end{array} \right) \in \LL^{(2)}(\R^{2d}).
\]
Finally, consider the path $\tilde X = (X,X) : [0,1] \mapsto \R^{2d}$, its canonically defined L{\'e}vy area $\Abb$ (which exists for $1/4 < H \leq 1$), and its level-$2$ lift {$\tilde \Xbf = \pi_2\exp_\tensor(\tilde X + \Abb)$}. The following result establishes convergence of the ``renormalised path'' $T_{\tilde v^n}(\tilde \Xbf^n)$.

\begin{theorem}[\cite{BCF18} Theorem~5]\label{thm:HoffConv}
Suppose $1/4 < H \leq 1/2$. Then for all $\alpha \in [0, H)$ and $q < \infty$, it holds that $T_{\tilde v^n}(\tilde \Xbf^n) \rightarrow \tilde \Xbf$ in $L^q$ and $\alpha$-H{\"o}lder topology. More precisely, as $n \rightarrow \infty$, in $L^q$
\[
\sup_{s,t \in [0,T]} \frac{|\tilde X^n_{s,t} - \tilde X_{s,t}|}{|t-s|^\alpha} + \sup_{s,t \in [0,T]} \frac{|\Abb^n_{s,t} + (t-s)\tilde v^n - \Abb_{s,t}|}{|t-s|^{2\alpha}} \rightarrow 0.
\]
\end{theorem}

\medskip
\textbf{Rough stochastic volatility and robust It{\^o} integration.}
Applications from quantitative finance recently led to the pathwise study of the
(1-dimensional)\ It{\^o}-integral,
\[
\int_{0}^{T}f (  \hat{B}_{t} ) dB_{t}\text{ with }\hat{B}
_{t}=\int_{0}^{t}\left\vert t-s\right\vert ^{H-1/2}dB_{s}
\]
where $f:\mathbb{R}\rightarrow \mathbb{R}$ is of the form $x \mapsto \exp \left( \eta
x \right) $. When $H\in \left( 0,1/2\right) $, the case relevant in applications, this
stochastic integration is singular in the sense that the mollifier approximations
actually diverge (infinite It{\^o}-Stratonovich correction, due to infinite
quadratic variation of $\hat{B}$ when $H<1/2$).  The integrand $f( \hat{B}_t )$, which plays the role of a stochastic volatility process ($\eta >0$ is a volatility-of-volatility parameter)
is {\it not} a controlled rough path, nor has the pair $(\hat B,B)$ a satisfactory rough path lift
(the It\^o integral $\int \hat B dB$ is well-defined, but $\int B d \hat B$ is not). The correct ``It\^o rough
path'' in this context is then an $\R^{n+1}$-valued ``partial'' branched rough path of the form 
\[
\left( B,\hat{B},\int \hat{B}d{B},...\int \hat{B}^{n}dB\right) 
\]
where $n \sim 1/H$. Again, mollifier approximations will diverge but
it is possible to see that one can carry out a renormalization which restores
convergence to the It\^o limit. (We note the similarity with SPDE situations like KPZ.)
See~\cite{BFGMS17} for details.

\section{Rough differential equations} \label{sec:RDE}

\subsection{Translated rough paths are rough paths} \label{sec:TRPrRP}

We now show that the maps $T_v$ and $M_v$ act on the spaces of weakly geometric and branched rough paths. Throughout, we regard these rough paths as fully lifted, as can always (and uniquely) be done thanks to the extension theorem. The action of our translation operator is then pointwise, i.e.
$$
                                  (M_v\mathbf{X})_{s,t} := M_v (\mathbf{X}_{s,t}),
$$
and similarly for the geometric rough path translation operator $T$.  In the following, we let $|w|$ denotes the length of a word $w \in T(\R^{1+d})$ (resp. number of nodes in a forest $w \in \BPrimal$), and equip the space of $\alpha$-H{\"o}lder weakly geometric (resp. branched) rough paths with the inhomogeneous H{\"o}lder norm
\[
\norm{\mathbf{X}}_{\Hol{\alpha};[s,t]} =  \max_{|w| \leq \floor{1/\alpha}} \sup_{u \neq v \in [s,t]}\frac{|\scal{\mathbf{X}_{u,v},w}|}{|v-u|^{|w|\alpha}},
\]
where the max runs over all words $w \in T(\R^{1+d})$ (resp. forests $w \in \BPrimal$) with $|w| \leq \floor{1/\alpha}$.

\begin{theorem}\label{thm:transRPs} 

Let $\alpha \in (0,1]$ and $\mathbf{X}$ a $\alpha$-H{\"older} weakly geometric (resp. branched) rough path over $\mathbb{R}^{1+d}$.

\begin{enumerate}[label=\upshape(\roman*\upshape)]
\item \label{point:trans1} Let $v = \left(v_{0},v_1\ldots, v_d \right)$ be a collection of elements in $\LL^N(\R^{1+d})$ (resp. in $\BB^N$).

Then $T_v\mathbf{X}$ (resp. $M_v \mathbf{X}$) is a $\alpha/N$-H{\"older} weakly geometric (resp. branched) rough path satisfying
\begin{equation}\label{eq:HolLipNorm}
\norm{T_v\mathbf{X}}_{\Hol{\alpha/N};[s,t]} \textnormal{ (resp. $\norm{M_v\mathbf{X}}_{\Hol{\alpha/N};[s,t]}$) } \leq C_v \norm{\mathbf{X}}_{\Hol{\alpha};[s,t]}
\end{equation}
for a constant $C_v$ depending polynomially on $v$.

\item \label{point:trans2} Let $v= \left( v_{0},0,\dots 0\right)$ for $v_{0}\in \LL^{N}\left( \mathbb{R}^{1+d}\right) $ (resp. $v_{0}\in \BB^N$). Suppose that $\mathbf{X}$ satisfies
\begin{equation}\label{eq:mixedVar}
\norm{\mathbf{X}}_{\Hol{(1,\alpha)};[s,t]} := \max_{|w| \leq \floor{1/\alpha}} \sup_{u \neq v \in [s,t]} \frac{|\scal{\mathbf{X}_{u,v}, w}|}{|v-u|^{(1-\alpha)|w|_0 + \alpha|w|}} < \infty,
\end{equation}
where the $\max$ runs over all words $w \in T(\R^{1+d})$ (resp. forests $w \in \BPrimal$) with $|w| \leq \floor{1/\alpha}$ and $|w|_0$ denotes the number of times the letter $e_0$ (resp. label $0$) appears in $w$.

Then $T_{v}\mathbf{X}$ (resp. $M_v\mathbf{X}$) is a $\alpha \wedge (1/N)$-H{\"older} weakly geometric (resp. branched) rough path over $\mathbb{R}^{1+d}$ satisfying
\[
\norm{T_v\mathbf{X}}_{\Hol{\alpha \wedge (1/N)};[s,t]} \textnormal{ (resp. $\norm{M_v\mathbf{X}}_{\Hol{\alpha \wedge (1/N)};[s,t]}$) } \leq C_v \norm{\mathbf{X}}_{\Hol{(1,\alpha)};[s,t]}
\]
for a constant $C_v$ depending polynomially on $v$.

Finally, in the setting of branched rough paths, let $M : \HH^* \rightarrow \HH^*$ be any algebra morphism which preserves $\GG$ and such that $M \tau = \tau$ for every forest $\tau\in \HH$ without a label $0$, and $M\bullet_0 = M_v\bullet_0 = \bullet_0 + v_0$. Then $M \mathbf{X} = M_v \mathbf{X}$.
\end{enumerate} 
\end{theorem}

Before the proof of the theorem, several remarks are in order.

\begin{remark}
In Theorem~\ref{thm:transRPs} we treat $\alpha$-H{\"older} weakly geometric rough paths as already enhanced with their iterated integrals. Thus $\mathbf{X}_{s,t}$ is an element of $T((\R^{1+d}))$ and $(T_v\mathbf{X})_{s,t}$ is just the image of $\mathbf{X}_{s,t}$ under $T_v$. Therefore the statement of the proposition is that not only does $(T_v\mathbf{X})_{s,t}$ have the correct regularity on the first $n=\floor{1/\alpha}$ levels to qualify as a rough path but that all further iterated integrals are already given, in a purely algebraic way, by $(T_v\mathbf{X})$. That said, if one takes the level-$n$ view, writing $\pi_{n} ( T_v\mathbf{X})$ for the translation only defined as a level-$n$ rough path, the extension theorem asserts that there is a unique full rough path lift, say $\mathbf{Z}$. But then, by the uniqueness part of the extension theorem, $\mathbf{Z} = T_v\mathbf{X}$ so that our construction is compatible with the rough path extension.

The same remark applies to branched rough paths, where we recall that, as a particular consequence of the sewing lemma, every $\alpha$-H{\"older} branched rough paths admits a unique lift (extension) to all of $\BDual$ (\cite{Gubinelli10} Theorem~7.3, or~\cite{HairerKelly15} p.223). We would also like to point out that Boedihardjo~\cite{Boedihardjo15} recently extended a result on the factorial decay of lifts of geometric rough paths (first shown in~\cite{Lyons98}) to the branched setting, answering a conjecture in~\cite{Gubinelli10}.
\end{remark}

\begin{remark}
In the case of geometric rough paths the previous remark points to an alternative (analytic) construction of the translation operator, first defined on a smooth path $X$ identified with its full lift $X \equiv (1,X^1,X^2, ...)$,  and subsequently extended to geometric rough paths by continuity. We stick to the case of one Lie polynomial $v_0 = v = (v^1,v^2,...v^N)$ which we want to add at constant speed to $X$. At level $1$, obviously $(T_v X)^1_{s,t} = X^1_{s,t} + (t-s) v^1$ and $(T_v X)$ is a Lipschitz path (a $1$-rough path). We then perturb the canonically obtained (extended) $2$-rough path which in turn we can perturb on the second level by adding $(t-s) v^2$, thereby obtaining a (non-canonical) $2$-rough path. Iterating this construction allows us to ``feed in, level-by-level'' the perturbation $v$ until we arrive at a rough path $T_v\mathbf{X}$ with regularity $\Hol \alpha \wedge (1/N)$. We leave it to the reader to check that this construction yields indeed $T_v\mathbf{X}$. The severe downside of this construction is its restricted to geometric rough paths, not to mention its repeated use of the (analytic) extension theorem, in a situation that is within reach of purely algebraic methods. 
\end{remark}

\begin{remark}\label{remark:colifting}
The condition on $\mathbf{X}$ in equation~\eqref{eq:mixedVar} is very natural and arises by ``colifting'' a Lipschitz path $X^0$ with a $d$-dimensional $\alpha$-H{\"o}lder weakly geometric rough path. Moreover, this is a special case of a weakly geometric $(p,q)$-rough path (see~\cite{FrizVictoir10} Section~9.4), and the statement can readily be extended to this general setting.
One can also make a statement about the continuity of the maps $(v,\mathbf{X}) \mapsto T_v\mathbf{X}$ and $(v,\mathbf{X}) \mapsto M_v\mathbf{X}$ in suitable rough path topologies. However these points will not be explored here further.
\end{remark}

\begin{remark}\label{remark:algMorphs}
The proof of Theorem~\ref{thm:transRPs} part~\ref{point:trans1} will reveal that the only properties required of $T_v$ (resp. $M_v$) is that it be an algebra morphism, preserves group-like (or equivalently primitive) elements, is upper-triangular (increases grading), and that it increases the grade of every word of length $k$ (resp. forest with $k$ nodes) to at most $Nk$. While already the first of these conditions uniquely determines $T_v$ once $T_v (e_i) = e_i + v_i$ is chosen, we emphasise that without demanding that $M_v$ is a pre-Lie algebra morphism, there is freedom to how $M_v$ can be extended to satisfy these properties even after $M_v(\bullet_i) = \bullet_i + v_i$ is chosen.

In general, different choices of $M_v$ will give rise to different branched rough paths $M_v(\mathbf{X})$. There is a notable exception to this, which is when $\mathbf{X}$ is the canonical lift of a Lipschitz (or more generally $\alpha$-H{\"o}lder, $\alpha \in (1/2,1]$) path in $\R^{1+d}$. Then for every algebra morphism $M : \HH^* \rightarrow \HH^*$ such that $M \bullet_i = M_v \bullet_i = \bullet_i + v_i$, it holds that $M \mathbf{X} = M_v \mathbf{X}$. Indeed, in this case $\mathbf{X}$ is necessarily in the image of $G(\R^{1+d}) \subset T((\R^{1+d}))$ under the embedding~\eqref{eq:tensorForestEmbed}, and since $M$ and $M_v$ agree on the generators $\bullet_i$, it follows that $M\mathbf{X} = M_v \mathbf{X}$ (this discussion relates of course to the final point of Theorem~\ref{thm:transRPs} part~\ref{point:trans2}, where upon demanding additional structure on $\mathbf{X}$, we see that all maps $M$ satisfying the specified properties agree on $\mathbf{X}$).
\end{remark}

\begin{remark} \label{rem:BSig}
Observe that the level-$N$ lift of a weakly geometric rough path is precisely the solution to the linear RDE
\[
dY_t = L(Y_t)d\mathbf{X}_t, \; \; Y_0 = 1 \in T^N(\R^{1+d}),
\]
where $L = (L_0,\ldots,L_d)$ are the linear vector fields on $T^N(\R^{1+d})$ given by right-multiplication by $(e_0,\ldots, e_d)$ respectively.
In much the same way, the level-$N$ truncation of the translated path $Y_t := \pi_N(T_v\mathbf{X}_t)$ is the solution to the modified linear RDE 
\[
dY_t = L^v(Y_t) d\mathbf{X}_t, \; \; Y_0 = 1 \in T^N(\R^{1+d}),
\]
where now $L^v = (L_{e_0 + v_0},\ldots,L_{e_d + v_d})$ are given by right-multiplication by $(e_0 + v_0,\ldots, e_d + v_d)$ (which is a special case of the upcoming Theorem~\ref{thm:transRDEs}).

We note however that the same conclusion does not hold for branched rough paths. Indeed, even the level-$N$ lift of a branched rough path $\mathbf{X}$, $N \geq \floor{1/\alpha}$, is in general not the solution of a linear RDE driven by $\mathbf{X}$, which can easily be seen from the fact that linear RDEs are completely determined by the values $\scal{\mathbf{X}_{s,t}, \tau}$ where $\tau$ ranges over all linear trees $\tau = [\ldots[\bullet_{i_1}]_{\bullet_{i_2}}\ldots]_{\bullet_{i_m}}$ (see, e.g., \cite{HairerKelly15} Example~3.11). A simple example is any branched rough path $\mathbf{X}$ for which $\scal{\mathbf{X}, \tau} \equiv 0$ for all linear trees $\tau$ (e.g., the $\frac{1}{3}$-H{\"o}lder branched rough path for which $\scal{\mathbf{X}_{s,t}, \tau} = t-s$ for some $\tau = [\bullet_i\bullet_j]_{\bullet_k}$ and zero for every other tree $\tau$ of size $|\tau| \leq 3$), so that every linear RDE driven by $\mathbf{X}$ is constant.
\end{remark}

\begin{proof}[Proof of Theorem~\ref{thm:transRPs}]
\ref{point:trans1} We are required to show that
\begin{enumerate}
\item $T_v \mathbf{X}$ takes values in $G(\R^{1+d})$,

\item Chen's relation $(T_v \mathbf{X})_{s,t} \tensor (T_v \mathbf{X})_{t,u} = (T_v \mathbf{X})_{s,u}$ holds, and

\item the analytic condition~\eqref{eq:HolLipNorm}.
\end{enumerate}
The first two properties follow immediately from the analogous properties of $\mathbf{X}$ and the fact that $\restr{T_v}{G(\R^{1+d})} : G(\R^{1+d}) \rightarrow G(\R^{1+d})$ is group morphism. To verify the final property, fix a word $w \in T(\R^{1+d})$. It readily follows from Proposition~\ref{prop:adjointTv} and Remark~\ref{remark:generalGeom} that $T_v^*w = \sum_i \lambda_i w_i$ where $\lambda_i \in \R$ and $w_i$ is a word which satisfies $N|w_i| \geq |w|$. However
\[
|\scal{\mathbf{X}_{s,t}, w_i}| \leq \norm{\mathbf{X}}_{\Hol{\alpha};[s,t]}|t-s|^{\alpha|w_i|},
\]
and thus
\[
|\scal{(T_v\mathbf{X})_{s,t}, w}| = |\scal{\mathbf{X}_{s,t}, T_v^*w}| \leq C\norm{\mathbf{X}}_{\Hol{\alpha};[s,t]} |t-s|^{\alpha|w|/N}
\]
with $C$ depending only on $w$ and (polynomially) on $v$. It follows that $T_v\mathbf{X}$ is indeed a $\alpha/N$-H{\"o}lder rough path, and the desired estimate~\eqref{eq:HolLipNorm} follows by running over all $w$ with $|w|\leq \floor{N/\alpha}$. The proof for the case of branched rough paths is identical, using now Proposition~\ref{prop:adjointMv}.

The proof of the first statement of~\ref{point:trans2} is virtually the same, except we now observe that Proposition~\ref{prop:adjointTv} and the condition $v = v_0 \in \LL^N(\R^{1+d})$ imply that $T_v^*w = \sum_{i} \lambda_i w_i$ where $\lambda_i \in \R$ and $w_i$ is a word which satisfies
\[
N|w_i|_0 + (|w_i| - |w_i|_0) \geq |w|.
\]
The first statement of~\ref{point:trans2} now follows from~\eqref{eq:mixedVar}, and the proof for the case of branched rough paths is again identical.

To show the last point of~\ref{point:trans2}, consider the subspace $\HH^k(\R^d) \oplus \scal{\bullet_0}\subset \HH^k$ spanned by $\bullet_0$ and all forests $\tau \in \HH^k$ without a label $0$. Observe that it suffices to show that for every $k \geq 0$, the level-$k$ truncation $\pi_k \mathbf{X}$ takes values in the subalgebra of $\HH^k$ generated by $\HH^k(\R^d) \oplus \scal{\bullet_0}$.

To this end, consider the space $\tilde C^{\infty}$ defined as the collection of all piecewise smooth paths $\x : [0,T] \rightarrow \GG^k$ for which $\dot{\x} \in \HH^k(\R^d) \oplus \scal{\bullet_0}$ (so that in fact $\dot{\x} \in \BB^k(\R^d) \oplus \scal{\bullet_0}$). For every partition $D = (t_0,\dots, t_m) \subset [0,T]$, we can construct $\x^D \in \tilde C^\infty$ as the piecewise geodesic path (for the Riemannian structure of $\GG^k$) whose increment over $[t_i,t_{i+1}]$ is $\exp( \pi_{\BB^k(\R^d) \oplus \scal{\bullet_0}} \log\mathbf{X}_{t_i,t_{i+1}})$. One can verify that condition~\eqref{eq:mixedVar} guarantees that $\x^D \rightarrow \pi_k \mathbf{X}$ uniformly as $|D| \rightarrow 0$. The conclusion now follows since, by construction, $\x^D$ takes values in the subalgebra generated by $\BB^k(\R^d) \oplus \scal{\bullet_0}$.
\end{proof}

\subsection{Effects of translations on RDEs}\label{subsec:effectsRDEs}

Throughout this section, we assume that $f=\left( f_0, \dots , f_{d}\right)$ is a collection of vector fields on $\R^e$ which are as regular as required for all stated operations and RDEs to make sense.

Observe that $f$ induces a canonical map from $\LL^N(\R^{1+d})$ to the space of vector fields $\Vect(\R^e)$ which extends the map $e_{i}\mapsto f_{i}$. Write $f_{v}$ for the image of $v\in \LL^{N}\left( \mathbb{R}
^{d}\right)$ under this map, e.g., for $v = \left[ e_{1},e_{2}\right] $, we have the vector
field $f_{\left[ e_{1},e_{2}\right] } \equiv \left[ f_{1},f_{2}\right] $. Given a collection $v = (v_0,\ldots, v_d) \subset \LL^N(\R^{1+d})$, we write
\[
f^v = (f^v_0,\ldots, f^v_d) = (f_{e_0 + v_0},\ldots, f_{e_d+v_d}).
\]

Similarly, $f$ induces a canonical map from $\BB^N$ to $\Vect(\R^e)$ which extends $\bullet _{i}\mapsto f_{i}$ using the pre-Lie product $\triangleright$ on $\Vect(\R^e)$ (recall from Example~\ref{example:vectfieldsPreLie} that in coordinates $\left( f^{i}\partial_{i}\right) \triangleright \left( g^{j}\partial _{j}\right) \equiv \left(f^{i}\partial _{i}g^{j}\right) \partial _{j}$). Once more write $f_{v}$ for
the image of $v\in \BB^N$ under this map, e.g., for $v = \left[ \bullet_{1} \right]_{\bullet_2} = \bullet_1 \curvearrowright \bullet_2$, we have the vector field
\[
f_{\bullet _{1}\curvearrowright \bullet _{2}} = f_{\left[ \bullet _{1} \right] _{\bullet_2}} \equiv f_{1}\triangleright f_{2}
\]
Again given a collection $v = (v_0,\ldots, v_d) \subset \BB^N$, we write
\[
f^v = (f^v_0,\ldots, f^v_d) = (f_{\bullet_0 + v_0},\ldots, f_{\bullet_d + v_d}).
\]

\begin{remark}
The map $v \mapsto f_{v}$ is closely related to the notion of elementary differentials in $B$-series~\cite{Calaque11} and has already been used to study solutions of branched RDEs in the works of Cass--Weidner~\cite{CW16} and Hairer--Kelly~\cite{HairerKelly15} (note also that our notation $f_v$ agrees with that of~\cite[Section~3]{HairerKelly15}).
\end{remark}

\begin{remark}
Treating $\LL^N(\R^{1+d})$ (resp. $\BB^N$) as a nilpotent Lie (resp. pre-Lie) algebra, the map considered above is not in general a Lie (resp. pre-Lie) algebra morphism into $\Vect(\R^e)$.
\end{remark}

\begin{theorem}\label{thm:transRDEs}
\begin{enumerate}[label=\upshape(\roman*\upshape)]
\item \label{point:RDE1} Let notation be as in Theorem~\ref{thm:transRPs} part~\ref{point:trans1}. Then $Y$ is an RDE solution flow to     
\[
dY=f\left( Y\right) d\left( T_{v}\mathbf{X}\right) \; \; (\textnormal{resp. } dY=f\left( Y\right) d\left( M_{v}\mathbf{X}\right) )
\]
if and only if $Y$ is an RDE solution flow to  
\[
dY= f^v\left( Y\right) d\mathbf{X}.
\]

\item \label{point:RDE2}  Let notation be as in Theorem~\ref{thm:transRPs} part~\ref{point:trans2}. Then $Y$ is an RDE solution flow to 
\[
dY=f\left( Y\right) d\left( T_{v}\mathbf{X}\right) \; \; (\textnormal{resp. } dY=f\left( Y\right) d\left( M_{v}\mathbf{X}\right) )
\]
if and only if $Y$ is an RDE solution flow to
\[
dY= f^v(Y)d\mathbf{X} \equiv f\left( Y\right) d\mathbf{X}+f_{v_{0}}\left( Y\right) dX^{0}.
\]
\end{enumerate} 
\end{theorem}

\begin{remark}\label{remark:geobranchEmbed}
Since the space of weakly geometric rough paths embeds into the space of branched rough paths using the map~\eqref{eq:tensorForestEmbed}, the statements in Theorem~\ref{thm:transRDEs} for weakly geometric rough paths are a special case of those for branched rough paths. We make a distinction between the two cases only for clarity.
\end{remark}

\begin{proof}
For clarity, we first prove the statement for geometric rough paths and then generalise to branched rough paths (although by Remark~\ref{remark:geobranchEmbed}, it suffices to prove the statement only in the branched case). 

Observe that for weakly geometric rough paths,~\ref{point:RDE1} will follow directly from the usual Euler RDE estimate (\cite{FrizVictoir10} Corollary~10.15)
once we show that
\begin{equation}\label{eq:geoEuler}
\sum_{|u| \leq \floor{1/\alpha}} \scal{\mathbf{X}_{s,t},u} f^v_u(y) = \sum_{|u| \leq \floor{N/\alpha}} \scal{T_v \mathbf{X}_{s,t}, u} f_u(y) + r_{s,t}, \; \; \text{for all } y \in \R^e, \; \; s,t \in [0,T],
\end{equation}
where $|r_{s,t}| = o(|t-s|)$ and where the sums run over any orthonormal basis of $\LL^{\floor{1/\alpha}}(\R^{1+d})$ and $\LL^{\floor{N/\alpha}}(\R^{1+d})$ respectively.

Consider for the moment that $f = (f_0,\ldots, f_d)$ is a collection of smooth vector fields, so that $\Phi_f : u \mapsto f_u$
is a genuine Lie algebra morphism from $\LL(\R^{1+d})$ into $\Vect^\infty(\R^e)$. Hence, whenever $f$ are smooth, the maps $\Phi_f \circ T_v$ and $\Phi_{f^v}$ are both Lie algebra morphisms from $\LL(\R^{1+d})$ into $\Vect^\infty(\R^e)$ which furthermore agree on the generators $e_i$. Thus $\Phi_f \circ T_v = \Phi_{f^v}$, and so
\begin{equation}\label{eq:LieMorph}
\sum_{u} \scal{x,u}f^v_u = \sum_{u} \scal{T_v x,u}f_u, \; \; \text{for all } x \in \LL(\R^{1+d}),
\end{equation}
where both sums run over any orthonormal basis of $\LL(\R^{1+d})$. This proves~\eqref{eq:geoEuler} for smooth $f = (f_0,\ldots, f_d)$. For the general case where $f$ are only sufficiently regular for the stated RDEs to make sense, we note that equality~\eqref{eq:LieMorph} is purely algebraic, so~\eqref{eq:geoEuler} can be readily deduced by truncation.

To extend this argument to the case of branched rough paths, ~\ref{point:RDE1} will follow directly from the Euler estimate derived in~\cite{HairerKelly15} Proposition~3.8 once we show that
\begin{equation}\label{eq:Euler}
\sum_{\tau \in \BB^{\floor{1/\alpha}}} \scal{\mathbf{X}_{s,t},\tau} f^v_\tau(y) = \sum_{\tau \in \BB^{\floor{N/\alpha}}} \scal{M_v \mathbf{X}_{s,t},\tau} f_\tau(y) + r_{s,t}, \; \; \text{for all } y \in \R^e, \; \; s,t \in [0,T],
\end{equation}
where $|r_{s,t}| = o(|t-s|)$ and where the sums run over all trees $\tau$ in $\BB^{\floor{1/\alpha}}$ and $\BB^{\floor{N/\alpha}}$ respectively.

As before, suppose first that $f = (f_0,\ldots, f_d)$ is a collection of smooth vector fields, so that $\Phi_f : x \mapsto f_x \equiv \sum_{\tau\in \BB} \scal{x,\tau}f_\tau$ is a pre-Lie algebra morphism from $\BB$ into $\Vect^\infty(\R^e)$. Hence, whenever $f$ are smooth, the maps $\Phi_f \circ M_v$ and $\Phi_{f^v}$ are both pre-Lie algebra morphisms from $\BB$ into $\Vect^\infty(\R^e)$ which furthermore agree on the generators $\bullet_i$. Thus $\Phi_f \circ M_v = \Phi_{f^v}$, and so
\[
\sum_{\tau \in \BB} \scal{x,\tau}f^v_\tau = \sum_{\tau \in \BB} \scal{M_v x,\tau}f_\tau, \; \; \text{for all } x \in \BB.
\]
As the above equality is purely algebraic, we again deduce~\eqref{eq:Euler} by truncation in the general case where $f$ are only sufficiently regular for the stated RDEs to make sense.

The desired result in~\ref{point:RDE2} for both geometric and branched rough paths follows in the same way.
\end{proof}

\begin{remark} \label{rem:new52}Recall that $M_v : \BB^* \rightarrow \BB^*$ was constructed, from Section \ref{subsubsec:PreLie} on, as a pre-Lie algebra morphism. This matters in part~\ref{point:RDE1} of Theorem~\ref{thm:transRDEs} above, where this property is needed to obtain a universal conversion formula for translated RDEs. For example, consider that $M_v$ was replaced by an algebra morphism $M$ (which satisfies the conditions of Remark \ref{remark:algMorphs}) such that $M(\bullet_i) = \bullet_i$ for all $i=0,...,d$, but acted non-trivially on some higher order trees (so that $M$ is not a pre-Lie morphism). Then given vector fields $f$, in general there does not exist another collection of vector fields $f_v$ such that for every branched rough path $\mathbf{X}$, the RDE driven by $M(\mathbf{X})$ along vector fields $f$ agrees with the RDE driven by $\mathbf{X}$ along $f_v$. Indeed, if such $f_v$ existed, then for every geometric (branched) rough path $\mathbf{X}$ (so that $M(\mathbf{X}) = \mathbf{X}$), the RDEs driven by $M(\mathbf{X})$ and $\mathbf{X}$ agree without the need to change the vector fields $f$, so that necessarily $f_v = f$. However if $\mathbf{X}$ is a non-geometric branched rough path, the RDE driven by $M(\mathbf{X})$ along vector fields $f$ will not in general agree with the RDE driven by $\mathbf{X}$ along $f$.
\end{remark}

\section{Link with renormalization in regularity structures}
\label{sec:renorm}

We now recall several notions from the theory of regularity structures and draw a link between the map $\extract$ from Section~\ref{subsec:DualMapBranched} and the coproduct $\Deltam$  associated to negative renormalization~\cite{BHZ16, Hairer16}. In particular, we demonstrate how negative renormalization maps on the regularity structure associated to branched rough paths carry a natural interpretation as rough path translations (see Theorem~\ref{thm:NegRenorm} below).

\subsection{Regularity structures}

{\it Regularity structures usually deal with (e.g. SPDE solutions) $u=u(z)$ where $z \in \R^n$ (e.g. space-time), $u$ takes values in $\R$ (or $\R^e$). Equations further involve a $\beta$-regularizing kernel, and there are $d$ sources of noise, say $\xi_1,...,\xi_d$, of arbitrary (negative) order $\alpha_{\min}$, as long as the equation is subcritical.}

\subsubsection{Generalities}  

{\it We review the general (algebraic) setup in the case $n=1$, $\beta = 1$ and $\alpha_{\min} \in (-1,0)$. }

\medskip

In the spirit of Hairer's formalism, consider the equation
\begin{equation}\label{eq:RDEEq}
u(t) = u(0) + \left(K * \sum_{i=1}^d f_i(u(\cdot))\xi_i(\cdot) \right)(t), \; \; t \in \R,
\end{equation}
where $u(t)$ is a real-valued function for which we solve, $\xi_i(t)$ are driving noises, $f_i$ are smooth functions on $\R$ (one could readily extend to the case that $u$ takes values in $\R^e$ and $f_i$ are vector fields on $\R^e$), and $K$ is a kernel which improves regularity by order $\beta = 1$.

\begin{remark}
The example to have in mind here is $K(s) = \exp (- \lambda s) 1_{s>0}$, which allows to incorporate an additional linear drift term (``$-\lambda u dt$''), or of course the case 
$\lambda =0$, i.e. the Heaviside step function, which leads to the usual setting of controlled differential equations. We shall indeed specialize to the Heaviside case in subsequent
sections, as this simplifies some algebraic constructions and so provides a clean link to rough path structures. For the time being, however, we find it instructive to work
with a general $1$-regularizing $K$, as this illustrates the need for polynomials decorations as well as symbols $\mathcal{J}_k$, representing $k$-th derivatives of the kernel.
\end{remark}

Our driving noises $\xi_i(t)$ should be treated as distributions on $\R$ of regularity $C^{\alpha-1}$ for some $\alpha \in (0, 1)$ (which will later correspond to the case of $\alpha$-H{\"o}lder branched rough paths). In the case that $\alpha \leq 1/2$, due to the product $f_i(u)\xi_i$,~\eqref{eq:RDEEq} is singular and thus cannot in general be solved analytically. However the equation is evidently sub-critical in the sense of~\cite{Hairer14}, and so one can build an associated regularity structure.

\medskip

{\bf Introducing the symbols} 

\medskip

We first collect all the symbols of the regularity structure required to solve~\eqref{eq:RDEEq} and which is stable under the renormalization maps in the sense of~\cite{BHZ16}.
Define the linear space $$\TT = \scal{\WW},$$
where $\WW$ is the set of all rooted trees where every node carries a ``polynomial'' decoration $k \in \N \cup \{0\}$ and where every edge which ends on a leaf may be (but is not necessarily) assigned a type $ \mathfrak{t}_{\Xi_i}, $ $i \in \{1,\ldots, d\}$. An edge with type $\mathfrak{t}_{\Xi_i}$ corresponds to the driving noise $ \xi_i $. Every other edge has a type $ \mathfrak{t}_{K} $ which means that it is associated to the kernel $ K $. (For now, we only assume $K$ is $1$-regularizing, later we will take it to be the Heaviside step function.) Also, each node has at most one incoming edge with type belonging to $  \{1,\ldots, d\}$.\footnote{This rules out symbols corresponding to products of noise, such as $\Xi_i \Xi_j$ with $i,j \in   \{1,\ldots, d\}$.} 
With regard to ~\cite{BHZ16}, we also note the absence of edge decorations.\footnote{ This is in contrast to, say, KPZ or $\Phi^4_3$, where edge decorations appear in view of $Du \to \II'$ or renormalization, respectively.}

\medskip
To avoid confusion between the different meaning of trees in $\WW$ and those introduced in Section~\ref{sec:Branched}, we will color every tree in $\WW$ blue. Every such tree has a corresponding symbol representation, e.g.,

\begin{equs}
\tikzexternaldisable  \begin{tikzpicture}[scale=0.2,baseline=0.1cm]
        \node at (0,0)  [fill,circle,scale=0.22,blue,label= {[label distance=-0.2em]below: \scriptsize  $  $} ] (root) {};
        \node at (0,3)  [fill,circle,scale=0.22,blue,label= {[label distance=-0.2em]right: \scriptsize  $  $} ] (center) {};
         \node at (1,1.5)  [label={[label distance=0em]center: \scriptsize  $ \mathfrak{t}_{K} $} ] (right1) {};
     \draw[kernel1] (root) to
     node [sloped,below] {\small }     (center);
     \end{tikzpicture} \color{black} \leftrightarrow   \begin{tikzpicture}[scale=0.2,baseline=0.1cm]
        \node at (0,0)  [fill,circle,scale=0.22,blue,label= {[label distance=-0.2em]below: \scriptsize  $  $} ] (root) {};
        \node at (0,3)  [fill,circle,scale=0.22,blue,label= {[label distance=-0.2em]right: \scriptsize  $  $} ] (center) {};
         \node at (0.8,1.5)  [label={[label distance=0em]center: \scriptsize  $  $} ] (right1) {};
     \draw[kernel1,blue,thin] (root) to
     node [sloped,below] {\small }     (center);
     \end{tikzpicture} \color{black} \leftrightarrow \II, \quad \tikzexternaldisable {{ \begin{tikzpicture}[scale=0.2,baseline=0.1cm]
        \node at (0,0)  [fill,circle,scale=0.22,blue,label= {[label distance=-0.2em]below: \scriptsize  $  $} ] (root) {};
        \node at (0,3)  [fill,circle,scale=0.22,blue,label= {[label distance=-0.2em]right: \scriptsize  $  $} ] (center) {};
         \node at (1,1.5)  [label={[label distance=0em]center: \scriptsize  $ \mathfrak{t}_{\Xi_i} $} ] (right1) {};
     \draw[kernel1] (root) to
     node [sloped,below] {\small }     (center);
     \end{tikzpicture}} } \leftrightarrow
     \begin{tikzpicture}[scale=0.2,baseline=0.1cm]
        \node at (0,0)  [fill,circle,scale=0.22,blue,label= {[label distance=-0.2em]below: \scriptsize  $  $} ] (root) {};
        \node at (0,3)  [fill,circle,scale=0.22,blue,label= {[label distance=-0.2em]right: \scriptsize  $  $} ] (center) {};
     \draw[kernel1,blue,thin] (root) to
     node [sloped,below] {\small }     (center);
     \node at (0,1.5)  [fill=white,label={[label distance=0em]center: \scriptsize  $ i $} ] (right1) {};
     \end{tikzpicture}  
  \leftrightarrow \Xi_{i}, \quad
      \quad \tikzexternaldisable   \begin{tikzpicture}[scale=0.2,baseline=0.1cm]
        \node at (0,1)  [fill,circle,scale=0.22,blue,label= {[label distance=-0.2em]below: \scriptsize  $ k $} ] (root) {};
     \end{tikzpicture} \color{black} \leftrightarrow X^k,
 \end{equs}

\begin{equs} 
\begin{tikzpicture}[scale=0.2,baseline=0.1cm]
        \node at (0,0)  [fill,circle,scale=0.22,blue,label= {[label distance=-0.2em]right: \tiny  $ 6 $} ] (root) {};
        \node at (0,3)  [fill,circle,scale=0.22,blue,label= {[label distance=-0.2em]right: \scriptsize  $  $} ] (center) {};
        \node at (0,6)  [fill,circle,scale=0.22,blue,label= {[label distance=-0.2em]right: \scriptsize  $  $} ] (centerc) {};
         \node at (-2,6)  [fill,circle,scale=0.22,blue,label= {[label distance=-0.2em]left: \tiny  $ 7 $} ] (centerl) {};
         \node at (2,6)  [fill,circle,scale=0.22,blue,label= {[label distance=-0.2em]right: \tiny  $ 5 $} ] (centerr) {};
         \node at (2,9)  [fill,circle,scale=0.22,blue,label= {[label distance=-0.2em]right: \tiny  $  $} ] (centerrc) {};
         \node at (0,9)  [fill,circle,scale=0.22,blue,label= {[label distance=-0.2em]right: \tiny  $  $} ] (centercc) {};
     \draw[kernel1,blue,thin] (root) to
     node [sloped,below] {\small }     (center);
     \draw[kernel1,blue,thin] (center) to
     node [sloped,below] {\small }     (centerl);
     \draw[kernel1,blue,thin] (center) to
     node [sloped,below] {\small }     (centerr);
     \draw[kernel1,blue,thin] (center) to
     node [sloped,below] {\small }     (centerc);
     \draw[kernel1,blue,thin] (centerc) to
     node [sloped,below] {\small }     (centercc);
     \draw[kernel1,blue,thin] (centerr) to
     node [sloped,below] {\small }     (centerrc);
      \node at (0,7.5)  [fill=white,label={[label distance=0em]center: \tiny  $ 1 $} ] (right1) {};
       \node at (2,7.5)  [fill=white,label={[label distance=0em]center: \tiny  $ 2 $} ] (right11) {};
     \end{tikzpicture} 
 \leftrightarrow \II(\II(\Xi_1) \II(\Xi_2 X^5) \II(X^7) ) X^6,
\end{equs}
where we implicitly drop the $0$ decoration ($ \leftrightarrow X^0 $) from the nodes. It is instructive to check that $\WW$ provides an example of a structure built from a subcritical complete rule (in the sense of~\cite{BHZ16} Section~5) arising from the equation~\eqref{eq:RDEEq}. Indeed, we can give the set of rules used for the construction of
\begin{equs}
R(\Xi_i) = \lbrace  () \rbrace, \quad R(\II) = \lbrace  ([\II]_{\ell}), ([\II]_{\ell},\Xi_i), \ell \in \mathbb{N} \cup \{ 0 \}, i \in \lbrace 1,...,d \rbrace \rbrace.
\end{equs}
The notation $ [\II]_{\ell} $ is a shorthand notation for $ \II,...,\II $ where $ \II $ is repeated $ \ell $ times.

We define a degree $ |\cdot| $ associated to an edge type and a decorated tree. For edge types and polynomials, we have
\[
 | \Xi_i | =  \alpha-1, \quad | \II | =1, \quad | X^k | = k. 
\]
Then by recursion,
\[
  | \II(\tau) | = | \tau | + |\II|, \quad \left| \prod_i \tau_i \right| = \sum_i |\tau_i| .  
\] 
For a non-recursive definition see~\cite{BHZ16} where the degree is described through a summation over all the edge types and the decorations in the tree.

\begin{remark} Remark that $\WW  \equiv \WW_{\BHZr}$ (the ``r'' in BHZr refers to {\it reduced}, in the terminology of~\cite{BHZ16} these are trees without any extended decorations) will contain certain symbols which do {\it not}  arise if one follows the original procedure of~\cite{Hairer14} (which, in some sense, is the most economical way to build the structure):
$$
     \WW_{\Hai} \subset  \WW_{\BHZr} \subset  \WW_{\BHZ}.
$$
Indeed in~\cite{Hairer14}, the set of rules is not necessarily complete so one has to add terms by hand coming from the renormalization procedure and in the end one works with a space 
$ \bar \WW_{\Hai} $ lying between $ \WW_{\Hai} $ and $ \WW_{\BHZr} $.
For example, $\II(\Xi_i)\II(\Xi_j)$, $\II (\II (\Xi_k))$, and $\II() \equiv \II(X^0)$ do not appear in $\WW_{\Hai}$, but all of these appear in $\WW_{\BHZr}$. These in turn are embedded in $\WW_{\BHZ}$, a set of trees with extended decorations on the nodes and also colourings of the nodes which give more algebraic properties. In the setting of~\cite{BHZ16}, we would work with an additional symbol $ \one_{\alpha} $ for $ \alpha \in \mathbb{R} $, representing an extended decoration, which provides information on some ``singular'' (negative degree) tree which has been removed, and all of these symbols are would be placed using a complete set of rules.
\end{remark} 

\medskip

{\bf Introducing  $ \TT_- $} 

\medskip

We define the space $ \TT_- $ as 
\begin{equ}[def:Tm]
  \TT_- = \lbrace \tau_1 \bullet \cdots \bullet \tau_n, \; \tau_i \in \WW, \; | \tau_i | < 0\rbrace.
\end{equ}
where $ \bullet $ is the forest product and the unit is given by the empty forest. (In other words,  $\TT_-$ is the free unital commutative algebra generated by elements 
in $\WW$ of negative degree.)
We now recall that $\TT_-$ can be equipped with a Hopf algebra structure $\TT_-$ for which there exists a coaction $\Deltam : \TT \rightarrow \TT_- \otimes \TT$ such that $(\TT, \Deltam)$ is a (left) comodule over $\TT_-$. Then the action of a character $\ell \in \TT_-^*$ on $x \in \TT$, termed ``negative renormalization'', is given by $M_\ell x = (\ell\otimes \id)\Deltam x$. 

Following ~\cite{Hairer16} Section~2 we can describe the coaction $\Deltam$ as follows. Fix a tree $\tau \in \WW$, consider
a subforest $A \subset \tau$, i.e., an arbitrary subgraph of $\tau$ which contains
no isolated vertices. We then write $R_A \tau$ for the 
tree obtained by contracting the connected components of $A$ in $\tau$.
With this notation at hand, we then define a linear map, the coaction,  
$$\Deltam \colon  \TT \to \TT_- \otimes  \TT$$ 
by setting, for $\tau \in  \WW$,
\begin{equs}\label{e:Deltabar}
\Deltam \tau &= 
 \sum_{A \subset \TT_-}  A \otimes R_A \tau.
\end{equs}
Unfortunately, this is not quite the correct coaction as it does not handle correctly the powers of $X$. However, upon restriction to $\tilde \TT \subset \TT$, as done in detail in the next section, this is precisely the form of the coaction (now on $\tilde \TT$). When moving to a coproduct this fortunately plays no role (since $\TT_-$ does not contain any non-zero powers of $X$ or a factor of the form $\II()$). By abuse of notation, $\Deltam$ also acts as a coproduct, that is 
\begin{equ} \label{Dmcp}
\Deltam \colon \TT_- \to \TT_- \otimes \TT_-.
\end{equ}
To be explicit, given $f = \tau_1 \cdots \tau_n \in \TT$, we have $\Deltam (f) = \Deltam (\tau_1) ... \Deltam(\tau_n)$ with each $\Deltam (\tau_i)$ as defined above, but with an additional projection to the negative trees on the right-hand side of the tensor-product. 

 \begin{remark} \label{cured_trees}The spaces $ \TT_- \equiv \TT^-_{\BHZr} $, 
 $ \TT^-_{\BHZ} $ and $ \TT^-_{\Hai} $ are the same {\it in this framework} (cf. assumptions from the beginning of this subsection). Indeed, all negative trees of $ \WW $ have a degree of the form $ N \alpha -1 $. Then if we remove one negative subtree, of degree $M \alpha -1 $ say, from a negative tree, we obtain a degree $(N-M)\alpha$ which is positive and hence the ``cured'' tree does not belong to $ \TT_- $. 
 \end{remark}

\medskip
{\bf Introducing  $ \TT_+ $} 
\medskip

In order to describe the space $ \TT_+ $ as in~\cite{BHZ16}, we need to associate to each edge a decoration $ k \in \mathbb{N} \cup \lbrace  0 \rbrace $ viewed as a derivation of the kernels or the driving noises. Such a decoration does not appear in $ \TT $. Thus we will replace the letter $ \II $ by $ \JJ $ in this context.  We do not give any graphical notation for $ \JJ_k $, the edge with type $ \mathfrak{t}_{K} $ and (edge) decoration $ k $ representing $ K^{(k)} $, because these symbols ultimately will not appear in our context.

\medskip

We define $ \TT_+ $  as the linear span of 
\[
\lbrace  X^k \prod_{i=1}^n  \mathcal{J}_{k_i}(\tau_i) \mid k,n \in \N \cup\{ 0\}, k_i \in \N \cup\{ 0\}, \tau_i \in \WW,\; |\tau_i| + 1 - k_i > 0 \rbrace.
\]
In other words,  $\TT_+$ is the free unital commutative algebra generated by 
\[
\WW_+ := \{X\} \cup \{\mathcal{J}_k\tau \mid \tau \in \WW, |\tau|+1-k > 0 \}.
\]
We use a different letter $ \mathcal{J} $ to stress that $ \WW $ is different from $ \WW^+ $. Moreover, the use of this letter is viewed in~\cite{BHZ16} as a colouration of the root and plays a role in the sequel. We also define the degree of a term
\[
\tau = X^k \prod_{i=1}^n  \mathcal{J}_{k_i}(\tau_i) \in \TT_+, \; \; |\tau| = k + \sum_{i=1}^n 1 - k_i + |\tau_i|.
\]

The space $\TT_+$ is used in the description of the structure group associated to $\TT$. More precisely, recall that $\TT_+$ can be equipped with a Hopf algebra structure for which there exists a coaction $\Deltap: \TT \rightarrow \TT \otimes \TT_+$ such that $(\TT, \Deltap)$ is a (right) comodule over $\TT_+$. Following Hairer's survey \cite{Hairer16}, the coaction
\begin{equ}[e:deltapFirst]
\Deltap \colon \TT \to \TT \otimes \TT_+
\end{equ}
is given by 
\begin{equ}[e:deltap1]
\Deltap X_i = X_i \otimes \one + \one \otimes X_i\;,\qquad
\Deltap \Xi_i = \Xi_i \otimes \one\;,
\end{equ}
and then recursively by
\begin{equ}[e:recDelta]
\Deltap \II(\tau) = (\II \otimes \id)\Deltap \tau
+ \sum_{\ell \in \N \cup \{0\}} \frac{X^\ell}{\ell!} \otimes \JJ_{\ell}(\tau)
\end{equ}
and
\begin{equation}\label{eq:DeltapMult}
\Deltap (\tau \bar \tau) = \Deltap \tau\,\Deltap \bar \tau.
\end{equation}
The coproduct $\Deltap: \TT_+ \rightarrow \TT_+ \otimes \TT_+$ is then defined in the same way by replacing~\eqref{e:recDelta} with
\begin{equ}
\Deltap \JJ_k(\tau) = (\JJ_k \otimes \id)\Deltap \tau
+ \sum_{\ell \in \N \cup \{0\}} \frac{X^\ell}{\ell!} \otimes \JJ_{k+\ell}(\tau),
\end{equ}
in which $\Deltap \tau$ is understood as the coaction $\Deltap : \TT \rightarrow \TT\otimes\TT_+$.

Then the action of a character $g \in \TT_+^*$ on $x \in \TT$, termed ``positive renormalization'', is given by 
$$
\Gamma_g x = (\id \otimes g)\Deltap x.
$$

\begin{remark} The space $ \TT_+ \equiv  \TT^+_{\BHZr} $ depends strongly on the space $ \WW $. We have 
$$
     \TT^+_{\Hai} \subset  \TT^+_{\BHZr} \subset  \TT^+_{\BHZ}.
$$
These two inclusions are Hopf subalgebra inclusions.
Indeed, as proved in~\cite{BHZ16}, the second one, with
 $ \TT_+$ equipped with coproduct $\Deltap$ is a Hopf subalgebra inclusion (with $\Deltap_{\BHZ}$ found in \cite[(4.14)]{BHZ16}). The same is also true for $ \TT^+_{\Hai} $. The key point for the Hopf algebra structure is that, in the terminology of ~\cite{BHZ16}, the symbols defined in \cite{Hairer14} and \cite{BHZ16} are obtained by a ``normal rule'' which guarantees the invariance under $ \Deltap $.
 In the case of $ \TT^{+}_{\BHZ} $, we use  the degree $ | \cdot |_+ $ which is exactly $ |\cdot| $ when we restrict ourselves to $ \TT^{+}_{\BHZr} $. 
\end{remark}

\begin{remark} \label{cointeraction}
Unfortunately, there is a problem here in that, with the definition in equation~\eqref{e:recDelta}, a desirable cointeraction between $\Deltap$ and $\Deltam$ fails as we shall explain momentarily.
The ``official'' remedy, following \cite{BHZ16}, is to use the extended decorations through another degree $ |\cdot|_+ $ which takes into account these decorations and behaves the same as $ |\cdot| $ for the rest. For example, one has $ | \II(\one_{\beta} \tau) |_+ = | \tau |_++1 + \beta $. The ``correct'' coaction $ \Deltap $ (see \cite[(4.14)]{BHZ16})  then also involves these extended decorations. The extended decorations are crucial in~\cite{BHZ16} for obtaining a  cointeraction between the two Hopf algebras $ (\TT_+,\Deltap) $ and $ (\TT_-,\Deltam) $:
 \[
  \mathcal{M}^{(13)(2)(4)} \left( \Deltam \otimes \Deltam \right) \Deltap = \left( \id \otimes \Deltap \right) \Deltam
 \]
 where $  \mathcal{M}^{(13)(2)(4)} $ is given as $ \mathcal{M}^{(13)(2)(4)} \left( \tau_1 \otimes \tau_2 \otimes \tau_3 \otimes \tau_4 \right) = \left( \tau_1 \bullet \tau_3 \right) \otimes \tau_2 \otimes \tau_4 $. This identity is both true on $ \TT $ through the comodule structures and on $ \TT_+ $ when the coproduct $ \Deltam $ is viewed as an action on $ \TT_+ $. We have already came across something similar in Lemma \ref{lem:adjointComm}, but in that case the maps involved were not really coproducts. In our simple framework, this property is not satisfied if we just consider the reduced structure. One can circumvent this issue without introducing extended decorations by changing the coproduct $\Deltap$ to the form~\eqref{e:recDeltaNew} given below. This approach is possible in our context (specifically, minimal degree $\alpha-1 > -1$ and $1$-regularizing kernel) because we know {\it a priori} that each edge type $ \II $ in the elements of $ \WW$ with negative degree has the same ``Taylor expansion'' of length $ 1 $ in \eqref{e:recDelta} ($ \ell = 0 $). In general, we would use the extended decorations to maintain this property, however, in
the specific setting of the Heaviside kernel, {\it to which we will specialize from this moment on to the rest of the paper,} we can just fix the length in the coproduct and not use the extended decorations. That is, we can get away by replacing~\eqref{e:recDelta} with the same formula, but only keeping $\ell = 0$ in the sum. Specifically, with $\JJ \equiv \JJ_0$ this amounts to make the (recursive) definition of $\Deltap$ with (\ref{e:recDelta}) replaced by
\begin{equ}\label{e:recDeltaNew}
\Deltap \II(\tau) = (\II \otimes \id)\Deltap \tau
+ 1 \otimes \JJ(\tau).
\end{equ}
We can also get rid of colours when we have no derivatives on the edges at the root: if we want to extract from $ \II(\tau_1 \Xi_i)  \II(\tau_2 \Xi_j) $ all the negative subtrees, we observe that it is not possible to extract one at the root, and thus are only left with negative subtrees in $\tau_1 \Xi_i $ and $ \tau_2 \Xi_j $, which ensures that
\[
  M_{\ell}  \II(\tau_1 \Xi_i) \II(\tau_2 \Xi_j)  = \II\left( M_{\ell} \left( \tau_1 \Xi_i\right) \right)  \II\left(   M_{\ell} \left(\tau_2 \Xi_j\right) \right).
\]
In the setting of~\cite{BHZ16}, this multiplicativity property is encoded by a colour at the root which avoids the extraction of a tree containing the root.
\end{remark}

\subsubsection{The case of rough differential equations} \label{subsubsec:RDEs}

{\it As in the last subsection: $n=1$, $\beta = 1$ and noise degree $\alpha_{\min} \in (-1,0) >-1$. We further specialize the algebraic set in that no symbols $\mathcal{J}_k$ and polynomials $X^k$ with $ k>0$ are required in describing $\TT_+$. }

\medskip

Assuming $K$ to be the Heaviside step function, all derivatives (away from the origin) are zero, hence there is no need (with regard to $\WW$)  to have any polynomial symbols ($X^k$ with $k > 0$). Removing these from $\WW$ leaves us with $\tilde \WW \subset \WW$ which we may list as

\begin{align}
\begin{split}
 \tilde \WW = \{ \Xi_i, ... & , \II(\Xi_i)\II(\Xi_j)\Xi_k,  ... , 1, \II(\Xi_i), \II(\Xi_i) \II(\Xi_j), ...  \label{equ:WWap} \\     
&  ... , \II ( \II(\Xi_i)\II(\Xi_j)  \Xi_k ), \II (  \II(\Xi_i) \II(\Xi_j)),  ... , \II()\II(), \II(\II()),... \},
\end{split}
\end{align}
(all indices are allowed to vary from $1,...,d$), with associated degrees $| \tau |$ as follows:\footnote{tacitly assuming $\alpha < 1/3$} 
$$
                \alpha -1, ..., 3 \alpha -1, .... , 0, \alpha, 2 \alpha, ... \ ...  , 3 \alpha, 2 \alpha +1 , .... , 2, 2, ...
$$
As in the case of $\WW$, elements of $\tilde \WW$  can be viewed as rooted trees, but {\it without} node decorations. For instance,

$$ \begin{tikzpicture}[scale=0.2,baseline=0.1cm]
        \node at (0,0)  [fill,circle,scale=0.22,blue,label= {[label distance=-0.2em]right: \scriptsize  $  $} ] (center) {};
        \node at (0,3)  [fill,circle,scale=0.22,blue,label= {[label distance=-0.2em]right: \scriptsize  $  $} ] (centerc) {};
         \node at (-2,3)  [fill,circle,scale=0.22,blue,label= {[label distance=-0.2em]left: \tiny  $  $} ] (centerl) {};
         \node at (2,3)  [fill,circle,scale=0.22,blue,label= {[label distance=-0.2em]right: \tiny  $  $} ] (centerr) {};
         \node at (-2,6)  [fill,circle,scale=0.22,blue,label= {[label distance=-0.2em]right: \tiny  $  $} ] (centerlc) {};
         \node at (0,6)  [fill,circle,scale=0.22,blue,label= {[label distance=-0.2em]right: \tiny  $  $} ] (centercc) {};
     \draw[kernel1,blue,thin] (center) to
     node [sloped,below] {\small }     (centerl);
     \draw[kernel1,blue,thin] (center) to
     node [sloped,below] {\small }     (centerr);
     \draw[kernel1,blue,thin] (center) to
     node [sloped,below] {\small }     (centerc);
     \draw[kernel1,blue,thin] (centerc) to
     node [sloped,below] {\small }     (centercc);
     \draw[kernel1,blue,thin] (centerl) to
     node [sloped,below] {\small }     (centerlc);
      \node at (-2,4.5)  [fill=white,label={[label distance=0em]center: \tiny  $ i $} ] (right1) {};
       \node at (0,4.5)  [fill=white,label={[label distance=0em]center: \tiny  $ j $} ] (right11) {};
       \node at (1,1.5)  [fill=white,label={[label distance=0em]center: \tiny  $ k $} ] (right11) {};
     \end{tikzpicture} \leftrightarrow  \II(\Xi_i)\II(\Xi_j)\Xi_k, \quad
\begin{tikzpicture}[scale=0.2,baseline=0.1cm]
        \node at (0,0)  [fill,circle,scale=0.22,blue,label= {[label distance=-0.2em]right: \tiny  $  $} ] (root) {};
        \node at (0,3)  [fill,circle,scale=0.22,blue,label= {[label distance=-0.2em]right: \scriptsize  $  $} ] (center) {};
        \node at (0,6)  [fill,circle,scale=0.22,blue,label= {[label distance=-0.2em]right: \scriptsize  $  $} ] (centerc) {};
         \node at (-2,6)  [fill,circle,scale=0.22,blue,label= {[label distance=-0.2em]left: \tiny  $  $} ] (centerl) {};
         \node at (2,6)  [fill,circle,scale=0.22,blue,label= {[label distance=-0.2em]right: \tiny  $  $} ] (centerr) {};
         \node at (-2,9)  [fill,circle,scale=0.22,blue,label= {[label distance=-0.2em]right: \tiny  $  $} ] (centerlc) {};
         \node at (0,9)  [fill,circle,scale=0.22,blue,label= {[label distance=-0.2em]right: \tiny  $  $} ] (centercc) {};
     \draw[kernel1,blue,thin] (root) to
     node [sloped,below] {\small }     (center);
     \draw[kernel1,blue,thin] (center) to
     node [sloped,below] {\small }     (centerl);
     \draw[kernel1,blue,thin] (center) to
     node [sloped,below] {\small }     (centerr);
     \draw[kernel1,blue,thin] (center) to
     node [sloped,below] {\small }     (centerc);
     \draw[kernel1,blue,thin] (centerc) to
     node [sloped,below] {\small }     (centercc);
     \draw[kernel1,blue,thin] (centerl) to
     node [sloped,below] {\small }     (centerlc);
      \node at (-2,7.5)  [fill=white,label={[label distance=0em]center: \tiny  $ i $} ] (right1) {};
       \node at (0,7.5)  [fill=white,label={[label distance=0em]center: \tiny  $ j $} ] (right11) {};
       \node at (1,4.5)  [fill=white,label={[label distance=0em]center: \tiny  $ k $} ] (right11) {};
     \end{tikzpicture}
 \leftrightarrow \II ( \II(\Xi_i)\II(\Xi_j)\Xi_k),$$
are trees ($ \leftrightarrow$ symbols) contained in $\WW$, and also in $\WW_{\Hai}$, the symbols arising in the construction of~\cite{Hairer14}, whereas 
$$
 \begin{tikzpicture}[scale=0.2,baseline=0.1cm]
        \node at (0,0)  [fill,circle,scale=0.22,blue,label= {[label distance=-0.2em]below: \scriptsize  $  $} ] (root) {};
        \node at (-1.5,3)  [fill,circle,scale=0.22,blue,label= {[label distance=-0.2em]right: \scriptsize  $  $} ] (left) {};
        \node at (1.5,3)  [fill,circle,scale=0.22,blue,label= {[label distance=-0.2em]right: \scriptsize  $  $} ] (right) {};
     \draw[kernel1,blue,thin] (root) to
     node [sloped,below] {\small }     (right);
     \draw[kernel1,blue,thin] (root) to
     node [sloped,below] {\small }     (left);
     \end{tikzpicture} 
     \leftrightarrow \II()\II(), \quad  
     \begin{tikzpicture}[scale=0.2,baseline=0.1cm]
        \node at (0,0)  [fill,circle,scale=0.22,blue,label= {[label distance=-0.2em]below: \scriptsize  $  $} ] (root) {};
        \node at (0,2)  [fill,circle,scale=0.22,blue,label= {[label distance=-0.2em]right: \scriptsize  $  $} ] (left) {};
        \node at (0,4)  [fill,circle,scale=0.22,blue,label= {[label distance=-0.2em]right: \scriptsize  $  $} ] (right) {};
     \draw[kernel1,blue,thin] (root) to
     node [sloped,below] {\small }     (right);
     \draw[kernel1,blue,thin] (root) to
     node [sloped,below] {\small }     (left);
     \end{tikzpicture} 
 \leftrightarrow \II( \II()), \quad  
\begin{tikzpicture}[scale=0.2,baseline=0.1cm]
        \node at (0,0)  [fill,circle,scale=0.22,blue,label= {[label distance=-0.2em]right: \scriptsize  $  $} ] (center) {};
         \node at (-1.5,3)  [fill,circle,scale=0.22,blue,label= {[label distance=-0.2em]left: \tiny  $  $} ] (centerl) {};
         \node at (1.5,3)  [fill,circle,scale=0.22,blue,label= {[label distance=-0.2em]left: \tiny  $  $} ] (centerr) {};
         \node at (-1.5,6)  [fill,circle,scale=0.22,blue,label= {[label distance=-0.2em]right: \tiny  $  $} ] (centerlc) {};
         \node at (1.5,6)  [fill,circle,scale=0.22,blue,label= {[label distance=-0.2em]right: \tiny  $  $} ] (centerrc) {};
     \draw[kernel1,blue,thin] (center) to
     node [sloped,below] {\small }     (centerl);
     \draw[kernel1,blue,thin] (center) to
     node [sloped,below] {\small }     (centerr);
     \draw[kernel1,blue,thin] (centerr) to
     node [sloped,below] {\small }     (centerrc);
     \draw[kernel1,blue,thin] (centerl) to
     node [sloped,below] {\small }     (centerlc);
      \node at (-1.5,4.5)  [fill=white,label={[label distance=0em]center: \tiny  $ i $} ] (right1) {};
       \node at (1.5,4.5)  [fill=white,label={[label distance=0em]center: \tiny  $ j $} ] (right11) {};
     \end{tikzpicture}
      \leftrightarrow \II(\Xi_i)\II(\Xi_j), \quad  
\begin{tikzpicture}[scale=0.2,baseline=0.1cm]
        \node at (0,0)  [fill,circle,scale=0.22,blue,label= {[label distance=-0.2em]right: \tiny  $  $} ] (root) {};
        \node at (0,3)  [fill,circle,scale=0.22,blue,label= {[label distance=-0.2em]right: \scriptsize  $  $} ] (center) {};
         \node at (-1.5,6)  [fill,circle,scale=0.22,blue,label= {[label distance=-0.2em]left: \tiny  $  $} ] (centerl) {};
         \node at (1.5,6)  [fill,circle,scale=0.22,blue,label= {[label distance=-0.2em]left: \tiny  $  $} ] (centerr) {};
         \node at (-1.5,9)  [fill,circle,scale=0.22,blue,label= {[label distance=-0.2em]right: \tiny  $  $} ] (centerlc) {};
         \node at (1.5,9)  [fill,circle,scale=0.22,blue,label= {[label distance=-0.2em]right: \tiny  $  $} ] (centerrc) {};
     \draw[kernel1,blue,thin] (root) to
     node [sloped,below] {\small }     (center);
     \draw[kernel1,blue,thin] (center) to
     node [sloped,below] {\small }     (centerl);
     \draw[kernel1,blue,thin] (center) to
     node [sloped,below] {\small }     (centerr);
     \draw[kernel1,blue,thin] (centerr) to
     node [sloped,below] {\small }     (centerrc);
     \draw[kernel1,blue,thin] (centerl) to
     node [sloped,below] {\small }     (centerlc);
      \node at (-1.5,7.5)  [fill=white,label={[label distance=0em]center: \tiny  $ i $} ] (right1) {};
       \node at (1.5,7.5)  [fill=white,label={[label distance=0em]center: \tiny  $ j $} ] (right11) {};
     \end{tikzpicture} \leftrightarrow \II (  \II(\Xi_i) \II(\Xi_j)), 
$$
are contained in $\WW$, following the above construction taken from \cite{BHZ16}, in order to obtain stability under the negative renormalization maps (but not included in $\WW_{\Hai}$.) 

\medskip 

\noindent A linear subspace of $ \TT = \scal{ \WW}$ is then given by
\begin{equ} 
   \tilde \TT := \scal{\tilde \WW}.    \label{def:tTT}
\end{equ} 

\medskip
{\bf Symbols for negative renormalization}
\medskip

 Recall that, thanks to $\beta =1$, noise degree $\alpha -1 \in (-1,0)$, no terms $X, X^2$ or $\II(), ... $ arise as symbol in $\WW_- := \{\tau \in \WW \mid |\tau| < 0 \}$. (As a consequence, replacing $\WW$ by $\WW_{\Hai}, \tilde \WW$ or $\WW_\mathrm{\BHZ}$ in the definition of the negative symbols makes no difference.) In particular,   
$$
\WW_- = \{ \Xi_i, \II(\Xi_i)\Xi_j,  ... , \II(\Xi_i)\II(\Xi_j)\Xi_k,  ... \}.
$$
(where $\WW_-$ ``ends'' right before the element $1$ in \eqref{equ:WWap} above) contains no powers of $ X $, (hence no need to introduce ``$ \tilde{\WW}_- $''). As previously defined (see~\eqref{def:Tm}), we have
$$ \TT_- = \text{free unital commutative algebra generated by }    \WW_-.$$
For instance, writing $\bullet$ for the (free, commutative) product in $\TT_-$,
\[
2 \Xi_i - \frac{1}{3} \Xi_i \bullet \Xi_j + \II(\Xi_i)\Xi_j \bullet ( \II(\Xi_i)\II(\Xi_j)\Xi_k)^{\bullet 2} \quad \in \TT_-.
\]
Interpreting $\bullet$ as the {\it forest product}, elements in $\TT_-$ can then be represented as linear combinations of forests, such as
$$
2 \begin{tikzpicture}[scale=0.2,baseline=0.1cm]
        \node at (0,0)  [fill,circle,scale=0.22,blue,label= {[label distance=-0.2em]below: \scriptsize  $  $} ] (root) {};
        \node at (0,3)  [fill,circle,scale=0.22,blue,label= {[label distance=-0.2em]right: \scriptsize  $  $} ] (center) {};
     \draw[kernel1,blue,thin] (root) to
     node [sloped,below] {\small }     (center);
     \node at (0,1.5)  [fill=white,label={[label distance=0em]center: \scriptsize  $ i $} ] (right1) {};
     \end{tikzpicture}   
-
\frac{1}{3} \begin{tikzpicture}[scale=0.2,baseline=0.1cm]
        \node at (0,0)  [fill,circle,scale=0.22,blue,label= {[label distance=-0.2em]below: \scriptsize  $  $} ] (root) {};
        \node at (0,3)  [fill,circle,scale=0.22,blue,label= {[label distance=-0.2em]right: \scriptsize  $  $} ] (center) {};
     \draw[kernel1,blue,thin] (root) to
     node [sloped,below] {\small }     (center);
     \node at (0,1.5)  [fill=white,label={[label distance=0em]center: \scriptsize  $ i $} ] (right1) {};
     \end{tikzpicture}   
     \begin{tikzpicture}[scale=0.2,baseline=0.1cm]
        \node at (0,0)  [fill,circle,scale=0.22,blue,label= {[label distance=-0.2em]below: \scriptsize  $  $} ] (root) {};
        \node at (0,3)  [fill,circle,scale=0.22,blue,label= {[label distance=-0.2em]right: \scriptsize  $  $} ] (center) {};
     \draw[kernel1,blue,thin] (root) to
     node [sloped,below] {\small }     (center);
     \node at (0,1.5)  [fill=white,label={[label distance=0em]center: \scriptsize  $ j $} ] (right1) {};
     \end{tikzpicture}  
+ \begin{tikzpicture}[scale=0.2,baseline=0.1cm]
        \node at (0,0)  [fill,circle,scale=0.22,blue,label= {[label distance=-0.2em]right: \scriptsize  $  $} ] (center) {};
         \node at (-1.5,3)  [fill,circle,scale=0.22,blue,label= {[label distance=-0.2em]left: \tiny  $  $} ] (centerl) {};
         \node at (1.5,3)  [fill,circle,scale=0.22,blue,label= {[label distance=-0.2em]left: \tiny  $  $} ] (centerr) {};
         \node at (-1.5,6)  [fill,circle,scale=0.22,blue,label= {[label distance=-0.2em]right: \tiny  $  $} ] (centerlc) {};
     \draw[kernel1,blue,thin] (center) to
     node [sloped,below] {\small }     (centerl);
     \draw[kernel1,blue,thin] (center) to
     node [sloped,below] {\small }     (centerr);
     \draw[kernel1,blue,thin] (centerl) to
     node [sloped,below] {\small }     (centerlc);
      \node at (-1.5,4.5)  [fill=white,label={[label distance=0em]center: \tiny  $ i $} ] (right1) {};
       \node at (0.75,1.5)  [fill=white,label={[label distance=0em]center: \tiny  $ j $} ] (right11) {};
     \end{tikzpicture}
     \begin{tikzpicture}[scale=0.2,baseline=0.1cm]
        \node at (0,0)  [fill,circle,scale=0.22,blue,label= {[label distance=-0.2em]right: \scriptsize  $  $} ] (center) {};
        \node at (0,3)  [fill,circle,scale=0.22,blue,label= {[label distance=-0.2em]right: \scriptsize  $  $} ] (centerc) {};
         \node at (-2,3)  [fill,circle,scale=0.22,blue,label= {[label distance=-0.2em]left: \tiny  $  $} ] (centerl) {};
         \node at (2,3)  [fill,circle,scale=0.22,blue,label= {[label distance=-0.2em]right: \tiny  $  $} ] (centerr) {};
         \node at (-2,6)  [fill,circle,scale=0.22,blue,label= {[label distance=-0.2em]right: \tiny  $  $} ] (centerlc) {};
         \node at (0,6)  [fill,circle,scale=0.22,blue,label= {[label distance=-0.2em]right: \tiny  $  $} ] (centercc) {};
     \draw[kernel1,blue,thin] (center) to
     node [sloped,below] {\small }     (centerl);
     \draw[kernel1,blue,thin] (center) to
     node [sloped,below] {\small }     (centerr);
     \draw[kernel1,blue,thin] (center) to
     node [sloped,below] {\small }     (centerc);
     \draw[kernel1,blue,thin] (centerc) to
     node [sloped,below] {\small }     (centercc);
     \draw[kernel1,blue,thin] (centerl) to
     node [sloped,below] {\small }     (centerlc);
      \node at (-2,4.5)  [fill=white,label={[label distance=0em]center: \tiny  $ i $} ] (right1) {};
       \node at (0,4.5)  [fill=white,label={[label distance=0em]center: \tiny  $ j $} ] (right11) {};
       \node at (1,1.5)  [fill=white,label={[label distance=0em]center: \tiny  $ k $} ] (right11) {};
     \end{tikzpicture}
     \begin{tikzpicture}[scale=0.2,baseline=0.1cm]
        \node at (0,0)  [fill,circle,scale=0.22,blue,label= {[label distance=-0.2em]right: \scriptsize  $  $} ] (center) {};
        \node at (0,3)  [fill,circle,scale=0.22,blue,label= {[label distance=-0.2em]right: \scriptsize  $  $} ] (centerc) {};
         \node at (-2,3)  [fill,circle,scale=0.22,blue,label= {[label distance=-0.2em]left: \tiny  $  $} ] (centerl) {};
         \node at (2,3)  [fill,circle,scale=0.22,blue,label= {[label distance=-0.2em]right: \tiny  $  $} ] (centerr) {};
         \node at (-2,6)  [fill,circle,scale=0.22,blue,label= {[label distance=-0.2em]right: \tiny  $  $} ] (centerlc) {};
         \node at (0,6)  [fill,circle,scale=0.22,blue,label= {[label distance=-0.2em]right: \tiny  $  $} ] (centercc) {};
     \draw[kernel1,blue,thin] (center) to
     node [sloped,below] {\small }     (centerl);
     \draw[kernel1,blue,thin] (center) to
     node [sloped,below] {\small }     (centerr);
     \draw[kernel1,blue,thin] (center) to
     node [sloped,below] {\small }     (centerc);
     \draw[kernel1,blue,thin] (centerc) to
     node [sloped,below] {\small }     (centercc);
     \draw[kernel1,blue,thin] (centerl) to
     node [sloped,below] {\small }     (centerlc);
      \node at (-2,4.5)  [fill=white,label={[label distance=0em]center: \tiny  $ i $} ] (right1) {};
       \node at (0,4.5)  [fill=white,label={[label distance=0em]center: \tiny  $ j $} ] (right11) {};
       \node at (1,1.5)  [fill=white,label={[label distance=0em]center: \tiny  $ k $} ] (right11) {};
     \end{tikzpicture} 
$$ 
One can readily verify that $\Deltam : \TT \rightarrow \TT_- \otimes \TT $ restricted to $ \tilde \TT$ maps $\tilde{\TT} \rightarrow \TT_- \otimes \tilde{\TT}  $, also denoted by 
$ \Deltam  $ so that $ (\tilde \TT,\Deltam) $ is a subcomodule of $ (\TT,\Deltam)$. 

\medskip

{\bf Symbols for positive renormalization and $\TT_+$.}
\medskip

Recall that $\TT_+$ was generated, as a free commutative algebra, by 
\[
\WW_+ := \{X\} \cup \{\mathcal{J}_k\tau \mid \tau \in \WW, |\tau|+1-k > 0 \}.
\]
Writing $\mathcal{J} \equiv \mathcal{J}_0$ as usual, we define a subset $\tilde \WW_+ \subset \WW_+$ as follows
\begin{align}
& \tilde \WW_+  := \{ \mathcal{J} \tau \mid \tau \in \tilde \WW  \} \\ \nonumber
& = \{ 1, \JJ(\Xi_i),  \JJ( \II(\Xi_i) \Xi_j ), \JJ (  \II(  \II(\Xi_i) \Xi_j ) \Xi_k ),  \JJ (\II(\Xi_i)\II(\Xi_j) \Xi_k), ... ,\JJ (\II(\Xi_i)\II(\Xi_j)), ... \} 
\end{align}
with degrees $0, \alpha, 2 \alpha, 3 \alpha, 3 \alpha, ..., 2 \alpha + 1,...  $ here. 

Recall that elements in $\WW_+$ can be represented by {\it elementary} trees, in the sense that - disregarding the trivial (empty) tree $1$ - only one edge departs from the root. The same is true for elements in $\tilde \WW_+$. Set

$$\tilde \TT_+ := \text{free unital commutative algebra generated by }    \tilde \WW_+.$$

For example, writing $\tau_1 \tau_2$ for the (free, commutative) product of $\tau_1, \tau_2 \in \tilde \TT_+$, an example of an element in this space would be 
\[
      \JJ ( \II (\Xi_i) \Xi_j ) + 
      \JJ (\II() )  \JJ (1) +
     3 \ \JJ(\Xi_i)  \JJ(\Xi_j)  + \JJ(  \II(\Xi_i) \Xi_j ) \JJ(  \II(\Xi_k) \Xi_l) \quad \in \tilde \TT_+.
\] 
Fortunately, every such element can still be represented as a tree; it suffices to interpret the free product in $\TT_+$ as the ``root-joining'' product
(which is possible since all constituting trees are elementary). The (abstract) unit element $1 \in \TT_+$ is then indeed given by the (trivial) tree ${ \bullet} \leftrightarrow X^0$, where we recall our convention to drop the node decoration ``$0$''. For instance, the above element becomes\footnote{Remark that $\JJ (1)$, which corresponds to the right branch of the second term, could also have been written as $\JJ()$, reflecting our convention to drop the decoration $0$ from nodes (here: $1 \equiv X^0$). By the same logic, we could also write $\II()$, one of the symbols arising in $\WW$, as $\II(1)$.}

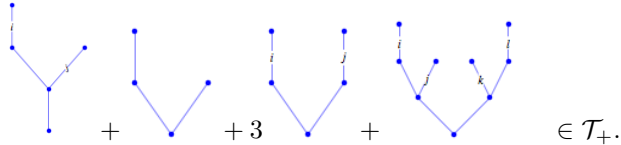
\[
\begin{tikzpicture}[scale=0.2,baseline=0.1cm]
       \node at (0,0)  [fill,circle,scale=0.22,blue,label= {[label distance=-0.2em]right: \scriptsize  $  $} ] (root) {};
        \node at (0,3)  [fill,circle,scale=0.22,blue,label= {[label distance=-0.2em]right: \scriptsize  $  $} ] (center) {};
         \node at (-1.5,6)  [fill,circle,scale=0.22,blue,label= {[label distance=-0.2em]left: \tiny  $  $} ] (centerl) {};
         \node at (1.5,6)  [fill,circle,scale=0.22,blue,label= {[label distance=-0.2em]left: \tiny  $  $} ] (centerr) {};
         \node at (-1.5,9)  [fill,circle,scale=0.22,blue,label= {[label distance=-0.2em]right: \tiny  $  $} ] (centerlc) {};
     \draw[kernel1,blue,thin] (root) to
     node [sloped,below] {\small }     (center);
     \draw[kernel1,blue,thin] (center) to
     node [sloped,below] {\small }     (centerl);
     \draw[kernel1,blue,thin] (center) to
     node [sloped,below] {\small }     (centerr);
     \draw[kernel1,blue,thin] (centerl) to
     node [sloped,below] {\small }     (centerlc);
      \node at (-1.5,7.5)  [fill=white,label={[label distance=0em]center: \tiny  $ i $} ] (right1) {};
       \node at (0.75,4.5)  [fill=white,label={[label distance=0em]center: \tiny  $ j $} ] (right11) {};
     \end{tikzpicture}
  +
\begin{tikzpicture}[scale=0.2,baseline=0.1cm]
        \node at (0,0)  [fill,circle,scale=0.22,blue,label= {[label distance=-0.2em]right: \scriptsize  $  $} ] (center) {};
         \node at (-1.5,3)  [fill,circle,scale=0.22,blue,label= {[label distance=-0.2em]left: \tiny  $  $} ] (centerl) {};
         \node at (1.5,3)  [fill,circle,scale=0.22,blue,label= {[label distance=-0.2em]left: \tiny  $  $} ] (centerr) {};
         \node at (-1.5,6)  [fill,circle,scale=0.22,blue,label= {[label distance=-0.2em]right: \tiny  $  $} ] (centerlc) {};
     \draw[kernel1,blue,thin] (center) to
     node [sloped,below] {\small }     (centerl);
     \draw[kernel1,blue,thin] (center) to
     node [sloped,below] {\small }     (centerr);
     \draw[kernel1,blue,thin] (centerl) to
     node [sloped,below] {\small }     (centerlc);
     \end{tikzpicture}   + 
   3 \begin{tikzpicture}[scale=0.2,baseline=0.1cm]
        \node at (0,0)  [fill,circle,scale=0.22,blue,label= {[label distance=-0.2em]right: \scriptsize  $  $} ] (center) {};
         \node at (-1.5,3)  [fill,circle,scale=0.22,blue,label= {[label distance=-0.2em]left: \tiny  $  $} ] (centerl) {};
         \node at (1.5,3)  [fill,circle,scale=0.22,blue,label= {[label distance=-0.2em]left: \tiny  $  $} ] (centerr) {};
         \node at (-1.5,6)  [fill,circle,scale=0.22,blue,label= {[label distance=-0.2em]right: \tiny  $  $} ] (centerlc) {};
         \node at (1.5,6)  [fill,circle,scale=0.22,blue,label= {[label distance=-0.2em]right: \tiny  $  $} ] (centerrc) {};
     \draw[kernel1,blue,thin] (center) to
     node [sloped,below] {\small }     (centerl);
     \draw[kernel1,blue,thin] (center) to
     node [sloped,below] {\small }     (centerr);
     \draw[kernel1,blue,thin] (centerr) to
     node [sloped,below] {\small }     (centerrc);
     \draw[kernel1,blue,thin] (centerl) to
     node [sloped,below] {\small }     (centerlc);
      \node at (-1.5,4.5)  [fill=white,label={[label distance=0em]center: \tiny  $ i $} ] (right1) {};
       \node at (1.5,4.5)  [fill=white,label={[label distance=0em]center: \tiny  $ j $} ] (right11) {};
     \end{tikzpicture}  +  \begin{tikzpicture}[scale=0.2,baseline=0.1cm]
        \node at (0,0)  [fill,circle,scale=0.22,blue,label= {[label distance=-0.2em]right: \scriptsize  $  $} ] (center) {};
         \node at (-3,3)  [fill,circle,scale=0.22,blue,label= {[label distance=-0.2em]left: \tiny  $  $} ] (centerl) {};
         \node at (3,3)  [fill,circle,scale=0.22,blue,label= {[label distance=-0.2em]left: \tiny  $  $} ] (centerr) {};
         \node at (-1.5,6)  [fill,circle,scale=0.22,blue,label= {[label distance=-0.2em]left: \tiny  $  $} ] (centerlr) {};
         \node at (-4.5,6)  [fill,circle,scale=0.22,blue,label= {[label distance=-0.2em]left: \tiny  $  $} ] (centerll) {};
          \node at (-4.5,9)  [fill,circle,scale=0.22,blue,label= {[label distance=-0.2em]left: \tiny  $  $} ] (centerllc) {};
   \node at (1.5,6)  [fill,circle,scale=0.22,blue,label= {[label distance=-0.2em]left: \tiny  $  $} ] (centerrl) {};
         \node at (4.5,6)  [fill,circle,scale=0.22,blue,label= {[label distance=-0.2em]left: \tiny  $  $} ] (centerrr) {};
          \node at (4.5,9)  [fill,circle,scale=0.22,blue,label= {[label distance=-0.2em]left: \tiny  $  $} ] (centerrrc) {};
     \draw[kernel1,blue,thin] (center) to
     node [sloped,below] {\small }     (centerl);
     \draw[kernel1,blue,thin] (center) to
     node [sloped,below] {\small }     (centerr);
     \draw[kernel1,blue,thin] (centerr) to
     node [sloped,below] {\small }     (centerrl);
     \draw[kernel1,blue,thin] (centerr) to
     node [sloped,below] {\small }     (centerrr);
    \draw[kernel1,blue,thin] (centerrr) to
     node [sloped,below] {\small }     (centerrrc);
     \draw[kernel1,blue,thin] (centerl) to
     node [sloped,below] {\small }     (centerll);
     \draw[kernel1,blue,thin] (centerl) to
     node [sloped,below] {\small }     (centerlr);
    \draw[kernel1,blue,thin] (centerll) to
     node [sloped,below] {\small }     (centerllc);
      \node at (-4.5,7.5)  [fill=white,label={[label distance=0em]center: \tiny  $ i $} ] (right1) {};
       \node at (-2.25,4.5)  [fill=white,label={[label distance=0em]center: \tiny  $ j $} ] (right11) {};
        \node at (2.25,4.5)  [fill=white,label={[label distance=0em]center: \tiny  $ k $} ] (right11) {};
        \node at (4.5,7.5)  [fill=white,label={[label distance=0em]center: \tiny  $ l $} ] (right1) {};
     \end{tikzpicture}        \quad \in \TT_+.
\]

\begin{remark}
Though we used the same formalism to draw trees as in the case of $\tilde \WW$ above, the interpretation here is slighly different in that all root-touching edges refer to $\JJ$ rather than $\II$. As mentioned before, in~\cite{BHZ16}, this is indicated by a blue colouring of the root.
\end{remark}

As before, we define a coaction of $\tilde \TT_+$ on $\tilde \TT$ (which we again denote $\Deltap : \tilde \TT \rightarrow \tilde \TT\otimes \tilde \TT_+$) by~\eqref{e:deltap1},~\eqref{eq:DeltapMult}, and~\eqref{e:recDeltaNew} as well as a coproduct $\Deltap : \tilde \TT_+ \rightarrow \tilde \TT_+ \otimes \tilde \TT_+$ defined in the same way, but with $\II$ changed to $\JJ$ in~\eqref{e:recDeltaNew}.
(In contrast to the case of $\Deltam$ discussed above,  it is not the case that $ (\tilde \TT,\Deltap) $ is a subcomodule of $ (\TT,\Deltap)$.)

We note already that $(\tilde \TT_+,\Deltap)$ is isomorphic to the Connes-Kreimer Hopf algebra $\HH$ arising from the identifications laid out in the following subsection (and which will be used crucially in the proof of the upcoming Proposition~\ref{prop:rosa}).

\subsection{Link with translation of rough paths}

\subsubsection{Identification of spaces} 

{\it We now give a precise description the map $\Deltam$ in our context as well as its connection to the map $\extract$ from Section~\ref{subsec:DualMapBranched}. To do so, we first need to introduce several identification of vector spaces and algebras, as well as appropriately identify branched rough paths as models on a regularity structure.}

\medskip

Recall the space $\BPrimal = \BPrimal ( \bullet_0, ... , \bullet_d)$ from Section~\ref{sec:Branched} spanned by labelled forests with label set $\{0,1,\ldots, d\}$. Consider now the enlarged vector space 
\begin{equation} 
\tilde \HH :=   \HH \oplus \BPrimal \Xi_1 \oplus ....  \oplus \BPrimal \Xi_d .       \label{def:tHH}
\end{equation}
driven by branched rough paths~\cite{Preiss16}).
With $\tilde \TT$ as defined in~\eqref{def:tTT}, and in particular with noise types $\Xi_1, ... ,\Xi_d$, we then have a vector space isomorphism 
\[
\tilde \BPrimal \leftrightarrow \tilde \TT
\]
obtained by adding an extra edge to indicate a noise $\Xi_i$, $i \neq 0$, and by ``forgetting'' the label $0$ (which is equivalent to setting the noise $\Xi_0$ to the constant $1$). For example,
\begin{equs}
\tikzexternaldisable  \begin{tikzpicture}[scale=0.2,baseline=0.1cm]
        \node at (0,0)  [dot,label= {[label distance=-0.2em]below: \scriptsize  $ 1 $} ] (root) {};
         \node at (1,2)  [dot,label={[label distance=-0.2em]above: \scriptsize  $ 0 $}] (right) {};
         \node at (-1,2)  [dot,label={[label distance=-0.2em]above: \scriptsize  $ 2 $} ] (left) {};
            \draw[kernel1] (right) to
     node [sloped,below] {\small }     (root); \draw[kernel1] (left) to
     node [sloped,below] {\small }     (root);
     \end{tikzpicture} \tikzexternaldisable
     \leftrightarrow {\II\left[\II(\Xi_2) \II(1) \Xi_1\right]}
     = \begin{tikzpicture}[scale=0.2,baseline=0.1cm]
        \node at (0,0)  [fill,circle,scale=0.22,blue,label= {[label distance=-0.2em]right: \tiny  $  $} ] (root) {};
        \node at (0,3)  [fill,circle,scale=0.22,blue,label= {[label distance=-0.2em]right: \scriptsize  $  $} ] (center) {};
        \node at (0,6)  [fill,circle,scale=0.22,blue,label= {[label distance=-0.2em]right: \scriptsize  $  $} ] (centerc) {};
         \node at (-2,6)  [fill,circle,scale=0.22,blue,label= {[label distance=-0.2em]left: \tiny  $  $} ] (centerl) {};
         \node at (2,6)  [fill,circle,scale=0.22,blue,label= {[label distance=-0.2em]right: \tiny  $  $} ] (centerr) {};
         \node at (-2,9)  [fill,circle,scale=0.22,blue,label= {[label distance=-0.2em]right: \tiny  $  $} ] (centerlc) {};
     \draw[kernel1,blue,thin] (root) to
     node [sloped,below] {\small }     (center);
     \draw[kernel1,blue,thin] (center) to
     node [sloped,below] {\small }     (centerl);
     \draw[kernel1,blue,thin] (center) to
     node [sloped,below] {\small }     (centerr);
     \draw[kernel1,blue,thin] (center) to
     node [sloped,below] {\small }     (centerc);
     \draw[kernel1,blue,thin] (centerl) to
     node [sloped,below] {\small }     (centerlc);
      \node at (-2,7.5)  [fill=white,label={[label distance=0em]center: \tiny  $ 2 $} ] (right1) {};
       \node at (1,4.5)  [fill=white,label={[label distance=0em]center: \tiny  $ 1 $} ] (right11) {};
     \end{tikzpicture} 
\end{equs}

\begin{equs}
\tikzexternaldisable  \begin{tikzpicture}[scale=0.2,baseline=0.1cm]
        \node at (0,0)  [dot,label= {[label distance=-0.2em]below: \scriptsize  $ 2 $} ] (root) {};
         \node at (1,2)  [dot,label={[label distance=-0.2em]above: \scriptsize  $ 1 $}] (right) {};
         \node at (-1,2)  [dot,label={[label distance=-0.2em]above: \scriptsize  $ 0 $} ] (left) {};
            \draw[kernel1] (right) to
     node [sloped,below] {\small }     (root); \draw[kernel1] (left) to
     node [sloped,below] {\small }     (root);
     \end{tikzpicture} \tikzexternaldisable
     \tikzexternaldisable \begin{tikzpicture}[scale=0.2,baseline=0.1cm]
        \node at (0,0)  [dot,label= {[label distance=-0.2em]below: \scriptsize  $ 0 $} ] (root) {};
         \node at (0,2)  [dot,label={[label distance=-0.2em]above: \scriptsize  $ 3$}] (right) {};
            \draw[kernel1] (right) to
     node [sloped,below] {\small }     (root);
     \end{tikzpicture} \tikzexternaldisable
     \Xi_4
     \leftrightarrow \II\left[\II(1)\II(\Xi_1)\Xi_2\right]\II\left[\II(\Xi_3)\right]\Xi_4
     =  \begin{tikzpicture}[scale=0.2,baseline=0.1cm]
        \node at (0,0)  [fill,circle,scale=0.22,blue,label= {[label distance=-0.2em]right: \scriptsize  $  $} ] (center) {};
        \node at (0,3)  [fill,circle,scale=0.22,blue,label= {[label distance=-0.2em]right: \scriptsize  $  $} ] (centerc) {};
        \node at (0,6)  [fill,circle,scale=0.22,blue,label= {[label distance=-0.2em]right: \scriptsize  $  $} ] (centercc) {};
        \node at (0,9)  [fill,circle,scale=0.22,blue,label= {[label distance=-0.2em]right: \scriptsize  $  $} ] (centerccc) {};
         \node at (-4,3)  [fill,circle,scale=0.22,blue,label= {[label distance=-0.2em]left: \tiny  $  $} ] (centerl) {};
         \node at (-2,6)  [fill,circle,scale=0.22,blue,label= {[label distance=-0.2em]left: \tiny  $  $} ] (centerlr) {};
         \node at (-6,6)  [fill,circle,scale=0.22,blue,label= {[label distance=-0.2em]left: \tiny  $  $} ] (centerll) {};
         \node at (-4,6)  [fill,circle,scale=0.22,blue,label= {[label distance=-0.2em]left: \tiny  $  $} ] (centerlc) {};
         \node at (-4,9)  [fill,circle,scale=0.22,blue,label= {[label distance=-0.2em]left: \tiny  $  $} ] (centerlcc) {};
         \node at (4,3)  [fill,circle,scale=0.22,blue,label= {[label distance=-0.2em]left: \tiny  $  $} ] (centerr) {};
     \draw[kernel1,blue,thin] (center) to
     node [sloped,below] {\small }     (centerc);
     \draw[kernel1,blue,thin] (centerc) to
     node [sloped,below] {\small }     (centercc);
     \draw[kernel1,blue,thin] (centercc) to
     node [sloped,below] {\small }     (centerccc);
     \draw[kernel1,blue,thin] (center) to
     node [sloped,below] {\small }     (centerr);
     \draw[kernel1,blue,thin] (center) to
     node [sloped,below] {\small }     (centerl);
     \draw[kernel1,blue,thin] (centerl) to
     node [sloped,below] {\small }     (centerlr);
     \draw[kernel1,blue,thin] (centerl) to
     node [sloped,below] {\small }     (centerll);
     \draw[kernel1,blue,thin] (centerl) to
     node [sloped,below] {\small }     (centerlc);
     \draw[kernel1,blue,thin] (centerlc) to
     node [sloped,below] {\small }     (centerlcc);
      \node at (-4,7.5)  [fill=white,label={[label distance=0em]center: \tiny  $ 1 $} ] (right1) {};
       \node at (-3,4.5)  [fill=white,label={[label distance=0em]center: \tiny  $ 2 $} ] (right11) {};
        \node at (0,7.5)  [fill=white,label={[label distance=0em]center: \tiny  $ 3 $} ] (right11) {};
        \node at (2,1.5)  [fill=white,label={[label distance=0em]center: \tiny  $ 4 $} ] (right1) {};
     \end{tikzpicture}.
\end{equs}

Recall that $\BB = \BB (\bullet_0, ... ,\bullet_d)$  denotes the subspace of $\HH$ spanned by trees, and define
$$\BB_- = \BB_-(\bullet_1, ... ,\bullet_d) \subset \BB \subset \HH $$
as the subspace of $\BB$ spanned by trees with no label $0$ and with at most $\floor{1/\alpha}$ nodes. Observe that there is a canonical vector space isomorphism
\begin{equation}
\label{eq:BFIsom}
\phi: \BB_- \rightarrow \langle \WW_- \rangle    \ \ \subset \ \ \BPrimal \Xi_1 \oplus ....  \oplus \BPrimal \Xi_d  \ \ \subset \ \ \tilde \HH,
\end{equation}
where we have used the identification $\tilde \HH \leftrightarrow \tilde \TT \supset \scal{\WW_-}$ for the first inclusion (and both inclusions being strict: for the first, just consider the element 
$
\tikzexternaldisable \begin{tikzpicture}[scale=0.2,baseline=0.1cm]
        \node at (0,0)  [dot,label= {[label distance=-0.2em]below: \scriptsize  $ 0 $} ] (root) {};
         \node at (0,2)  [dot,label={[label distance=-0.2em]above: \scriptsize  $ 3$}] (right) {};
            \draw[kernel1] (right) to
     node [sloped,below] {\small }     (root);
     \end{tikzpicture} \tikzexternaldisable
     \Xi_1 \notin \langle \WW_- \rangle
$).
We denote this isomorphism also by
$$
\tau \mapsto \dot \tau := \phi (\tau).
$$          
For example,
\begin{equs}
\phi :
\tikzexternaldisable  \begin{tikzpicture}[scale=0.2,baseline=0.1cm]
        \node at (0,0)  [dot,label= {[label distance=-0.2em]below: \scriptsize  $ 3 $} ] (root) {};
         \node at (1,2)  [dot,label={[label distance=-0.2em]above: \scriptsize  $ 2 $}] (right) {};
         \node at (-1,2)  [dot,label={[label distance=-0.2em]above: \scriptsize  $ 1 $} ] (left) {};
            \draw[kernel1] (right) to
     node [sloped,below] {\small }     (root); \draw[kernel1] (left) to
     node [sloped,below] {\small }     (root);
     \end{tikzpicture} \tikzexternaldisable
\mapsto \II(\Xi_1)\II(\Xi_2)\Xi_3,
\end{equs}
where we assume $\alpha \in (0,1/3)$ so the tree appearing on the left is indeed an element in $\BB_-$. Correspondingly, the symbol on the right has negative degree as an element of $\WW$, hence is an element of $\WW_-$.

Write $\BB_-^*$ for the dual of the (finite-dimensional) vector space $\BB_-$. Of course, $\BB_-^* \cong \BB_-$ which allows us to identify $\BB_-^*$ with $\langle \WW_- \rangle$.
Recall that  $(\TT_-,\bullet)$ was defined as the free unital commutative algebra generated by $\WW_-$, and that $\GG_-\subset \TT_{-}^{* }$  denotes the group of characters on $\TT_-$. By definition of $\TT_-$, we then have a bijection 
\begin{equ} 
\BB_-^* \leftrightarrow \GG_-.   \label{id:BBmGm}
\end{equ} To be fully explicit about this, recall that
$$
      \TT_{-} = \langle \dot \tau_1 \bullet .... \bullet \dot \tau_n : \dot \tau_i \in \WW_-, \ n=1,2,... \rangle,
$$
so writing $\tau_i = \phi^{-1} (\dot \tau_i) \in \BB_-$, we have that associated to $v \in \BB_-^*$ the character $\ell \in \GG_-$ given explicitly by the formula
$$
    \ell ( \dot \tau_1 \cdot .... \cdot \dot \tau_n ) =  \ell ( \dot \tau_1) ...  \ell (\dot \tau_n ) = \langle v, \tau_1 \rangle ... \langle v, \tau_n \rangle.
$$   
Define now $$(\HH_-,\cdot)$$ as the free commutative algebra generated by the subspace $ \BB_- $ of $\tilde \HH$ (remark that the product in $\HH_-$ has nothing to do with the product in $\HH$ itself), so that there is an algebra isomorphism
\[
\HH_- \leftrightarrow \TT_-.
\]
A typical element of $ \HH_- $ looks like:
\begin{equs}
\tikzexternaldisable 
\begin{tikzpicture}[scale=0.2,baseline=-0.1cm]
        \node at (0,0)  [dot,label= {[label distance=-0.2em]below: \scriptsize  $ 2 $} ] (root) {};
     \end{tikzpicture} \tikzexternaldisable  \Xi_1 +\Xi_2 + \Xi_2 \cdot
\tikzexternaldisable  \begin{tikzpicture}[scale=0.2,baseline=0.1cm]
        \node at (0,0)  [dot,label= {[label distance=-0.2em]below: \scriptsize  $ 1 $} ] (root) {};
         \node at (1,2)  [dot,label={[label distance=-0.2em]above: \scriptsize  $ 2 $}] (right) {};
         \node at (-1,2)  [dot,label={[label distance=-0.2em]above: \scriptsize  $ 3 $} ] (left) {};
            \draw[kernel1] (right) to
     node [sloped,below] {\small }     (root); \draw[kernel1] (left) to
     node [sloped,below] {\small }     (root);
     \end{tikzpicture} \tikzexternaldisable\Xi_3,
     \end{equs}
  whereas one has $ \tikzexternaldisable 
\begin{tikzpicture}[scale=0.2,baseline=-0.1cm]
        \node at (0,0)  [dot,label= {[label distance=-0.2em]below: \scriptsize  $ 2 $} ] (root) {};
     \end{tikzpicture} \tikzexternaldisable  \notin \HH_- $.

Note that we can also make the identification of algebras 
$$\HH \leftrightarrow \tilde \TT_+.$$ 
For instance, using the bracket notation,
\[    [\bullet_0]_{\bullet_0} \bullet_0 + [\bullet_i]_{\bullet_j} [\bullet_k]_{\bullet_l} \leftrightarrow   \JJ (\II() )  \JJ (1) + \JJ(  \II(\Xi_i) \Xi_j ) \JJ(  \II(\Xi_k) \Xi_l) \quad \in \tilde \TT_+.
\] 
We denote by $\tilde \GG_+ \subset \tilde \TT_+^*$ the characters on $\tilde \TT_+$ and note that there is also a bijection $\GG \leftrightarrow \tilde\GG_+$, where we recall that $\GG\subset\HH^*$ is the Butcher group over $\R^{1+d}$, i.e.,
the set of characters on $\HH$.

To summarise, we have the following identifications in place
\begin{align*}
\tilde \HH &\leftrightarrow \tilde \TT, \\
\HH_- &\leftrightarrow \TT_-, \\
\HH &\leftrightarrow \tilde \TT_+, \\
\BB_-^* &\leftrightarrow \langle \WW_- \rangle \leftrightarrow \GG_- \subset \TT_-^* \\
\GG &\leftrightarrow \tilde \GG_+ \subset \tilde \TT_+^*.
\end{align*}

\subsubsection{Renormalization as rough path translations} 
It now only remains to identify (a family of) branched rough paths with a class of models on a suitable regularity structure. Define the index set $A := \{0\} \cup \alpha\mathbb{N}\cup(\alpha\mathbb{N}-1)$. Recall that the action of $g \in \tilde \GG_+$ on $\tilde \TT$ is given exactly as before by
\[
\Gamma_g \tau = (\id \otimes g)\Deltap \tau, \; \;\text{for all } \tau \in \tilde \TT.
\]
Note that $\Gamma_g$ indeed maps $\tilde\TT$ to itself due to the definition of $\tilde \GG_+$. Note further that
$\Gamma_g \Gamma_h$ (as a composition of linear maps) is exactly $\Gamma_{g \circ h}$ (with $\circ$ the product in $\tilde \GG_+$ given as the dual of $\Deltap$), and so
$$
     G := \{ \Gamma_g:  g \in (\tilde \GG_+, \circ) \}. 
$$
is indeed a group of endomorphisms of $\tilde\TT$.

Recall now the definition of a regularity structure from~\cite{Hairer14} Definition~2.1.

\begin{lemma} 
The triplet $(A,\tilde \TT, G)$ is a regularity structure.
\end{lemma}

\begin{proof}
The only non-trivial property to check is that for all $\tau \in \tilde \TT$ of degree $\alpha \in A$ and $\Gamma \in G$, $\Gamma \tau - \tau$ is a linear combination of terms of degree strictly less than $\alpha$, which in turn is a direct consequence of the definition of $\Deltap : \tilde \TT \rightarrow \tilde \TT \otimes \tilde \TT_+$ from~\eqref{e:recDeltaNew} (see end of Section~\ref{subsubsec:RDEs}).
\end{proof}

Recall also the definition of a model on a regularity structure (see~\cite{Hairer14} Definition~2.17). Let $\mathscr{M}_{[0,T]}$ denote the set of all models $(\Pi,\Gamma)$ for $(A,\tilde\TT, G)$ on $\mathbb{R}$ such that
\begin{enumerate}[label={(\roman*)}]
 \item\label{point:M1} $\Pi_t 1$ is the constant function $1$ for all $t\in\mathbb{R}$,
 \item\label{point:M2} $\Gamma_{st}=\operatorname{id}$ for $s,t\in(-\infty,0]$ and for $s,t\in[T,\infty)$,
 \item\label{point:M3} $(\Pi_t\II y)'=\Pi_t y$ for all $t\in\mathbb{R}$ and $y\in\tilde \TT$. (Here $(..)'$ denotes the Schwartz derivative.).
\end{enumerate}

\medskip
\medskip

On the other hand, let $\mathscr{R}^\alpha_{[0,T]}$ be the set of all $(1+d)$-dimensional $\alpha$-H\"{o}lder branched rough paths $\mathbf{X}:\,[0,T]^2\rightarrow \GG$ whose zeroth component is time, i.e., $\scal{\mathbf{X}_{s,t}, \bullet_0} = t-s$ and
\begin{equation}\label{eq:zeroTime}
\scal{\Xbf_{s,t}, [\tau]_{\bullet_0}} = \int_{s}^t \scal{\Xbf_{s,u}, \tau} du, \; \; \text{for all } \tau \in \HH, \; s,t \in [0,T].
\end{equation}
Observe that this condition necessarily implies that $\Xbf$ satisfies condition~\eqref{eq:mixedVar} from Theorem~\ref{thm:transRPs} (cf. Remark~\ref{remark:colifting}). Note that $\mathbf{X}_{s,t}$ can be identified with an element of $\tilde \GG_+$ due to the identification $\GG \leftrightarrow \tilde \GG_+$,.

Finally, observe that $\phi$ defined in~\eqref{eq:BFIsom} may be extended to a vector space isomorphism
\begin{equation} \label{phiIso}
 \phi:    \BB \leftrightarrow  \BPrimal \Xi_0  \oplus \BPrimal \Xi_1 \oplus ....  \oplus \BPrimal \Xi_d \cong \HH \oplus \BPrimal \Xi_1 \oplus ....  \oplus \BPrimal \Xi_d \equiv \tilde \HH
\end{equation} 
which maps a tree $\tau \in \BB$ into a forest $\phi(\tau) \equiv \dot \tau$, as illustrated in the following two examples:
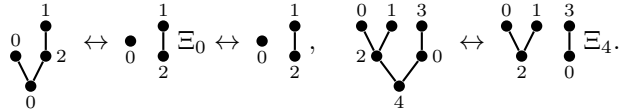
\begin{equs}
 \tikzexternaldisable \begin{tikzpicture}[scale=0.2,baseline=0.1cm]
       \node at (-2,0)  [dot,label= {[label distance=-0.2em]above: \scriptsize  $ 0 $} ] (left) {};
        \node at (0,0)  [dot,label= {[label distance=-0.2em]right: \scriptsize  $ 2 $} ] (root1) {};
         \node at (0,2)  [dot,label={[label distance=-0.2em]above: \scriptsize  $ 1$}] (right) {};
         \node at (-1,-2)  [dot,label={[label distance=-0.2em]below: \scriptsize  $ 0 $}] (root2) {};
            \draw[kernel1] (right) to
     node [sloped,below] {\small }     (root1);
     \draw[kernel1] (left) to
     node [sloped,below] {\small }     (root2);
     \draw[kernel1] (root1) to
     node [sloped,below] {\small }     (root2);
     \end{tikzpicture} \tikzexternaldisable
      \leftrightarrow
      \tikzexternaldisable 
\begin{tikzpicture}[scale=0.2,baseline=-0.1cm]
        \node at (0,0)  [dot,label= {[label distance=-0.2em]below: \scriptsize  $ 0 $} ] (root) {};
     \end{tikzpicture} \tikzexternaldisable 
  \tikzexternaldisable \begin{tikzpicture}[scale=0.2,baseline=0.1cm]
        \node at (0,0)  [dot,label= {[label distance=-0.2em]below: \scriptsize  $ 2 $} ] (root) {};
         \node at (0,2)  [dot,label={[label distance=-0.2em]above: \scriptsize  $ 1$}] (right) {};
            \draw[kernel1] (right) to
     node [sloped,below] {\small }     (root);
     \end{tikzpicture} \tikzexternaldisable
     \Xi_0 
     \leftrightarrow
     \tikzexternaldisable 
\begin{tikzpicture}[scale=0.2,baseline=-0.1cm]
        \node at (0,0)  [dot,label= {[label distance=-0.2em]below: \scriptsize  $ 0 $} ] (root) {};
     \end{tikzpicture} \tikzexternaldisable 
  \tikzexternaldisable \begin{tikzpicture}[scale=0.2,baseline=0.1cm]
        \node at (0,0)  [dot,label= {[label distance=-0.2em]below: \scriptsize  $ 2 $} ] (root) {};
         \node at (0,2)  [dot,label={[label distance=-0.2em]above: \scriptsize  $ 1$}] (right) {};
            \draw[kernel1] (right) to
     node [sloped,below] {\small }     (root);
     \end{tikzpicture} \tikzexternaldisable,
\quad\quad\quad
\tikzexternaldisable  \begin{tikzpicture}[scale=0.2,baseline=0.1cm]
        \node at (0,0)  [dot,label= {[label distance=-0.2em]left: \scriptsize  $ 2 $} ] (root1) {};
         \node at (1,2)  [dot,label={[label distance=-0.2em]above: \scriptsize  $ 1 $}] (right) {};
         \node at (-1,2)  [dot,label={[label distance=-0.2em]above: \scriptsize  $ 0 $} ] (left) {};
        \node at (3,0)  [dot,label= {[label distance=-0.2em]right: \scriptsize  $ 0 $} ] (root2) {};
         \node at (3,2)  [dot,label={[label distance=-0.2em]above: \scriptsize  $ 3$}] (up) {};
         [scale=0.2,baseline=0.1cm]
        \node at (1.5,-2)  [dot,label= {[label distance=-0.2em]below: \scriptsize  $ 4 $} ] (root3) {};
            \draw[kernel1] (right) to
     node [sloped,below] {\small }     (root1); \draw[kernel1] (left) to
     node [sloped,below] {\small }     (root1);
            \draw[kernel1] (up) to
     node [sloped,below] {\small }     (root2);
     \draw[kernel1] (root1) to
     node [sloped,below] {\small }     (root3);
     \draw[kernel1] (root2) to
     node [sloped,below] {\small }     (root3);
     \end{tikzpicture} \tikzexternaldisable
      \leftrightarrow
\tikzexternaldisable  \begin{tikzpicture}[scale=0.2,baseline=0.1cm]
        \node at (0,0)  [dot,label= {[label distance=-0.2em]below: \scriptsize  $ 2 $} ] (root) {};
         \node at (1,2)  [dot,label={[label distance=-0.2em]above: \scriptsize  $ 1 $}] (right) {};
         \node at (-1,2)  [dot,label={[label distance=-0.2em]above: \scriptsize  $ 0 $} ] (left) {};
            \draw[kernel1] (right) to
     node [sloped,below] {\small }     (root); \draw[kernel1] (left) to
     node [sloped,below] {\small }     (root);
     \end{tikzpicture} \tikzexternaldisable
     \tikzexternaldisable \begin{tikzpicture}[scale=0.2,baseline=0.1cm]
        \node at (0,0)  [dot,label= {[label distance=-0.2em]below: \scriptsize  $ 0 $} ] (root) {};
         \node at (0,2)  [dot,label={[label distance=-0.2em]above: \scriptsize  $ 3$}] (right) {};
            \draw[kernel1] (right) to
     node [sloped,below] {\small }     (root);
     \end{tikzpicture} \tikzexternaldisable
     \Xi_4.
\end{equs}
Conversely, $\phi^{-1}$ adds an extra node (which becomes the root) and should be thought of as taking the integral of a symbol in $\tilde \HH$.  The following result makes this precise by giving a bijection between $\mathscr{M}_{[0,T]}$ and $\mathscr{R}^\alpha_{[0,T]}$.

\begin{proposition} \label{prop:rosa}
There is a bijective map $I:\,\mathscr{R}^\alpha_{[0,T]}\rightarrow\mathscr{M}_{[0,T]}$ which maps a branched rough path $\mathbf{X}$ to the unique model $(\Pi,\Gamma)\in\mathscr{M}_{[0,T]}$ with the property that
\begin{align*}
(\Pi_s \II\dot \tau)(t) &= \scal{\mathbf{X}_{s,t},\tau}\quad \text{for all } \tau\in\BB, \; \; s,t \in [0,T],
\end{align*}
where we have made the identifications $\phi (\tau) \equiv \dot \tau \in \tilde \HH \leftrightarrow \tilde \TT$. Furthermore, the model $(\Pi,\Gamma)$ satisfies $\Gamma_{ts} = \Gamma_{\mathbf{X}_{s,t}}$ (where we have made the identification $\mathbf{X}_{s,t} \in \GG \cong \tilde \GG_+$) and the multiplicativity property
\begin{equation}\label{eq:modelMult}
\Pi_t ((\II y_1) \ldots (\II y_n)) = \Pi_t (\II y_1) \ldots \Pi_t (\II y_n), \quad \text{for all } n \in \N, \;\; y_i \in \tilde \TT.
\end{equation}
\end{proposition}

\begin{proof}
Consider $\Xbf \in \mathscr{R}^\alpha_{[0,T]}$. For all $s,t \in [0,T]$ define $\Gamma_{ts} = \Gamma_{\mathbf{X}_{s,t}}$ and $(\Pi_s\II\dot\tau)(t) = \scal{\Xbf_{s,t},\tau}$ for all $\tau \in \BB$. Observe that we may further impose on $(\Pi,\Gamma)$ that properties~\ref{point:M1} and~\ref{point:M2} hold. Furthermore, for every $\tau \notin \II\tilde \TT$, we may define $\Pi_t \tau = (\Pi_t \II\tau)'$, which completely characterises $\Pi$. It remains to verify~\eqref{eq:modelMult}, that property~\ref{point:M3} holds for all $\tau \in \II\tilde \TT$, and that $(\Pi,\Gamma)$ is indeed a model.

For~\eqref{eq:modelMult}, note that from~\eqref{eq:zeroTime} we have
\begin{align*}
\Pi_t(\II \dot \tau_1\ldots \II \dot \tau_n) &= (\Pi_t \II(\II \dot \tau_1\ldots \II \dot\tau_n))' \\
&= (\scal{\Xbf_{t,\cdot}, \phi^{-1}(\II \dot \tau_1\ldots \II \dot\tau_n)})' \\
&= (\scal{\mathbf{\Xbf}_{t,\cdot},[\tau_1\ldots\tau_n]_{\bullet_0}})' \\
&= \scal{\mathbf{\Xbf}_{t,\cdot}, \tau_1\ldots\tau_n} \\
&= \scal{\mathbf{\Xbf}_{t,\cdot},\tau_1}\ldots\scal{\mathbf{\Xbf}_{t,\cdot},\tau_n} = \Pi_t(\II\dot\tau_1)\ldots \Pi_t(\II\dot\tau_n).
\end{align*}

To show property~\ref{point:M3} for $\dot \tau = \II\dot{\bar\tau} \in \II\tilde \TT$, where $\dot{\bar\tau} \in \tilde \TT$, observe that $\phi([\bar\tau]_{\bullet_0}) = \dot \tau$, so that again by~\eqref{eq:zeroTime}
\begin{align*}
\Pi_t \dot \tau &= \Pi_t \II \dot{\bar\tau} \\
&= \scal{\Xbf_{t,\cdot}, \bar\tau} \\
&= (\scal{\Xbf_{t,\cdot}, [\bar\tau]_{\bullet_0}})' \\
&= (\Pi_t \II\phi([\bar \tau]_{\bullet_0}))' \\
&= (\Pi_t \II\dot \tau)'.
\end{align*}

It remains to show that $(\Pi,\Gamma)$ is a model. 
We first verify that $\Pi_s \Gamma_{s,t} = \Pi_t$. Let $\tau \in \BB$, so that $\II(\dot \tau) \in \tilde \TT$. Recall that the Connes-Kreimer coproduct $\Delta_\star : \HH \rightarrow \HH \otimes \HH$ as was introduced in Section~\ref{subsec:PrelimsForest} can be defined recursively by
\[
\Delta_\star [\tau_1\ldots\tau_n]_{\bullet_i} = [\tau_1\ldots\tau_n]_{\bullet_i} \otimes 1 + (\id \otimes [\cdot]_{\bullet_i}) \Delta_\star(\tau_1\ldots \tau_n), \; \; \text{for all } \tau_1,\ldots,\tau_n \in \BB, \; i \in \{0,\ldots, d\}.
\]
With this recursion, one can verify that
\[
\Deltap : \II(\tilde \TT) \rightarrow \II(\tilde \TT) \otimes \tilde \TT_+
\]
agrees with the ``reversed'' Connes-Kreimer coproduct 
\[
\sigma_{1,2}\Delta_\star : \BB \rightarrow \BB \otimes \HH,
\]
where $\sigma_{1,2} : \HH \otimes \BB \rightarrow \BB \otimes \HH$, $\sigma_{1,2} : \tau \otimes \bar \tau \mapsto \bar \tau \otimes \tau$, and where we make the usual identification $\HH\leftrightarrow \tilde \TT_+$ as well as $\phi^{\II} : \BB \rightarrow \II(\tilde \TT)$ via $\phi^{\II} : \tau \mapsto \II(\dot\tau)$ (which is of course just $\II\circ \phi$).
Therefore, treating $\Xbf_{s,t}$ as a character on $\HH \leftrightarrow \tilde \TT_+$, we have for all $\tau \in \BB$
\begin{align}\label{eq:Xtensor}
\begin{split}
(\Pi_t \Gamma_{ts}\II\dot \tau )(u) &= (\Pi_t(\id \otimes \Xbf_{s,t})\Deltap \II\dot\tau)(u) \\
&= \scal{\Xbf_{t,u}, (\phi^{\II})^{-1} (\id \otimes \Xbf_{s,t})\Deltap \II\dot\tau)} \\
&= \scal{\Xbf_{t,u}, (\Xbf_{s,t} \otimes \id)\Delta_\star \tau)} \\
&= \scal{\Xbf_{s,t} \otimes \Xbf_{t,u}, \Delta_\star \tau} \\
&= \scal{\Xbf_{s,t} \tensor \Xbf_{t,u}, \tau} \\
&= \scal{\Xbf_{s,u}, \tau} \\
&= \Pi_s(\II\dot \tau)(u).
\end{split}
\end{align}
Observe now that for $\tau \in \tilde \TT$, we have
\[
\Gamma_{ts}\II\tau = \II\Gamma_{ts}\tau + \scal{\Xbf_{s,t},\II\tau} 1,
\]
where we emphasize the symbol $1 \in \tilde \TT$. Therefore, by the (already established) properties~\ref{point:M1} and~\ref{point:M3}, it follows that for any $\tau \in \tilde \TT$
\[
\Pi_t \Gamma_{ts}\tau = (\Pi_t \II \Gamma_{ts} \tau )' = (\Pi_t(\Gamma_{ts} \II\tau - \scal{\Xbf_{s,t},\II\tau} 1))' = (\Pi_t\Gamma_{ts} \II\tau)' = (\Pi_s \II\tau)' = \Pi_s \tau,
\]
which shows that $\Pi_t\Gamma_{ts} = \Pi_s$.

It remains to verify the analytic bounds on $(\Pi,\Gamma)$. As in Theorem~\ref{thm:transRPs}, denote by $|\tau|$ the number of nodes in $\tau$ and by $|\tau|_0$ the number of nodes with the label $0$. It follows that the degree of $\II\dot\tau$ is given by $|\II\dot\tau| = |\tau|_0(1-\alpha) + |\tau|\alpha$. Since $\Xbf$ satisfies~\eqref{eq:mixedVar}, we have the analytic bound
\[
|(\Pi_s\II\dot\tau)(t)| = |\scal{\Xbf_{s,t}, \tau}| \lesssim |t-s|^{|\II\dot\tau|}.
\]
Since $\Pi_s\tau = (\Pi_s\II\tau)'$ by property~\ref{point:M3}, we see that $\Pi$ satisfies the correct analytic bounds. The exact same argument applies to $\Gamma$ upon using the identification of $\Deltap$ with $\sigma_{1,2}\Delta_\star$ above. Therefore $(\Pi,\Gamma)$ is a model in $\mathscr{M}_{[0,T]}$ as claimed.

Finally, it remains to observe that we may reverse the construction. Indeed, starting with a model $(\Pi,\Gamma)$ in $\mathscr{M}_{[0,T]}$, we may define $\Xbf$ by $\scal{\Xbf_{s,t},\tau} = (\Pi_s\II\dot\tau)(t)$. The facts that $\Xbf$ satisfies~\eqref{eq:zeroTime} follows from property~\ref{point:M3}, while the required analytic bounds for $\Xbf$ to be an $\alpha$-H{\"o}lder branched rough path follow from the analytic bounds associated to $\Pi$. To conclude, it suffices to verify that $\Xbf$ thus defined satisfies $\Gamma_{ts} = \Gamma_{\Xbf_{s,t}}$ and $\Xbf_{s,t}\tensor \Xbf_{t,u} = \Xbf_{s,u}$. To this end, note that by definition of the structure group $G$, there exists $\gamma_{ts} \in \tilde \GG_+ \cong \GG$ such that $\Gamma_{ts} = (\id \otimes \gamma_{ts})\Deltap$. Let $\tilde \Xbf_{s,t} \in \GG$ be the element associated to $\gamma_{ts}$ in the identification $\tilde \GG_+ \cong \GG$, and we aim to show $\tilde \Xbf_{s,t} = \Xbf_{s,t}$. Indeed, from our identification $\HH \leftrightarrow \tilde \TT_+$, it follows that for all $\tau \in \BB$
\[
\scal{\gamma_{ts}, \JJ\dot\tau} = \scal{\tilde \Xbf_{s,t}, \tau}.
\]
On the other hand, we know that for all $\tau \in \BB$
\[
\scal{\Xbf_{s,t},\tau} = (\Pi_s\II\dot\tau)(t) = (\Pi_t \Gamma_{ts}\II\dot \tau) (t) = (\Pi_t (\id\otimes \gamma_{ts})\Deltap\II\dot\tau)(t) = \scal{\gamma_{ts},\JJ\dot\tau},
\]
where for the last equality we have used property~\ref{point:M1} and the fact that
\[
\Deltap \II\dot\tau = 1\otimes \JJ\dot\tau + \sum \II(\dot\tau^{(1)})\otimes \dot\tau^{(2)},
\]
where every term $\II(\dot\tau^{(1)})$ is of positive degree, and so $(\Pi_t\II(\dot\tau^{(1)}))(t) = 0$. This concludes the proof that $\Gamma_{ts} = \Gamma_{\Xbf_{s,t}}$. To verify that $\Xbf_{s,t}\tensor \Xbf_{t,u} = \Xbf_{s,u}$, we can now simply reorder the sequence of equalities~\eqref{eq:Xtensor}.
\end{proof}

Following \cite{BHZ16} we introduce the {\it renormalization map} $\mathrm{M}_{\ell}$ given by\footnote{While we deliberately used the same letter, do not confuse $\mathrm{M}_{\ell}: \tilde \TT \to \tilde \TT$ with $M_v : \HH^* \to \HH^*$.} 
\[
\mathrm{M}_{\ell}:\tilde \TT \to \tilde \TT,\,\,\,\tau \mapsto \left( \ell\otimes \id\right) \Deltam \tau,
\]
for a given character $\ell \in \mathcal{G}_{-} \subset \TT_{-}^{* }$. In our case, we have the fact that $\mathrm{M}_{\ell}$ commutes with $\II$ (cf. end of Remark~\ref{cointeraction})
\begin{equ}
        \mathrm{M}_{\ell} \II = \II \mathrm{M}_{\ell},    \label{equ:MIeIM}
\end{equ}
which is readily verified by hand: $\II$ amounts to adding another edge to the root (thereby creating a new root), whereas $\mathrm{M}_{\ell}$ amounts to extracting (negative) subtrees and maps them to $\R$ (via $\ell$). Clearly, the afore-mentioned edge (of degree $1$) can not possibly be part of any singular subtree, hence the desired commutation.

 This map acts on a model $\mathbf{\Pi}  = ( \Pi, \Gamma )$ and yields the {\it renormalised model} (see~\cite{BHZ16} Theorem~6.15) given by
\[
\Pi_s^{\mathrm{M}_{\ell }} := {\Pi_s } \mathrm{M}_{\ell }, \quad \Gamma_{t,s}^{\mathrm{M}_{\ell }} = \left( \id \otimes \gamma_{t,s}^{\mathrm{M}_{\ell }} \right) \Deltap, \quad 
\gamma_{t,s}^{\mathrm{M}_{\ell }} = \gamma_{t,s} \mathrm{M}_{\ell }.
\] 
Recall from Section~\ref{subsec:DualMapBranched} the map $\delta : \BB \rightarrow \A \otimes \BB$, where $\A$ is the free commutative algebra generated by $\BB$ (thought of as an isomorphic but different space to $\HH$). Recall also the (vector space) isomorphism $\phi: \tau \mapsto \dot \tau$ as detailed in (\ref{phiIso}) with which we identify $ \tilde \HH \cong \BB$.  Let $\pi_- : \tilde \HH \cong \BB \rightarrow \BB_- \cong \scal{\WW_-}$ denote the projection onto terms of negative degree, which we extend multiplicatively to an algebra morphism $\pi_- : \A \rightarrow \HH_-$. We now define the map
\[
\extract^- = (\pi_- \otimes \id) \extract : \tilde\HH \rightarrow \HH_-\otimes \tilde\HH.
\]
For instance
\[
\extract^- \bullet_0 = 1 \otimes \bullet_0,
\]
whereas
\[
\extract \bullet_0 = \bullet_0 \otimes \bullet_0 + 1\otimes\bullet_0.
\]

We are now ready to state the link between translation of branched rough paths and negative renormalization in the following two results.

\begin{lemma} \label{lem:MvsM}
\begin{enumerate}[label=\upshape(\roman*\upshape)]
\item \label{point:MvsM1} For all $\tau \in \BB$ it holds that
\[
\Delta ^{-}\dot{\tau}=\Delta ^{-}\phi \left( \tau \right) =\left( \phi
\otimes \phi \right) \delta ^{-}\left( \tau \right).
\]
\item \label{point:MvsM2}
Let $v$ be an element of $\BB_-^*$ and let $\ell \in \GG_-$  by the associated element in $\GG_- \subset \TT_-^*$, as was detailed in~\eqref{id:BBmGm}. Then
$$
\mathrm{M}_{\ell}\dot{\tau} = \mathrm{M}_{\ell} \phi \left( \tau \right) = \phi \left( M^*_{v}\tau \right) 
$$
\end{enumerate}
\end{lemma}

\begin{proof}
\ref{point:MvsM1} Let us consider $[\tau]_{\bullet_i} \in \mathcal{B}$. We then have the following identities:
\begin{equs}\label{rec_identity}
\Deltam \phi([\tau]_{\bullet_i}) & = \Deltam \tau\Xi_i =
\sum_{C=A \cdot B \subset \tau  } \left( C \otimes (\mathcal{R}_{C} \tau) \Xi_i + A \cdot B \Xi_i \otimes \mathcal{R}_{C} \tau \right). 
\end{equs}
The sum is taken over all the couples $(A,B)$ where $ A $ is a negative subforest of $ \tau $ which does not include the root of $ \tau $ and $ B $ is a subtree of $\tau $ at the root disjoint from $ A $. In the sum in~\eqref{rec_identity}, the first term means that $ \Xi_i $ does not belong to the tree extracted at the root, while for the second term, $ \Xi_i $ belongs to the tree which comes from the product between $ \Xi_i $  and $ B $ giving a subtree of negative degree.
 One can derive the same identity for $ \delta^- $. We first rewrite $ \delta^- $:
 \begin{equs}
 \delta^- \tau = \sum_{A \subset \tau} A \otimes \tilde{\mathcal{R}}_{A} \tau,
 \end{equs}
 where $ A $ is a subforest of $ \tau $ and  $\tilde{\mathcal{R}}_{A} \tau$  means that we contract the trees of $ A $ in $ \tau $ and we leave a $ 0 $ decoration on their roots. Then the equivalent of~\eqref{rec_identity} in that context is given by:
\begin{equs}
\delta^{-} [\tau]_{\bullet_i} & = \sum_{\tilde C = \tilde A \cdot \tilde{B} \subset \tau } \left( \tilde C \otimes [\tilde{\mathcal{R}}_{\tilde{C}} \tau]_{\bullet_i} + \tilde A \cdot [\tilde{B}]_{\bullet_i} \otimes \tilde{\mathcal{R}}_{\tilde A \cdot [\tilde B]_{\bullet_i}} [\tau]_{\bullet_i} \right) \\
\left( \phi \otimes \phi \right) \delta^-  [\tau]_{\bullet_i} & = \sum_{\tilde C = \tilde A \cdot \tilde{B} \subset \tau  } \left( \phi(\tilde{C}) \otimes (\tilde{\mathcal{R}}_{\tilde{C}} \tau) \Xi_i + \phi(\tilde A) \cdot \tilde{B} \Xi_i \otimes \phi\left( \tilde{\mathcal{R}}_{\tilde A \cdot [\tilde B]_{\bullet_i}} [\tau]_{\bullet_i}  \right) \right).
\end{equs} 
Now we have the following identifications:
\begin{equs}
\phi(\tilde C)   \leftrightarrow C, \quad \tilde{B} \Xi_i \leftrightarrow B \Xi_i, \quad   \phi\left( \tilde{\mathcal{R}}_{\tilde A \cdot [\tilde B]_{\bullet_i}} [\tau]_{\bullet_i}  \right) = \tilde{\mathcal{R}}_{\tilde{C}} \tau \leftrightarrow  \mathcal{R}_{C} \tau, \quad (\tilde{\mathcal{R}}_{\tilde{C}} \tau) \Xi_i \leftrightarrow (\mathcal{R}_{C} \tau) \Xi_i,
\end{equs}
which gives the result.

\ref{point:MvsM2}  Recall that $\delta ^{-}\left( \tau \right) $ has an image of the form
``forest $\otimes $ tree'', and that $\ell \circ \phi = v$ (which is a ``dual'' tree and multiplicative over forests). Also note that $M^*_{v}\tau
=\left( v\otimes \id\right) \delta =\left( v\otimes \id\right) \delta ^{-}$
whenever $\ v\in \mathcal{B}_{-}^{\ast }$ (which not true for general $v\in 
\mathcal{B}^{\ast }$), so that  
\begin{eqnarray*}
\mathrm{M}_{\ell}\dot{\tau} &=&\left( \ell \otimes \id\right) \Delta ^{-}\dot{\tau}
\\
&=&\left( \ell \otimes \id\right) \Delta ^{-}\phi \left( \tau \right)  \\
&=&\left( v\otimes \phi \right) \delta ^{-}\left( \tau \right)  \\
&=&\phi \left( \left( v\otimes \id\right) \delta ^{-}\right)  \\
&=&\phi \left( M^*_{v}\tau \right).
\end{eqnarray*}

\end{proof}

\begin{theorem}\label{thm:NegRenorm}
\begin{enumerate}[label=\upshape(\roman*\upshape)]
\item \label{point:Neg1} It holds that the restriction $\Deltam : \tilde \TT \rightarrow \TT_-\otimes \tilde \TT$ coincides with $\extract^- : \tilde\HH \rightarrow \HH_- \otimes \tilde\HH$, where we have made the identifications $\tilde\HH \leftrightarrow \tilde \TT$ and $\HH_- \leftrightarrow \TT_-$ as above.

\item \label{point:Neg2} Let $v$ be an element of $\BB_-^*$ and let $\ell \in \GG_-$  by the associated element in $\GG_- \subset \TT_-^*$, as was detailed in~\eqref{id:BBmGm}.
Then the following diagram commutes
\[
\begin{array}{lll}
\mathbf{X} & \longleftrightarrow  & \mathbf{\Pi } \\ 
\downarrow  &  & \downarrow  \\ 
M_{v}\mathbf{X} & \longleftrightarrow  & \mathbf{\Pi }^{\mathrm{M}_{\ell}}
\end{array}
\]

\item \label{point:Neg3} For $v,v' \in \BB_-$ with associated characters $\ell, \ell' \in \GG_-$, it holds that the character associated to $v+ v'$ is $\ell \circ \ell'$, so that $\left( \mathcal{B}_{-},+\right) \cong \left( \mathcal{G}_{-},\circ \right) $.
\end{enumerate}
\end{theorem}

\begin{remark}\label{remark:algMorphs2}
In view of the final statement of Theorem~\ref{thm:transRPs} part~\ref{point:RDE2}, we see that the commuting diagram in~\ref{point:Neg2} holds upon replacing $M_v$ by any algebra morphism $M : \HH^* \rightarrow \HH^*$ which preserves $\BB^*$, leaves invariant every forest without a label $0$, and satisfies $M \bullet_0 = \bullet_0 + v$.
\end{remark}

\begin{remark}
The final statement~\ref{point:Neg3} effectively says that the renormalization group associated to branched rough path is always abelian, despite the highly non-commutative nature of the Grossman-Larson Hopf algebra $\HH^*$.
\end{remark}

\begin{proof}[Proof of Theorem~\ref{thm:NegRenorm}]
Part~\ref{point:Neg1} is a just of a reformulation of Lemma~\ref{lem:MvsM}~\ref{point:MvsM1}. 
To verify part \ref{point:Neg2}, in view of Proposition \ref{prop:rosa}, we only need to check that for all $\tau \in \BB$
\[
\Pi _{s}^{\mathrm{M}_{l}}\II\dot{\tau}=\left\langle M_{v}\mathbf{X}_{s,\cdot
},\tau \right\rangle =\left\langle \mathbf{X}_{s,\cdot },M_{v}^{\ast }\tau
\right\rangle.
\]
The LHS can be rewritten as, thanks to~\eqref{equ:MIeIM} and Lemma \ref{lem:MvsM}~\ref{point:MvsM2}
\begin{eqnarray*}
\Pi _{s}^{\mathrm{M}_{\ell}}\II\dot{\tau} &=&\Pi _{s}\mathrm{M}_{\ell}\II\dot{\tau} \\
&=&\Pi _{s}\II \mathrm{M}_{\ell}\dot{\tau} \\
&=&\Pi _{s}\II \phi \left( M^*_{v}\tau \right).
\end{eqnarray*}
Applying Proposition \ref{prop:rosa} with $\dot\tau = \phi \left( M_{v}\tau \right) $ then shows
that 
\[
\Pi _{s}\II \phi \left( M_{v}\tau \right) =\left\langle \mathbf{X}_{s,\cdot
},M_{v}^{\ast }\tau \right\rangle 
\]
which is what we wanted to show.

Finally, to show~\ref{point:Neg3}, we note that
\[
\scal{\ell \circ \ell', \tau} = \scal{\ell \otimes \ell', \Deltam \tau} = \scal{\ell,\tau} + \scal{\ell',\tau}, \; \; \text{for all } \tau \in \WW_-,
\]
where the first equality follows by definition and the second from the fact that every element of $\WW_-$ is primitive with respect to the coproduct $\Deltam$. Indeed from the Remark \ref{cured_trees}, we deduce that the coaction $\Deltam$ maps every  $ \tau \in  \WW_- $ into $ \tau \otimes \one + \sum_{(\tau)} \tau' \otimes \tau'' $ such that $ \tau'' $ is a tree of positive degree. However, $\Deltam$ as coproduct on $\TT_-$ (see~\eqref{Dmcp}), will annihilate any term with $\tau''$ of (strictly) positive degree. In particular then, $\Deltam \tau = 1 \otimes \tau + \tau \otimes 1$ for all $\tau \in \WW_-$, that is, any such $\tau$ is primitive.
\end{proof}

\printbibliography

\end{document}